\documentclass[10pt,reqno]{amsart}
\usepackage{authblk}
\usepackage{amsfonts,amssymb,amsmath}

\usepackage{todonotes}

\parskip        \baselineskip
\usepackage{aliascnt}
\usepackage[utf8]{inputenc}
\usepackage[english]{babel}
\usepackage[foot]{amsaddr}
\newtheorem{theorem}{Theorem}[section]
\newtheorem{lemma}[theorem]{Lemma}

\newcommand{\suchthat}{\;\ifnum\currentgrouptype=16 \middle\fi|\;}
\def\R{{\mathbb{R} }}
\renewcommand{\leq}{\leqslant}
\renewcommand{\geq}{\geqslant}

\DeclareMathOperator{\dvr}{div}
\newcommand{\tder}{\partial_t}
\newcommand{\LPND}[1]{L^{#1}_{\dvr}(\Omega)}
\newcommand{\LND}{\LPND{2}}
\newcommand{\WND}{W^{1,2}_{0,\dvr}(\Omega)}

\numberwithin{equation}{section}
\newtheorem{thm}{Theorem}
\numberwithin{thm}{section}

\newaliascnt{lemma}{thm}
\newtheorem{lem}[lemma]{Lemma}
\aliascntresetthe{lemma}

\newaliascnt{proposition}{thm}
\newtheorem{prop}[proposition]{Proposition}
\aliascntresetthe{proposition}

\newaliascnt{corollary}{thm}

\aliascntresetthe{corollary}

\newaliascnt{definition}{thm}
\newtheorem{mydef}[definition]{Definition}
\aliascntresetthe{definition}

\newaliascnt{remark}{thm}
\newtheorem{remark}[remark]{Remark}
\aliascntresetthe{remark}

\newaliascnt{assumption}{thm}
\newtheorem{assumption}[assumption]{Assumption}
\aliascntresetthe{assumption}

\usepackage[colorlinks=true]{hyperref}
\def\eR{{\mathbb{R} }}
\def\eN{\mathbb{N}}

\title[Existence of weak solutions to a diffuse interface model for magnetic fluids]{Existence of weak solutions to a diffuse interface model involving magnetic fluids with unmatched densities}

\author{Martin Kalousek$^\dagger$}
\address{$^\dagger$Institute of Mathematics, Czech Academy of Sciences, \v{Z}itn\'a 25, 11567 Prague, Czech Republic}

\author{Sourav Mitra$^\ddagger$ and Anja Schl\"omerkemper$^\ddagger$}
\address{$^\ddagger$Institute of Mathematics, University of W{\"u}rzburg, Emil-Fischer-Str.\ 40, 97074 W{\"u}rzburg, Germany}

\email{kalousek@math.cas.cz}\email{sourav.mitra@mathematik.uni-wuerzburg.de}  \email{anja.schloemerkemper@mathematik.uni-wuerzburg.de}

\begin{document}

		\maketitle
		\vspace{-1cm}
		\begin{abstract}
		In this article we prove the global existence of weak solutions for a diffuse interface model in a bounded domain (both in 2D and 3D) involving incompressible magnetic fluids with unmatched densities. The model couples the incompressible Navier–Stokes equations,
 gradient flow of the magnetization vector and the
Cahn–Hilliard dynamics describing the partial mixing of two fluids. The density of the mixture depends on an order parameter and the modelling (specifically the density dependence) is inspired from Abels, Garcke and Gr\"{u}n 2011.
		\end{abstract}
\section{Introduction}
 This article is devoted to the mathematical analysis of a system modeling the flow of two viscous incompressible fluids with magnetic properties and unmatched densities undergoing partial mixing in a bounded domain $\Omega\subset \R^d$, $d=2,3.$ with the boundary $\partial\Omega$ of class $C^{2}$. Let $T>0$ and define the space time cylinder as $Q_T = \Omega \times (0,T)$. Further let $\Sigma_{T}$ denote  $\partial\Omega\times(0,T)$.\\
	The mixing of the fluids is described by an order parameter $\phi:Q_T \to \R$, which is the difference of the volume fractions of the fluids involved and undergoes a smooth and rapid transition in an interfacial region between the two fluids. We denote by $v:Q_T\to \R^d$ the mean fluid velocity, by $\rho=\rho(\phi):Q_{T}\to \R$ the mean mass density, $p: Q_T\to \R$ the pressure, $M:Q_{T}\rightarrow\mathbb{R}^{3}$ the magnetization and $\mu:Q_T \to \R$ the chemical potential. The system we consider reads as follows			
	\begin{equation}\label{diffviscoelastic*}
		\left\{ \begin{array}{llll}
			&\displaystyle\partial_{t}(\rho v)+\mbox{div}(\rho v\otimes v)-\mbox{div}\,(2\nu(\phi)\mathbb{D}(v))+\mbox{div}(v\otimes J)+\nabla p&&\\
			&\displaystyle\qquad\qquad\qquad\qquad\qquad=\mu\nabla\phi+\frac{\xi(\phi)}{\alpha^{2}}((|M|^{2}-1)M)\nabla M-\dvr(\xi(\phi)\nabla M)\nabla M&&\mbox{ in } Q_{T},\\
			&\displaystyle\dvr v=0&&\mbox{ in } Q_{T},\\
			&	\displaystyle\partial_{t}M+(v\cdot\nabla)M=\dvr(\xi(\phi)\nabla M)-\frac{\xi(\phi)}{\alpha^{2}}(|M|^{2}-1)M&&\mbox{ in } Q_{T},\\
			&\displaystyle\partial_{t}\phi+(v\cdot\nabla)\phi=\Delta \mu&&\mbox{ in } Q_{T},\\
			&\displaystyle\mu=-\eta\Delta\phi+\Psi'(\phi)+\xi'(\phi)\frac{|\nabla M|^{2}}{2}+\frac{\xi'(\phi)}{4\alpha^{2}}(|M|^{2}-1)^{2}&&\mbox{ in }Q_{T},\\
			&\displaystyle v=0,\ \partial_{n}M=0,\ \partial_{n}\phi=\partial_{n}\mu=0&&\mbox{ on } \Sigma_{T},\\
			&\displaystyle(v,M,\phi)(\cdot,0)=(v_{0},M_{0},\phi_{0})&& \mbox{ in }\Omega.
		\end{array}\right.
	\end{equation} 	 
	where $J$ is a relative flux related to the diffusion of the components and is given by the following expression:
	\begin{equation}\label{extJ}
		\begin{array}{l}
			\displaystyle J=-\frac{\widetilde{\rho}_{2}-\widetilde{\rho}_{1}}{2}\nabla\mu,
		\end{array}
	\end{equation}
	with $\widetilde{\rho}_{i}$ ($i=1,2$) denoting the specific density of the $i-$th fluid. In system \eqref{diffviscoelastic*}, $\alpha$ and $\eta$ are positive constants, $\nu(\phi)$ is the concentration dependent viscosity coefficient which is assumed to be non degenerate and $\mathbb{D}(v)=\frac{1}{2}\left(\nabla v+(\nabla v)^\top\right)$ denotes the symmetric part of the velocity gradient. The factor $\alpha$ penalizes the saturation condition of the length of the magnetization vector $|M|$ from 1 and $\eta$ measures the thickness of the region where the two fluids mix. The function $\xi(\phi)$ denotes the exchange parameter, which reflects the tendency of the magnetization to align in one direction. We assume that $\xi(\cdot)$ is non degenerate, i.e., that it has a positive lower bound and further both $\xi$ and $\xi'$ are bounded from above, cf.\ \eqref{XiAssum}. The homogeneous free energy density of the fluid mixture is denoted by $\Psi(\phi).$ We will consider a class of singular free energies and 
	our consideration (cf.\ \eqref{PsiReg}) will include the homogeneous free energy of the form
		\begin{equation}\label{exPsi}
		\begin{array}{l}
		\displaystyle\Psi(s)=\frac{a}{2}\left((1+s)\ln(1+s)+(1-s)\ln(1-s)\right)-\frac{b}{2}s^{2},
		\end{array}
		\end{equation}
		where $s\in[-1,1]$ and $a,b>0,$ introduced in \cite{Cahn}. The logarithmic terms in \eqref{exPsi} relate to the entropy of the system. We note that the function $\Psi$ in \eqref{exPsi} is convex iff $a\geqslant b$ and $\Psi'$ shows singular behavior at $s=\pm 1$.\\
	The present article is devoted to prove the existence of a global weak solution (i.e.\ without any restriction on time and the size of the initial data) of the model \eqref{diffviscoelastic*}. In Section~\ref{Sec:Notation} we will first introduce some notations corresponding to the functional spaces, the Leray projector which will be essential to deal with our incompressible bi-fluid model and then present our result on the global well posedness of the system \eqref{diffviscoelastic*}. In Section~\ref{Sec:StratAppr} we will comment on the strategies of the proof and related approaches. After a discussion on the physical background of the system in Section~\ref{Sec:Derivation}, we devote Section~\ref{Sec:BibNot} to bibliographical notes. 
	\subsection{Functional framework and main result}\label{Sec:Notation}
	\subsubsection{Functional settings} The Lebesgue and Sobolev spaces are denoted by the notations  $L^{p}(\Omega)$ and $W^{s,p}(\Omega)$ respectively.  The standard notations $C^k(\overline\Omega)$ and $C^{k,\gamma}(\overline\Omega)$ ($\gamma\in(0,1]$) are used to denote respectively the spaces of $k$--times continuously differentiable functions and H\"older continuous functions. The functional spaces with elements having compact support are denoted by using a subscript $c.$ The hooked arrow notations $X\hookrightarrow Y$ ($X\stackrel{C}{\hookrightarrow} Y$) are used to write the continuous (compact) embedding of a Banach space $X$ to a Banach space $Y$. The duality pairing between a Banach space $X$ and its dual $X'$ is written as $\left\langle\cdot,\cdot\right\rangle.$ The functional space $C_w([0,T];X)$ denotes a subspace of $L^\infty(0,T;X)$ containing functions $f$ for which the mapping $t\mapsto\left\langle \phi, f(t)\right\rangle$ is continuous on $[0,T]$ for each $\phi\in X'$. We set 
	\begin{align*}
	\LND=\overline{\{v\in C^\infty_c(\Omega):\ \dvr v=0\text{ in }\Omega\}}^{\|\cdot\|_{L^2}},\quad
	&\WND=\overline{\{v\in C^\infty_c(\Omega):\ \dvr v=0\text{ in }\Omega\}}^{\|\cdot\|_{W^{1,2}}},\\[3.mm]
	W^{2,2}_n(\Omega)=\{u\in W^{2,2}(\Omega): \partial_n u=0\text{ on }\partial\Omega\},\quad
	&V(\Omega)=\WND\cap W^{2,2}(\Omega).
	\end{align*}
From now onward we will use $C$ to denote a generic positive constant which might vary from line to line. For simplicity of notations we will always use $\cdot$ to denote both the scalar product of vectors and tensor products.\\
 Let us now introduce the Leray projector $\mathbb{P}_{\dvr}:L^{2}(\Omega)\rightarrow L^{2}_{\dvr}(\Omega)$ defined as 
 \begin{equation}\label{Leray1}
 \begin{array}{l}
 \displaystyle\mathbb{P}_{\dvr}(f)=f-\nabla p\,\, \mbox{for}\,\, f\in L^{2}(\Omega),\,\, \mbox{where}\,\, p\in W^{1,2}(\Omega)\,\, \mbox{with} \,\,\int_{\Omega} p=0
 \end{array}
 \end{equation}
 solves the weak Neumann problem
 \begin{equation}\label{Leray}
 \begin{array}{l}
 \displaystyle\left(\nabla p,\nabla\varphi\right)_{\Omega}=\left(f,\nabla\varphi\right)\,\,\mbox{for all}\,\,\varphi\in C^{\infty}(\overline{\Omega}).
 \end{array}
 \end{equation}
	\subsubsection{Existence of weak solutions}
	In order to state the precise definition of a weak solution to \eqref{diffviscoelastic*} we summarize the assumptions on the functions $\xi,$ $\nu,$ $\Psi$ and also make the dependence of the mean density $\rho(\phi)$ on $\phi$ precise.\\
	\begin{assumption}\label{Assumption}
		The function $\xi\in C^1(\eR)$ satisfies
		\begin{equation}\label{XiAssum}
			\begin{split}
				&0<c_1\leq\xi\leq c_2\text{ on }\eR, \text{ for some }c_1,c_2>0,\\
				&\xi'\leq c_3\text{ on }\eR, \text{ for some }c_3>0.\
			\end{split}
		\end{equation}
		The viscosity coefficient $\nu\in C^1(\eR)$ satisfies
		\begin{equation}\label{nuAssum}
		0<\nu_1\leq\nu\leq \nu_2\text{ on }\eR, \text{ for some }\nu_1,\nu_2>0.
		\end{equation}
		The homogeneous free energy density $\Psi\in C([-1,1])\cap C^2((-1,1))$ solves
		\begin{equation}\label{PsiReg}
			\begin{split}
				&\lim_{s\rightarrow -1}\Psi'(s)=-\infty,\ \lim\limits_{s\rightarrow 1}\Psi'(s)=\infty,\\ &\Psi''(s)\geqslant -\kappa, \text{ for some }\kappa\in\mathbb{R}.
			\end{split}
		\end{equation}
		The mean mass density $\rho$ and the phase field $\phi$ are related via
		\begin{equation}\label{DefRho}
		\rho(\phi)=\frac{1}{2}(\widetilde{\rho}_{1}+\widetilde{\rho}_{2})+\frac{1}{2}(\widetilde{\rho}_{2}-\widetilde{\rho}_{1})\phi \,\,\mbox{in}\,\, \overline{Q}_{T},
		\end{equation}
		where $\widetilde{\rho}_{i}>0$, $i=1,2$ are specific constant mass densities of the unmixed fluids.
	\end{assumption}
	\begin{remark}\label{Rem:DensityBound}
		The expression for $\rho$ in \eqref{DefRho} implies that 
		\begin{equation}\label{BoundsOnRho}
			0<\min\{\widetilde\rho_1,\widetilde\rho_2\}\leq \rho(\phi)\leq \max\{\widetilde\rho_1,\widetilde\rho_2\}
		\end{equation}
		provided that $|\phi|\leq 1$.
	\end{remark}
	
	\begin{remark}
			Let $\xi_{1},\xi_{2}>0$ be the exchange constants for the magnetic fluids undergoing partial mixing. The function
			$$\xi(\phi)=(1-\mathcal{H}_{\eta}(\phi))\xi_{1}+\mathcal{H}_{\eta}(\phi)\xi_{2},$$
			 where $\eta>0$ corresponds to the thickness of the interface and $\mathcal{H}_{\eta}(x)=\frac{1}{1+e^{-\frac{x}{\eta}}}$ is a regularized approximation of the Heaviside step function, provides an example of a non degenerate function $\xi$ satisfying the assumptions \eqref{XiAssum}, cf.\ \cite{KMS20, Nochetto,yangmao}, in which such a regularized Heaviside function is used in a similar contexts.
	\end{remark}
	Next we define the notion of weak solution to the system \eqref{diffviscoelastic*}.
%	\begin{remark}
	%	As an example of the homogeneous free energy density $\Psi$ one can consider
	%	\begin{equation}\label{exPsi}
	%		\begin{array}{l}
		%		\displaystyle\Psi(s)=\frac{a}{2}\left((1+s)\ln(1+s)+(1-s)\ln(1-s)\right)-\frac{b}{2}c^{2},
		%	\end{array}
		%\end{equation}
	%	where $s\in[-1,1]$ and $a,b>0.$
	%\end{remark}
	
	\begin{mydef}[Definition of weak solutions]\label{DefWS}
		Let $T>0.$ For a given triplet
		\begin{equation}\label{InValAs}
			(v_0,M_0,\phi_0)\in L^2_{\dvr}(\Omega)\times W^{1,2}(\Omega)\times W^{1,2}(\Omega)\text{ with }|\phi_0|\leq 1,
			%\text{ and }\Psi'(\phi_0)\in L^2(\Omega)  
		\end{equation} 
		 we call the quadruple $(v,M,\phi,\mu)$ possessing the regularity
		\begin{equation}\label{functionalspaces}
			\begin{split}
				v&\in C_w([0,T];L^2_{\dvr}(\Omega))\cap L^2(0,T;W^{1,2}_{0,\dvr}(\Omega)),\\
				M&\in C_w([0,T];W^{1,2}(\Omega))\cap C([0,T];L^{2}(\Omega))\cap W^{1,2}(0,T;L^{\frac{3}{2}}(\Omega)),\\
				\phi&\in \displaystyle C_w([0,T];W^{1,2}(\Omega))\cap C([0,T];L^{2}(\Omega))\cap L^2(0,T;W^{2,1}(\Omega))\text{ with }\Psi'(\phi)\displaystyle\in L^1(Q_T),\\
				\mu&\in L^2(0,T;W^{1,2}(\Omega)),
			\end{split}
		\end{equation}
		a weak solution to \eqref{diffviscoelastic*} if it satisfies 
		\begin{equation}\label{WeakForm}
			\begin{split}
				&\int_{\Omega}\rho(t)v(t)\cdot\psi_1(t)-\int_{\Omega}\rho_0v_0\cdot\psi_{1}(0)=\int_0^t\int_{\Omega}\biggl(\rho v\cdot\tder\psi_{1} +\left(\rho v\otimes v+v\otimes J\right)\cdot\nabla \psi_{1}\\ &\qquad-2\nu(\phi)\mathbb{D}v\cdot\mathbb{D}\psi_{1}-\nabla\mu\phi\cdot\psi_{1}
		       +\left(\frac{\xi(\phi)}{\alpha^{2}}\bigl((|M|^{2}-1)M\bigr)\nabla M
				-\dvr\bigl(\xi(\phi)\nabla M\bigr)\nabla M\right)\cdot\psi_{1}\biggr),\\
				&\int_\Omega M(t)\cdot\psi_2(t)-\int_\Omega M_0\cdot \psi_{2}(0)=\int_0^t\int_\Omega\biggl(M\cdot\tder\psi_{2}-(v\cdot\nabla) M\cdot\psi_{2}-\xi(\phi)\nabla M\cdot\nabla\psi_{2}\\
				&\qquad\qquad\qquad\qquad\qquad\qquad\qquad\quad-\frac{1}{\alpha^2}\bigl(\xi(\phi)(|M|^2-1)M\bigr)\cdot\psi_{2}\biggr),\\
				&\int_\Omega \phi(t)\psi_3(t)-\int_\Omega\phi_0\psi_3(0)=\int_0^t\int_\Omega\left(\phi\tder\psi_{3}-(v\cdot\nabla)\phi\psi_{3}-\nabla \mu\cdot\nabla\psi_{3}\right),\\
				&\mu-\xi'(\phi)\frac{|\nabla M|^{2}}{2}-\frac{\xi'(\phi)}{4\alpha^{2}}(|M|^{2}-1)^{2}=-\eta\Delta\phi+\Psi'(\phi)\text{ a.e.\ in }Q_T
			\end{split}
		\end{equation}
		for all $t\in(0,T)$, for all $\psi_{1}\in C^{1}_{c}([0,T);V(\Omega))$, $\psi_{2}\in C^{1}_{c}([0,T);W^{1,2}(\Omega))$ and all \\
		$\psi_{3}\in C^{1}_{c}\left([0,T);W^{1,2}(\Omega)\right)$.
		The initial data are attained in the form 
		\begin{equation}\label{InitDataAtt}
			\lim_{t\to 0_+}\left(\|v(t)-v_0\|_{L^2(\Omega)}+\|M(t)-M_0\|_{W^{1,2}(\Omega)}+\|\phi(t)-\phi_0\|_{W^{1,2}(\Omega)}\right)=0.
		\end{equation}
	\end{mydef}
	Now we present the central result of our article that concerns the global existence of a weak solutions of the model \eqref{diffviscoelastic*}. 
	\begin{theorem}\label{Thm:Main}
		Let $T>0$, $\Omega\subset\eR^d$ be a bounded domain of class $C^{2}$ and let the initial datum $(v_0,M_0,\phi_0)$ satisfy \eqref{InValAs}. Then under Assumption~\ref{Assumption} there exists a quadruple $(v,M,\phi,\mu)$ which solves\eqref{diffviscoelastic*} in the sense of Definition~\ref{DefWS}. Moreover there exists a $p>2$ such that the triplet $(M,\phi,\Psi'(\phi))$ enjoys the following  improved regularity
		\begin{equation}\label{AddReg}
		\begin{split}
		    M&\in L^2(0,T;W^{1,p}(\Omega)),\\
		    \phi&\in L^2(0,T;W^{2,\frac{2p}{p+2}}(\Omega)),\\ \Psi'(\phi)&\in L^2(0,T;L^\frac{2p}{p+2}(\Omega)). 
		\end{split}
		\end{equation}
	  Further the following items hold:\\
	  $(i)$ The obtained weak solution $(v,M,\phi,\mu)$ of \eqref{diffviscoelastic*} satisfies the following energy inequality:
	  \begin{equation}\label{EnIneq*}
	  \begin{split}
	  E_{tot}(v(t),M(t),\phi(t))+&\int_0^t\Biggl(\|\sqrt{2\nu(\phi)}\mathbb{D} v\|^2_{L^2(\Omega)}+\|\nabla \mu\|^2_{L^2(\Omega)}\\
	  &+\left\|\dvr(\xi(\phi)\nabla M)-\frac{\xi(\phi)}{\alpha^2}M(|M|^2-1)\right\|^2_{L^2(\Omega)}\Biggr)\leq E_{tot}(v_0,M_0,\phi_0)
	  \end{split}
	  \end{equation}
	  for all $t\in(0,T),$ where
	  \begin{equation}\label{defEtot}
	  \begin{split}  
	  E_{tot}(v,M,\phi)=&\frac{1}{2}\int_{\Omega}\rho(\phi)|v|^{2}+\frac{1}{2}\int_{\Omega}\xi(\phi)|\nabla M|^{2}+\frac{1}{4\alpha^{2}}\int_{\Omega}\xi(\phi)(|M|^{2}-1)^{2}\\
	  & +\frac{\eta}{2}\int_{\Omega}|\nabla\phi|^{2}+\int_{\Omega}{\Psi}(\phi)
	  \end{split}
	  \end{equation}
	  with $\rho(\phi)$ being defined as in \eqref{DefRho}.\\
	  $(ii)$ The magnetization $M$ attains the homogeneous Neumann boundary condition in a weak sense, i.e.\ for $a.e.$ $t\in(0,T)$ the following holds
	  \begin{equation}\label{trace}
	  \begin{array}{l}
	  \gamma_{n}(\xi(\phi)\nabla M)=0\quad \mbox{in}\quad (W^{\frac{1}{2},2}(\partial\Omega))',
	  \end{array}
	  \end{equation}
	    where $\gamma_{n}$ is the normal trace operator defined on $L^{2}(\Omega)\rightarrow (W^{\frac{1}{2},2}(\partial\Omega))'$ such that the following Stokes formula holds for $a.e.$ $t\in(0,T):
	   $
	   \begin{equation}\label{Stokes}
	   \begin{array}{l}
	   \displaystyle\int_{\Omega}\dvr(\xi(\phi)\nabla M)\cdot \psi_{2}=-\int_{\Omega}\xi(\phi)\nabla M\cdot\nabla \psi_{2}+\langle\gamma_{n}(\xi(\phi)\nabla M),\gamma_{0}\psi_{2}\rangle_{(W^{\frac{1}{2},2}(\partial\Omega))',W^{\frac{1}{2},2}(\partial\Omega)}
	   \end{array}
	   \end{equation} 
	   for all $\psi_{2}\in W^{1,2}(\Omega)$, where $\gamma_{0}: W^{1,2}(\Omega)\rightarrow W^{\frac{1}{2},2}(\partial\Omega)$ is the Dirichlet trace operator.\\
	  $(iii)$ Moreover, if $M_{0}\in W^{1,2}(\Omega)\cap L^{r}(\Omega),$ $r>6,$  then $M\in L^{\infty}(0,T;L^{r}(\Omega))$ and 
	    \begin{equation}\label{afLpGron}
	    \begin{array}{l}
	    \displaystyle\|M(t)\|_{L^{r}(\Omega)}\leqslant \|M_{0}\|_{L^{r}(\Omega)}e^{\delta t}\quad\mbox{for all}\quad t\in[0,T]
	    \end{array}
	    \end{equation}
	    for some positive constant $\delta>0.$ Additionally, if $M_{0}\in L^{\infty}(\Omega),$ then
	    \begin{equation}\label{afLpGronLin}
	    \begin{array}{l}
	    \displaystyle\|M(t)\|_{L^{\infty}(\Omega)}\leqslant \|M_{0}\|_{L^{\infty}(\Omega)}e^{\delta t}\quad\mbox{for all}\quad t\in[0,T].
	    \end{array}
	    \end{equation} 
	\end{theorem}
	We stress on the fact that the improved regularity \eqref{AddReg}, more precisely the unifrom bounds on the suitable approximates of $(M,\phi,\Psi'(\phi))$ in the spaces mentioned in \eqref{AddReg} play a key role in the passage of limit and recovering the weak formulations \eqref{WeakForm} (more precisely  \eqref{WeakForm}$_{4}$).\\ 
	 To the best of our knowledge, \cite{Nochetto,yangmao,KMS20} are the only articles in the literature studying diffuse interface models for magnetic fluids. The article \cite{Nochetto} develops a simplified model describing the behavior of two-phase ferrofluid flows using phase field-techniques and present an energy-stable numerical scheme for the same. The authors of \cite{Nochetto} further analyse the stability and the  convergence of the numerical scheme developed and as a by-product they prove  the existence of weak solutions of their model. In the article \cite{yangmao} the authors propose a diffuse interface model and finite element approximation for two-phase magnetohydrodynamic (MHD) flows with different viscosities and electric conductivities. Their model involves the incompressible Navier-Stokes equations, the Maxwell equations of electromagnetism and the Cahn-Hilliard equations. Unlike \cite{Nochetto} and \cite{yangmao}, in one of our previous articles \cite{KMS20}, we studied a diffuse interface magnetic fluid model where the magnetization vector $M$ is modeled by a gradient flow dynamics. \\
	 As far as we know, our current article is the first mathematical study of a diffuse interface model for a magnetic fluid with unmatched densities. We consider the model \eqref{diffviscoelastic*} where the mean mass density of the mixture depends on the order parameter $\phi$ via the formula \eqref{DefRho}. We show that the mean density $\rho(\phi)$ is always strictly positive and bounded. In order to do so we prove that the order parameter $\phi$ solves the physically reasonable bound $\phi(x,t)\in[-1,1]$ for a.e.\ $(x,t)\in Q_{T}.$ \\
	 To this end it is important to use a singular potential $\Psi(\cdot)$, cf.\ \eqref{PsiReg}, as a homogeneous free energy density of the mixture. 
	 Often in the literature this singular free energy is approximated by a suitable smooth free energy density. For instance, in \cite{KMS20}, we considered a Ginzburg–Landau double-well potential $\tfrac{1}{4\eta}(|\phi|^{2}-1)^{2}$ instead of the singular potential $\Psi(\cdot).$ But using such a polynomial potential one can not ensure that the order parameter $\phi$ stays in the physical reasonable interval $[-1,1]$ due to the lack of a comparison principle for fourth order diffusion equation and hence in particular can not deal with the model \eqref{diffviscoelastic*} comprising of fluids with unmatched densities.\\
	 Unlike the model considered in \cite{yangmao}, in the present case (and also in the one considered in \cite{KMS20}) the magnetization $M$ enters into the Cahn-Hilliard dynamics. Due to the presence of $|\nabla M|^{2}$ in the Cahn-Hilliard part, cf.\ \eqref{diffviscoelastic*}$_{5}$ we can not obtain $L^{2}(0,T;W^{2,2}(\Omega))$ regularity of the order parameter $\phi$ and we only recover $\phi\in L^{2}(0,T;W^{2,q}(\Omega))$ for some $q>1$ by a bootstrap argument using  $L^{2}(0,T;W^{1,p}(\Omega))$ ($p>2$) elliptic regularity for $M.$
	 
   \subsection{Ideas of proof and related discussion}\label{Sec:StratAppr}
   The proof of Theorem~\ref{Thm:Main} is given in  Section~\ref{Thmmain}. It relies on an unconditionally stable time discretization scheme designed in Section~\ref{sec2}. Before going into the analysis of a time discrete problem we first write the singular potential $\Psi$ as a perturbation of a convex function. This helps in reformulating the Cahn-Hilliard equation \eqref{diffviscoelastic*}$_{5}$ as the subdifferential of a convex potential and to use the monotone operator theory and regularity results for Cahn-Hilliard equation developed in \cite{AbelsWilke} and \cite{Abels2007}. The reformulation is done in Section~\ref{reformCahn}. Roughly speaking the content of Section~\ref{sec2} is to consider a sequence $0=t_0<t_1<\ldots<t_k<t_{k+1}<\ldots$, $k\in\eN_0$ of equidistant nodes and next to construct a solution $(v_{k+1},M_{k+1},\phi_{k+1},\mu_{k+1})$ to a stationary problem (cf.\ \eqref{timediscretesystem}) at the point $t_{k+1}$ using $(v_k,M_k,\phi_k)$ which corresponds to the solution at the time $t_k$.\\
    There is no common rule to write a time discretization of a nonlinear PDE model. It is generally done in a way so that the discrete system admits an energy inequality which is in close proximity with the formal energy balance of the original unsteady model. Here we follow a strategy devised in our previous article \cite{KMS20} to suitably discretize the term $\displaystyle\frac{\xi(\phi)}{\alpha^{2}}(|M|^{2}-1)M$ appearing in \eqref{diffviscoelastic*}$_{3}.$ In order to obtain an energy type inequality for \eqref{timediscretesystem} one in particular tests $\displaystyle\frac{M_{k+1}-M_{k}}{h}$ (which is the discretization of the time derivative $\displaystyle\partial_{t}{M}$, with $h=t_{k+1}-t_{k}$ in the discrete magnetization equation \eqref{timediscretesystem}$_{3}$) with an approximation of $(|M|^{2}-1)M.$ Since the map $\displaystyle M\mapsto (|M|^{2}-1)M$ is not monotone one can check that
    $$(M_{k+1}-M_{k})\left(|M_{k+1}|^{2}M_{k+1}-M_{k+1}\right)\ngeqslant \frac{1}{4}(|M_{k+1}|^{2}-1)^{2}-\frac{1}{4}(|M_{k}|^{2}-1)^{2}$$
    and hence the discretization $\displaystyle\frac{\xi(\phi)}{\alpha^{2}}(|M|^{2}-1)M\approx \frac{{\xi(\phi_{k})}}{\alpha^{2}}(|M_{k+1}|^{2}M_{k+1}-M_{k+1})$ does not lead to an unconditionally stable scheme. Following the convex splitting scheme used in our previous article \cite{KMS20} for vector valued functions, which is itself inspired from the convex splitting used in \cite{yangmao} for scaler functions, we use the approximation $\displaystyle\frac{\xi(\phi)}{\alpha^{2}}(|M|^{2}-1)M\approx \frac{{\xi(\phi_{k})}}{\alpha^{2}}(|M_{k+1}|^{2}M_{k+1}-M_{k})$, which along with Lemma~\ref{algebriclem} provides the desired estimate
    $$(M_{k+1}-M_{k})\left(|M_{k+1}|^{2}M_{k+1}-M_{k}\right)\geqslant \frac{1}{4}(|M_{k+1}|^{2}-1)^{2}-\frac{1}{4}(|M_{k}|^{2}-1)^{2}.$$
    We then deal with the time discrete system \eqref{timediscretesystem} by considering it as a  perturbation of a certain nonlinear operator and solving the operator equation by employing a fixed point argument.\\
    After the proof of an existence result for the discrete problem \eqref{timediscretesystem} in Section~\ref{sec2}, we define piecewise constant interpolants in Section~\ref{Thmmain} which approximate $(v,M,\phi,\mu),$ the solution to \eqref{WeakForm}. The weak compactness of the interpolants are obtained from an energy type inequality and the strong compactness properties result by using the classical Aubin-Lions lemma and some suitable interpolation estimates. At this point a crucial observation is the obtainment of the strong convergence of $\{\nabla M^{N}\}_{N}$, where $M^{N}$ approximates $M$. This convergence plays a key role to pass to the limit in the terms approximating $\displaystyle\mbox{div}(\xi(\phi)\nabla M)\nabla M$ (cf.\ \eqref{diffviscoelastic*}$_{1}$) and $\displaystyle\xi'(\phi)\frac{|\nabla M|^{2}}{2}$ (cf.\ \eqref{diffviscoelastic*}$_{5}$), as was presented in our previous article \cite{KMS20}. These arguments can be directly adapted to the current setting. We comment on this at the end of Section~\ref{Ubic}.\\
     To pass to the limit in the approximate of the nonlinear term $\widetilde{\Psi}'_{0}(\phi)$, where $\widetilde{\Psi}_{0}(\cdot)$ corresponds to the convex part of $\Psi(\cdot)$ and is defined on entire $\mathbb{R},$ cf.\ \eqref{deftpsi0}, one first needs to show an apriori estimate of the same in $L^1(Q_{T})$ and next identify the weak limit for a non relabeled subsequence with $\widetilde{\Psi}'_{0}(\phi).$ For the first part the idea roughly is to write the Cahn-Hilliard equation as
    \begin{equation}\label{ellipticphi}
    \begin{alignedat}{2}
   \displaystyle-\eta\Delta\phi^{N}+\widetilde{{\Psi}}'_{0}(\phi^{N})=&\text{lower order terms}+\frac{|\nabla M^{N}|^{2}}{2}&&\text{ in }\Omega\\
   \displaystyle \partial_{n}\phi^{N}=&0 &&\text{ on }\partial\Omega
    \end{alignedat}
    	\end{equation}
   and to use the elliptic structure to obtain a suitable uniform bound for $\phi^{N}$ and $\widetilde{{\Psi}}'_{0}(\phi^{N}).$ In that direction one needs to estimate the right hand side of \eqref{ellipticphi}$_{1}$ in $L^{q}(\Omega)$ with $q>1.$ The boundedness of $E_{tot}$ (defined in \eqref{defEtot}) alone does not provide this information and hence we exploit the dissipative part of the energy and the equation solved by $M^{N}$ to obtain that
   \begin{equation}\label{diveq}
    \dvr\,({{\xi}(\phi^{N})\nabla M^{N}})\in L^{2}(0,T;L^{\frac{3}{2}}(\Omega)).
    \end{equation}
     Next one would expect to recover an improved bound on $\{M^{N}\}$ from \eqref{diveq} but in view of the uniform bound on $\{\phi^{N}\}$ in $L^{\infty}(0,T;W^{1,2}(\Omega))$) one can only use the fact that $\{\xi(\phi^{N})\}$ is bounded in $L^{\infty}(Q_{T})$ and nondegenerate. With this setup we can use \cite{konrad}, which deals with the regularity of weak solutions to elliptic problems with nondegenerate, bounded and measurable coefficients, to obtain a uniform estimate of $M^{N}$ in $L^{2}(0,T;W^{1,p}(\Omega))$ for some $p>2,$ in fact $p$ is slightly greater than two and tends to two as the operator in \eqref{diveq} degenerates. At this moment one can recover a uniform estimate of $|\nabla M^{N}|^{2}$ in $L^{2}(0,T;L^{\frac{2p}{p+2}}(\Omega))$ since $\displaystyle\nabla M^{N}\in L^{\infty}(0,T;L^{2}(\Omega))\cap L^{2}(0,T;L^{p}(\Omega))$. This in turn allows us to use \eqref{ellipticphi} and to obtain a uniform bound of $\widetilde{\Psi}'_{0}(\phi^{N})$ in $L^{2}(0,T;L^{\frac{2p}{p+2}}(\Omega)).$ The details of obtaining these uniform bounds of $M^{N}$ and $\widetilde{\Psi}'_{0}(\phi^{N})$ can be found in Section~\ref{phiNMNuest}. To identify the limit of a subsequence of $\{\widetilde{\Psi}'_{0}(\phi^{N})\}$ with $\widetilde{\Psi}'_{0}(\phi),$ we adapt the ideas related to the estimate of the measure of the set $\{|\phi|=1\}$ developed in \cite{DebDet95} and \cite{FriGra12}.\\
     After the recovery of the weak formulations \eqref{WeakForm} solved by $(v,M,\phi,\mu)$ we prove that the obtained weak solution attains the initial data in a strong sense (cf.\ \eqref{InitDataAtt}).\\
     Items $(ii)$ and $(iii)$ of Theorem~\ref{Thm:Main}, which correspond to the obtainment of boundary condition for $M$ in a weak sense and some further regularity results of $M$ in Lebesgue spaces, are proved in Section~\ref{bndryregM}.
       \subsection{Physical background and comments on the derivation of the model}\label{Sec:Derivation}
       In this section we comment on the physical background of the model \eqref{diffviscoelastic*}. In \cite{KMS20} we already have derived a model for diffuse interface magnetic fluids but with matched densities and a smooth double well potential for the mixing energy. In the present article the density consideration of the mixture is inspired from \cite{AbelsGrun}. To deal with the density dependence it is important to consider a singular logarithmic potential for the mixing energy which is also more physical than a smooth double well potential considered in \cite{KMS20}.\\
     We assume that the mean velocity satisfies the homogeneous Dirichlet boundary condition on $\partial\Omega.$ For the derivation of a modified momentum balance equation solved by a solenoidal mean velocity field $v$ and mean density $\rho(\phi)$ (cf.\ \eqref{DefRho}) we refer to \cite[Section 2]{AbelsGrun}. The obtained momentum balance equation with a general stress tensor $\mathbb{S}$ is of the form:
     \begin{equation}\label{momentumlin}
     \begin{array}{ll}
     \rho\partial_{t}v+\left((\rho v+J)\cdot\nabla\right)v=\dvr\mathbb{S}\quad\mbox{in}\,\, Q_{T},
     \displaystyle 
     \end{array}
     \end{equation} 
      along with the incompressibility $\dvr\,v=0$ where $J$ is the relative diffusion flux defined in \eqref{extJ}. We assume that the stress tensor $\mathbb{S}$ is the sum of the standard viscous Newtonian stress tensor $2\nu(\phi)\mathbb{D}(v)-\pi\mathbb{I}$ (where $\mathbb{I}$ is the identity matrix and $\pi$ being the mean pressure) and an extra contribution from the mixing energy and a simplified micromagnetic energy.\\
      In order to derive the dependence of $\mathbb{S}$ on $\phi,$ $\nabla \phi,$ $M$ and $\nabla M$ we start with the expression of $\mathcal{E}_{mix}$ (the mixing energy) and $\mathcal{E}_{mag}$ (the micromagnetic energy). The mixing energy reads as follows \begin{equation}\label{freeenergy-mix}
      \begin{split}
      \mathcal{E}_{mix}= \frac{\eta}{2}\int_{\Omega}|\nabla\phi|^{2}+\int_{\Omega}\Psi(\phi),
      \end{split}
      \end{equation}
      where $\phi$ is the order parameter, $\eta >0$ denotes the thickness of the interface where the two fluids mix and $\Psi(\cdot)$ is the singular potential satisfying \eqref{PsiReg} (or one can in particular consider the expression \eqref{exPsi}).\\
      In the present article the magnetic energy contribution $\mathcal{E}_{mag}$ is inspired from micromagnetics, for details we refer to \cite{DiFratta-etal2019} and the references therein. We only consider the exchange energy contribution, which reflects the tendency of the magnetization to orient in one direction. We further consider the dependence of the micromagnetic energy on the order parameter $\phi$ which in turn allows us to study a system involving fluids with different magnetic behavior. The simplified micromagnetic energy reads as follows
      \begin{equation*}
      \begin{split}
      &\mathcal{E}_{mag}=\int_{\Omega}\xi(\phi)\frac{|\nabla M|^{2}}{2}+\frac{1}{4\alpha^{2}}\int_{\Omega}\xi(\phi)(|M|^{2}-1)^{2},
      \end{split}
      \end{equation*}
      where $\alpha>0$ is a parameter. The first term in the expression of $\mathcal{E}_{mag}$ is the exchange energy contribution and the second one is a penalization term punishing the derivation of $|M|$ from one (which is a more physical constraint). The consideration of such a penalization term is standard in the literature, cf.\ \cite[Section 1.2]{Kurzke}, \cite{chipotshafrir} or \cite{anjazab}. The magnetic energy in our case is coupled with the order parameter via a regular, bounded and non degenerate function $\xi$ (we refer to \eqref{XiAssum} for the assumptions on $\xi$). In a little different situation, for the modeling and numerical analysis of liquid crystals, one can find specific expressions of degenerate functions $\xi(\cdot)$ in the articles \cite{shenyang} and \cite{yuefeng}. 
      For us it seems important to choose a non degenerate function $\xi$ which in turn plays a crucial role to obtain the strong compactness of $\nabla M.$\\
      Now exactly as in \cite[Section 2.3]{yuefeng} one can use the principle of virtual work to compute that the contribution of $\mathcal{E}_{mix}$ and $\mathcal{E}_{mag}$ to the stress tensor $\mathbb{S}$ is given by
      $$- \frac{\partial \mathcal{E}_{mag}}{\partial \nabla M} \odot \nabla M  -\frac{\partial \mathcal{E}_{mix}}{\partial \nabla \phi} \otimes \nabla \phi
      = -\xi(\phi)(\nabla M\odot\nabla M)-\eta(\nabla\phi\otimes\nabla\phi)$$ 
      where $(\nabla M\odot\nabla M)_{ij}=\sum_{k=1}^{3}(\nabla_{i}M_{k})(\nabla_{j}M_{k})$ and $(\nabla\phi\otimes\nabla\phi)_{ij}=\nabla_{i}\phi\nabla_{j}\phi$.\\
      Hence altogether $\mathbb{S}$ has the following expression
      \begin{equation}\label{mathbbS}
      \begin{array}{l}
      \mathbb{S}=2\nu(\phi)\mathbb{D}(v)-\xi(\phi)(\nabla M\odot\nabla M)-\eta(\nabla\phi\otimes\nabla\phi).
      \end{array}
      \end{equation}
      The momentum balance \eqref{momentumlin} along with \eqref{mathbbS} reads as follows:
      \begin{equation}\label{linmomentum2}
      \begin{array}{ll}
      \rho\partial_{t}v+\left((\rho v+J)\cdot\nabla\right)v-\dvr(2\nu(\phi)\mathbb{D}(v))+\nabla\pi=-\dvr\left(\xi(\phi)(\nabla M\odot\nabla M)+\eta(\nabla\phi\otimes\nabla\phi)\right)\,\,\mbox{in}\,\, Q_{T},
      \end{array}
      \end{equation} 
      with $\dvr\,v=0$ and $v=0$ on $\Sigma_{T}.$\\
      Next we assume physically reasonable boundary conditions $\partial_{n}M=\partial_{n}\phi=\partial_{n}\mu=0$ on $\Sigma_{T}.$ The derivation of the magnetization equation \eqref{diffviscoelastic*}$_{3}$ is based on gradient flow dynamics. The obtainment of the Cahn-Hilliard equations \eqref{diffviscoelastic*}$_{4,5}$ relies on the generalized Fick’s law, i.e., the mass flux be proportional to the
      gradient of the chemical potential (we refer to \cite{Cahn,shenliu} for details). The detailed derivation of \eqref{diffviscoelastic*}$_{3,4,5}$ can be done by following the arguments presented in \cite[p.~8]{KMS20}, with modifications since here we use a singular potential in the mixing energy.\\
      In view of \eqref{diffviscoelastic*}$_{4}$ one at once derives the following mass conservation
     \begin{equation}\label{masconsrv}
     \begin{array}{l}
     \partial_{t}\rho+\dvr\left(\rho v+J\right)=0\quad\mbox{in}\,\,Q_{T}.
     \end{array}
     \end{equation}  
     In the spirit of \cite{AbelsGrun}, we explain here the dynamics behind \eqref{masconsrv}. The equation \eqref{masconsrv} implies that the flux of the density consists of two parts: $\rho v$, describing the transport by the mean velocity, and a relative flux $J$ (cf.\ \eqref{extJ}) related to diffusion of the components. Hence for the unmatched density case, diffusion of the components leads to the diffusion of the mass density.\\
     The modification of the momentum balance (cf.\ \eqref{momentumlin} and \eqref{linmomentum2}) by adding the relative diffusion flux $J$ was proposed in \cite{AbelsGrun} to obtain a local dissipation inequality and global energy estimate for their model. It serves the same purpose for our case and we recover the following formal energy balance for the system \eqref{linmomentum2}--\eqref{diffviscoelastic*}$_{2,3,4,5,6,7}:$
     \begin{equation}\nonumber
     \begin{array}{l}
     \displaystyle \frac{d}{dt}E_{tot}+\left(\|\sqrt{2\nu(\phi)}\mathbb{D} v\|^2_{L^2(\Omega)}+\|\nabla \mu\|^2_{L^2(\Omega)}
     +\left\|\dvr(\xi(\phi)\nabla M)-\frac{\xi(\phi)}{\alpha^2}M(|M|^2-1)\right\|^2_{L^2(\Omega)}\right)=0,
     \end{array}
     \end{equation}
     where $E_{tot}$ is given by \eqref{defEtot}.\\
     Finally in view of the mass balance \eqref{masconsrv}, the incompressibility of $v$ and identities
     \begin{equation}\label{ptwiseidentities}
     \begin{split}
     &\eta\dvr(\nabla\phi\otimes\nabla\phi)=\eta\Delta\phi\nabla\phi+\eta\nabla\left(\frac{|\nabla\phi|^{2}}{2}\right),\\
     & \dvr(\xi(\phi)(\nabla M\odot\nabla M))=\dvr(\xi(\phi)\nabla M)\nabla M+\xi(\phi)\nabla\left(\frac{|\nabla M|^{2}}{2}\right),
     \end{split}
     \end{equation}
     one can rewrite \eqref{linmomentum2} (we refer to \cite[p.~8]{KMS20} for the use of the identities \eqref{ptwiseidentities}) in the form \eqref{diffviscoelastic*}$_{1},$ of course with a modified pressure $p.$ We emphasize that the term $\mu\nabla\phi$ has better compactness properties compared to $\dvr(\nabla\phi\otimes\nabla\phi)$ and this will be exploited in the following analysis.
     \subsection{Bibliographical remarks}\label{Sec:BibNot}  Diffuse interface models without magnetization and comprising of fluids with matched densities date back to the works  \cite{anderson,gurtin,hohenberg,star}. The article \cite{gurtin} provides a unified framework for coupled Navier-Stokes and Cahn-Hilliard equations using the balance law for microforces with a mechanical version of the second law of thermodynamics. The mathematical analysis of such a model first appeared in \cite{star}, where the author deals with strong solutions and stability of stationary solutions (as $t\rightarrow\infty$) in the setting $\Omega=\mathbb{R}^{2}$ assuming a smooth double well potential for the mixing energy. For a detailed review of the subject we also refer to \cite{anderson}.\\
     A detailed study is then performed in the article \cite{boyer} where the author proves the existence of global weak solution of the model both in dimension two and three in a channel with a smooth double well potential for the mixing energy. The article \cite{boyer} further proves that the model under consideration (with non degenerate mobility) admits a strong solution which is global (in time) in dimension two and local (in time) in dimension three. The case of degeneracy in the Cahn-Hilliard equation is also considered in \cite{boyer}. A more complete mathematical description (existence, uniqueness, regularity of solutions and asymptotic behavior) of a similar model with the physically relevant logarithmic potential is discussed in \cite{abelsarma}. The article \cite{abelsarma} proves the existence of strong solutions for an initial velocity $v_{0}$ in interpolation spaces between $W^{1,2}_{0}(\Omega)$ and $W^{2,2}(\Omega)\cap W^{1,2}_{0}(\Omega)$ and satisfying $\mbox{div}\,v_{0}=0$ on $\Omega.$ Further the author shows that any weak solution of the system becomes regular for large times and the order parameter converges to a solution of the stationary Cahn–Hilliard system/a critical point of the mixing energy and the mean velocity tends to zero.  We would also like to quote the article \cite{Giorgini} where the authors study the uniqueness and regularity of weak and strong solutions of the model with the logarithmic potential. The article proves the existence and uniqueness of a global strong solution in dimension two under the assumption that the initial velocity $v_{0}\in W^{1,2}_{0,\dvr}(\Omega)$ and the local in time existence and uniqueness of strong solution in dimension three. Unlike \cite{abelsarma} and \cite{Giorgini} we study the case where the fluids involved have different density and they show magnetic behavior. The coupling of the magnetization with the Cahn-Hilliard dynamics do not allow us to obtain sufficient regularity of the unknowns to prove the uniqueness of weak solutions in dimension two.\\
     In the literature there has been several approaches modeling diffuse interface systems (without magnetization) where the density is not constant. The article \cite{lowengrub} derived a quasi-incompressible diffuse interface model where the velocity field is not divergence free. Analytical results for the model considered in \cite{lowengrub} first appeared in \cite{Abels2009} and \cite{Abels2012}. We also refer to \cite{boyer2} and \cite{ding} for other diffuse interface models involving fluids with unmatched densities. For a slightly non homogeneous case, i.e., with the assumption that the densities of the fluids undergoing partial mixing tend to be equal, \cite{boyer2} proved the global existence of a weak solution and of a unique local strong solution (which is global in two dimension) for their model. The non homogeneous magnetic fluid dynamics we consider in the present article is inspired from  the model introduced in \cite{AbelsGrun}. The global existence of weak solutions for the model introduced in \cite{AbelsGrun} can be found in \cite{AbDeGa113} (the case of non degenerate mobility) and \cite{AbDeGa213} (the case of  degenerate mobility). In a recent article \cite{Giorgini2} the author proves the local in time existence of a strong solution in a bounded domain in dimension two of the model introduced in \cite{AbelsGrun}. The author also shows the global existence in the space periodic set-up. Other variants of the Abels, Garcke and Gr\"{u}n model (\cite{AbelsGrun}) can be found in \cite{GalGrasselli}, \cite{Frigeri1} and \cite{Frigeri2}. The article \cite{GalGrasselli} considers a general diffuse interface model for incompressible two-phase flows with unmatched densities describing the evolution of free interfaces in contact with the solid boundary whereas \cite{Frigeri1} and \cite{Frigeri2} deal with a non-local version of the model derived in \cite{AbelsGrun}. For other analytical results on varying density diffuse interface models we also refer the readers to \cite{AbelsFeireisl}, \cite{FeireislLaurence} (compressible fluids),  \cite{AbelsBreit} (non Newtonian fluids), \cite{Jiang} (coupling between Allen-Cahn and Navier-Stokes) and \cite{Lopes} (non-isothermal diffuse interface model).\\
     Concerning the analysis of single phase magnetic fluids the readers can consult \cite{liubenesova} (magnetization is modeled by LLG dynamics), \cite{anjazab} (gradient flow of magnetization) and for various models involving diffuse interface magnetic fluids with matched densities and smooth double well potential we refer to \cite{Nochetto,yangmao,KMS20}.\\
     A crucial idea in our article is to recover the hidden $L^{2}(0,T;W^{1,p})$ ($p>2$) regularity of $M^{N},$ uniform with respect to $N$ (where $M^{N}$ is the approximate of $M$ defined in Section~\ref{Thmmain}) by using the fact that $\dvr(\xi(\phi^{N})\nabla M^{N})$ is uniformly bounded in $L^{2}(0,T;L^{\frac{3}{2}}(\Omega))$ and $\partial_{n}M^{N}\mid_{\partial\Omega}=0.$ In view of the regularity of $\phi^{N}$ we can only use that $\xi(\phi^{N})$ is bounded, measurable and non degenerate uniformly in $N$. To the best of our knowledge for such an elliptic boundary value problem there are two classes of result. One deals with the H\"{o}lder regularity ($C^{0,\gamma}(\Omega)$) of solutions and the second one proving $W^{1,p}(\Omega)$ ($p>2$) regularity. None of these results imply one another. For the first type of result we refer to \cite{GilTr01} and \cite{Nittka} whereas for the second we quote \cite{konrad} and \cite{Meyers}. In the present article we have used in particular the result proved in \cite{konrad}.

\section{Existence of weak solutions to a time discrete model} 
In this section we prove the existence of weak solution to a time discrete problem corresponding to system~\eqref{diffviscoelastic*}. In that direction we need some regularity results for the Cahn-Hilliard equations proved in \cite{Abels2007}. In order to access the regularity results from \cite{Abels2007} one needs to reformulate the equation \eqref{diffviscoelastic*}$_{5}$ by using the subdifferential of a convex potential. This will be done in the next section and the arguments are inspired from \cite{AbDeGa113}. 
\subsection{Reformulation of the problem using subdifferential of a convex potential}\label{reformCahn}
 First we define a potential $\widetilde{\Psi}$ as follows:
	\begin{equation}\label{deftpsi}
		\widetilde{\Psi}:\mathbb{R}\longrightarrow\mathbb{R},\qquad \widetilde{\Psi}(s)=\begin{cases}
		\Psi(s)&\text{ if }s\in[-1,1],\\
			+\infty&\text{ else}\end{cases} 
	\end{equation}
	where $\Psi$ is introduced in \eqref{PsiReg}. Since $\Psi''(s)\geqslant -\kappa,\,\,\mbox{for some}\,\,\kappa\in\mathbb{R},$ $\widetilde{\Psi}$ is not necessarily convex. In order to use the theory of subdifferentials we introduce a convex function 
	\begin{equation}\label{deftpsi0}
	\widetilde{\Psi}_{0}(r)=\widetilde{\Psi}(r)+\frac{\kappa}{2}r^{2}\,\,\mbox{for}\,\,r\in[-1,1],\,\,\widetilde{\Psi}_{0}\in C([-1,1])\cap C^{2}((-1,1)).
	\end{equation}
	With this new function $\widetilde{\Psi}_{0},$ the equation \eqref{diffviscoelastic*}$_{5}$ can be equivalently written as
	\begin{equation}\label{reformeqmu}
		\begin{array}{l}
			\displaystyle \mu+{\kappa}\phi=\widetilde{\Psi}'_{0}(\phi)-\eta\Delta\phi+\xi'(\phi)\frac{|\nabla M|^{2}}{2}+\frac{\xi'(\phi)}{4\alpha^{2}}(|M|^{2}-1)^{2}\mbox{ in }Q_{T}.
		\end{array}
	\end{equation}
	Inspired by \cite{AbDeGa113}, we define the energy $\widetilde{E}:L^{2}(\Omega)\longrightarrow\mathbb{R}\cup\{+\infty\}$ with the domain
	\begin{equation}\label{defdom}
		\begin{array}{l}
			\mbox{dom}\,\widetilde{E}=\{\phi\in W^{1,2}(\Omega)\suchthat -1\leqslant \phi\leqslant 1\,\,\mbox{a.e.}\}
		\end{array}
	\end{equation}
	as
	\begin{equation}\label{Ephi}
		\widetilde{E}(\phi)=\begin{cases}
			\displaystyle\frac{\eta}{2}\int\limits_{\Omega}|\nabla\phi|^{2}+\int\limits_{\Omega}\widetilde{\Psi}_{0}(\phi)&\text{ for }\phi\in\mbox{dom}\,\widetilde{E},\\
			+\infty&\text{ else.}
		\end{cases}
	\end{equation} 
\begin{prop}\label{ellipticreg}
	\cite[Theorem 3.12.8]{Abels2007} Let $\widetilde{E}:L^{2}(\Omega)\longrightarrow \mathbb{R}\cup\{+\infty\}$ be as defined in \eqref{defdom}--\eqref{Ephi}. Then $\partial \widetilde{E}(\phi)=-\eta\Delta\phi+\widetilde{\Psi}'_{0}(\phi)$ and
	\begin{equation}\label{domainsubg}
	\begin{array}{l}
	\displaystyle\mathcal{D}(\partial\widetilde{E})=\{\phi\in W^{2,2}_{n}(\Omega)\suchthat \widetilde{\Psi}'_{0}(\phi)\in L^{2}(\Omega),\,\,\widetilde{\Psi}''_{0}(\phi)|\nabla\phi|^{2}\in L^{1}(\Omega)\}
	\end{array}
	\end{equation}
	is the domain of definition of the subgradient $\partial\widetilde{E}.$ Moreover, there exists a positive constant $C$ such that 	
	\begin{equation}\label{regularityphi}
	\begin{array}{l}
	\displaystyle\|\phi\|^{2}_{W^{2,2}(\Omega)}+\|\widetilde{\Psi}'_{0}(\phi)\|^{2}_{L^{2}(\Omega)}+\int_{\Omega}\widetilde{\Psi}''_{0}(\phi(\cdot))|\nabla\phi(\cdot)|^{2}\leqslant C\left(\|\partial\widetilde{E}(\phi)\|^{2}_{L^{2}(\Omega)}+\|\phi\|^{2}_{L^{2}(\Omega)}+1\right).
	\end{array}
	\end{equation}
	Further for every $1<p\leqslant 2$, there exists a constant $C_{p}>0$ such that
	\begin{equation}\label{regLp}
	\begin{array}{l}
	\|\phi\|_{W^{2,p}(\Omega)}+\|\widetilde{\Psi}'_{0}(\phi)\|_{L^{p}(\Omega)}\leqslant C_{p}\left(\|\partial\widetilde{E}(\phi)\|_{L^{p}(\Omega)}+\|\phi\|_{L^{2}(\Omega)}+1\right).
	\end{array}
	\end{equation}
\end{prop}
\begin{remark}\label{implicitassurtion}
	It follows from \eqref{PsiReg} and the definitions \eqref{deftpsi} and \eqref{deftpsi0} that if $\phi\in\mathcal{D}(\partial\widetilde{E}),$ then $|\phi|\leqslant 1$ a.e.\ 
\end{remark}
	With the help of the subgradient $\partial\widetilde{E}$ the equation \eqref{reformeqmu} can be written as
	\begin{equation}\label{reformeqmu*}
		\begin{array}{l}
			\displaystyle \mu+{\kappa}\phi=\partial\widetilde{E}(\phi)+\xi'(\phi)\frac{|\nabla M|^{2}}{2}+\frac{\xi'(\phi)}{4\alpha^{2}}(|M|^{2}-1)^{2}\mbox{ in }Q_{T}.
		\end{array}
	\end{equation}
	It is interesting to note that for $\phi\in\mbox{dom}\,\widetilde{E}$ the free energy corresponding to the system \eqref{diffviscoelastic*} is related to $\widetilde{E}(\phi)$ via the following relation  
	\begin{equation}\nonumber
		\begin{array}{ll}
			\displaystyle \mathcal{E}_{free}
			&\displaystyle=\frac{1}{2}\int_{\Omega}\xi(\phi)|\nabla M|^{2}+\frac{1}{4\alpha^{2}}\int_{\Omega}\xi(\phi)(|M|^{2}-1)^{2}+\widetilde{E}(\phi)-\frac{\kappa}{2}\|\phi\|^{2}_{L^{2}(\Omega)}.
		\end{array}
	\end{equation}
	\subsection{Analysis of a time discrete model}\label{sec2}
    To begin with we define a suitable time discretization of the model \eqref{diffviscoelastic*} keeping in mind the reformulation \eqref{reformeqmu} (or \eqref{reformeqmu*}) of \eqref{diffviscoelastic*}$_{5}.$\\
     Let $h>0$ be a constant,
	\begin{equation}\label{assumk}
		v_{k}\in \LND,\,\,M_{k}\in W^{2,2}_{n}(\Omega),\,\,\phi_{k}\in \mathcal{D}(\partial\widetilde{E})
	\end{equation} 
	with $\mathcal{D}(\partial\widetilde{E})$ as in \eqref{domainsubg} and
	\begin{equation}\label{rhok}
	\begin{array}{l}
	 \rho_{k}=\frac{1}{2}(\widetilde{\rho}_{1}+\widetilde{\rho}_{2})+\frac{1}{2}(\widetilde{\rho}_{2}-\widetilde{\rho}_{1})\phi_{k}
	 \end{array}
	\end{equation}
	 be the information at time step $t_{k},$ $k\in\mathbb{N}_{0}.$
	The quadruple $(v_{k+1},M_{k+1},\phi_{k+1},\mu_{k+1}),$ solution at the time step $t_{k+1},$ is determined as a weak solution to the following system
	\begin{equation} \label{timediscretesystem}
		 \begin{array}{lll}
			&\displaystyle \frac{\rho_{k+1}v_{k+1}-\rho_{k}v_{k}}{h}+\dvr(\rho_{k}v_{k+1}\otimes v_{k+1})+\nabla p_{k+1}-\mu_{k+1}\nabla\phi_{k}+\dvr(v_{k+1}\otimes J_{k+1})\\
			&\displaystyle -\dvr(2\nu(\phi_{k})\mathbb{D}v_{k+1})-\frac{{\xi(\phi_{k})}}{\alpha^{2}}(|M_{k+1}|^{2}M_{k+1}-M_{k})\nabla M_{k+1}=
			-\dvr(\xi(\phi_{k})\nabla M_{k+1})\nabla M_{k+1}&\text{ in }\Omega\\[4.mm]
			&\displaystyle\dvr v_{k+1}=0 &\text{ in }\Omega\\[4.mm]
			&\displaystyle\frac{M_{k+1}-M_{k}}{h}+(v_{k+1}\cdot\nabla)M_{k+1}=\dvr(\xi(\phi_{k})\nabla M_{k+1}) -\frac{{\xi(\phi_{k})}}{\alpha^{2}}(|M_{k+1}|^{2}M_{k+1}-M_{k})&\text{ in }\Omega\\[4.mm]
			&\displaystyle\frac{\phi_{k+1}-\phi_{k}}{h}+(v_{k+1}\cdot\nabla)\phi_{k}=\Delta\mu_{k+1}&\text{ in }\Omega\\[4.mm]
			&\displaystyle\mu_{k+1}+\kappa\frac{\phi_{k+1}+\phi_{k}}{2}+\eta\Delta\phi_{k+1}-\widetilde{\Psi}'_{0}(\phi_{k+1})= H_{0}(\phi_{k+1},\phi_{k})\frac{|\nabla M_{k+1}|^{2}}{2}\\
			&\displaystyle\qquad\qquad\qquad\qquad\qquad\qquad\qquad+\frac{1}{4\alpha^{2}}{H_{0}(\phi_{k+1},\phi_{k})}(|M_{k+1}|^{2}-1)^{2}&\text{ in }\Omega\\[4.mm]
			&\displaystyle v_{k+1}=0,\ \partial_{n}M_{k+1}=0,\ \partial_{n}\phi_{k+1}=\partial_{n}\mu_{k+1}=0&\text{ on }\partial\Omega
		\end{array}
	\end{equation}
	where
	\begin{equation}\label{Jrho}
	\begin{array}{l} \displaystyle J=J_{k+1}=-\frac{\widetilde{\rho}_{2}-\widetilde{\rho}_{1}}{2}\nabla\mu_{k+1},\qquad \rho_{k+1}=\frac{1}{2}(\widetilde{\rho}_{1}+\widetilde{\rho}_{2})+\frac{1}{2}(\widetilde{\rho}_{2}-\widetilde{\rho}_{1})\phi_{k+1}
	\end{array}
	\end{equation}
	 and $H_{0}:\mathbb{R}\times\mathbb{R}\to \mathbb{R}$ is defined as
	\begin{equation}\label{H0}
		\ H_{0}(a,b)=\begin{cases}
			\frac{\xi(a)-\xi(b)}{a-b} &\mbox{if}\qquad a\neq b,\\[4.mm]
			\xi'(b) & \mbox{if}\qquad a= b.
		\end{cases}
	\end{equation}
	
	\enlargethispage{-\baselineskip}
	Now let us introduce the notion of weak solution to the time discrete system \eqref{timediscretesystem}. In the following definition of weak solution, the term $\displaystyle\int_{\Omega}(\mathrm{div}(v_{k+1}\otimes J_{k+1}))\widetilde{\psi}_{1}$ (originated from the fifth term of \eqref{timediscretesystem}$_{1}$) is replaced using the identity
	\begin{equation}\label{JKIdent}
	\int_{\Omega}(\mathrm{div}(v_{k+1}\otimes J_{k+1}))\widetilde{\psi}_{1}=\int_{\Omega}\left(\mathrm{div}J_{k+1} -\frac{\rho_{k+1}-\rho_{k}}{h}-v_{k+1}\cdot\nabla\rho_{k}\right)\frac{v_{k+1}}{2}\cdot \widetilde{\psi}_{1}+\int_{\Omega}(J_{k+1}\cdot\nabla)v_{k+1}\cdot\widetilde{\psi}_{1}
	\end{equation}
	(we refer to  \eqref{identity1}) and the derivation can be found in \cite[Remark 4.1 (i)]{AbDeGa113}. This reformulation helps mainly to obtain later the discrete energy estimate \eqref{discreteestimate}.
	\begin{mydef}\label{defweaksolta}[Weak solution to the problem \eqref{timediscretesystem}]
		Let \eqref{assumk}--\eqref{rhok} hold. The quadruple 
		\begin{equation}\label{regularityk1}
			\begin{array}{l}
				(v_{k+1},M_{k+1},\phi_{k+1},\mu_{k+1})\in \WND\times W^{2,2}_{n}(\Omega)\times \mathcal{D}(\partial\widetilde E)\times W^{2,2}_{n}(\Omega),
			\end{array}
		\end{equation}
		is a weak solution to system~\eqref{timediscretesystem} if the following identities are true
		\begin{equation}\label{identity1}
			\begin{split}
				&\int_{\Omega} \displaystyle\frac{\rho_{k+1}v_{k+1}-\rho_kv_{k}}{h}\cdot\widetilde\psi_{1}+\int_{\Omega}\dvr(\rho_{k}v_{k+1}\otimes v_{k+1})\cdot\widetilde{\psi}_{1}\\
				& \displaystyle+\int_{\Omega}\left(\mathrm{div}J_{k+1} -\frac{\rho_{k+1}-\rho_{k}}{h}-v_{k+1}\cdot\nabla\rho_{k}\right)\frac{v_{k+1}}{2}\cdot \widetilde{\psi}_{1}+\int_{\Omega}(J_{k+1}\cdot\nabla)v_{k+1}\cdot\widetilde{\psi}_{1}\\
				&\displaystyle-\int_{\Omega}\left(\frac{{\xi(\phi_{k})}}{\alpha^{2}}(|M_{k+1}|^{2}M_{k+1}-M_{k})\nabla M_{k+1}\right)\cdot\widetilde\psi_{1}+\int_{\Omega}\left(\mathrm{div}(\xi(\phi_{k})\nabla M_{k+1})\nabla M_{k+1}\right)\cdot\widetilde\psi_{1}\\
				&\displaystyle=-2\int_{\Omega}\nu(\phi_{k})\mathbb{D} v_{k+1}\cdot\mathbb{D}\widetilde\psi_{1}-\int_{\Omega}\nabla\mu_{k+1}\phi_{k}\cdot\widetilde\psi_{1}
			\end{split}
		\end{equation}
		for all $\widetilde\psi_{1}\in \WND,$\\
		\begin{equation}\label{identity2}
			\begin{split}
				&\int_{\Omega}\frac{M_{k+1}-M_{k}}{h}\cdot\widetilde\psi_{2}+\int_{\Omega}(v_{k+1}\cdot\nabla)M_{k+1}\cdot\widetilde\psi_{2}\\
				& =\int_{\Omega}\left(\dvr\left(\xi(\phi_{k})\nabla M_{k+1}\right)-\frac{\xi(\phi_{k})}{\alpha^{2}}(|M_{k+1}|^{2}M_{k+1}-M_{k})\right)\cdot\widetilde\psi_{2}
			\end{split}
		\end{equation}
		for all $\widetilde\psi_{2}\in L^2(\Omega),$\\
		\begin{equation}\label{identity3}
			\begin{split}
				&\frac{\phi_{k+1}-\phi_{k}}{h}+(v_{k+1}\cdot\nabla)\phi_{k}=\Delta\mu_{k+1}
			\end{split}
		\end{equation}
		a.e.\  in $\Omega$ and 
		\begin{equation}\label{identity4}
			\begin{split}
				&\mu_{k+1}+\kappa\frac{\phi_{k+1}+\phi_{k}}{2}-H_{0}(\phi_{k+1},\phi_{k})\frac{|\nabla M_{k+1}|^{2}}{2}-\frac{H_{0}(\phi_{k+1},\phi_{k})}{4\alpha^{2}}(|M_{k+1}|^{2}-1)^{2}\\
				&=-\eta\Delta\phi_{k+1}+\widetilde{\Psi}'_{0}(\phi_{k+1})
			\end{split}
		\end{equation}
		a.e.\  in $\Omega,$ where $J_{k+1}$ and $\rho_{k+1}$ are as defined in \eqref{Jrho} and $H_{0}$ is defines in \eqref{H0}. 
	\end{mydef}
	%%%%%%%%%%%%%%%%%%%%%%%%%%%%%%%%%%%%%%%%%%%%%%%%%%%%%%%%%%%%%%%%%%%%%%%%%%%%%%%%%%%%%%%%%%%%%%%%%%%%%%%%%%%%%%%%%%%%%%%%%%%%%%%%%%%%%%%%%%%%%%%%%%%%%%%%%%%%%%%%%%%%%%%%%%%%%%%%%%%%%%%%%%%%%%%%%%%%%%%%%%%%%%%%%%%%%%%%%%%%%%%%%%%%%%%%%%%%%%%%%%%%%%%%%%%%%%%%%%%%%%%%%%%%%%
In the next lemma we first prove an estimate of the $L^{1}$ norm of $\widetilde{\Psi}^{'}_{0}(\phi_{k+1})$ assuming the existence of a triplet $(M_{k+1},\phi_{k+1},\mu_{k+1})$ solving \eqref{identity4}. Then using the obtained estimate of $\|\widetilde{\Psi}^{'}_{0}(\phi_{k+1})\|_{L^{1}(\Omega)}$ we further prove an estimate of $\displaystyle\left|\int_{\Omega}\mu_{k+1}\right|.$ This estimate will be specifically used in Section~\ref{Thmmain} to show \eqref{MuCInterpBound}. We will also use a similar estimate while showing \eqref{estimatesofar}. 
	\begin{lem}\label{avestimates}
		Let $\phi=\phi_{k+1}\in\mathcal{D}(\partial \widetilde{E})$ and $\mu=\mu_{k+1}\in W^{1,2}(\Omega)$ solve \eqref{identity4} with $\phi_{k}\in W^{2,2}_{n}(\Omega),$ $|\phi_{k}|\leqslant 1$ in $\Omega$ such that
		$$\frac{1}{|\Omega|}\int_{\Omega}\phi_{k}=\frac{1}{|\Omega|}\int_{\Omega}\phi\in(-1,1),$$
		and $M=M_{k+1}\in W^{2,2}_{n}(\Omega).$
		Then there exists a constant $\displaystyle C=C\left(\int_{\Omega}\phi_{k}\right)>0,$ such that
		\begin{equation}\label{boundsubdiff}
			\begin{array}{ll}
				&\displaystyle\|\widetilde{\Psi}'_{0}(\phi)\|_{L^{1}(\Omega)}+\left|\int_{\Omega}\mu\right|\leqslant C\left(\|\nabla\mu\|_{L^{2}(\Omega)}+\|\nabla\phi\|^{2}_{L^{2}(\Omega)}
				+\|M\|^{2}_{W^{1,2}(\Omega)}+\|M\|^{4}_{W^{1,2}(\Omega)}+1\right).\\
				%&\displaystyle(ii)\,\,\|\partial\widetilde{E}(\phi)\|_{L^{2}(\Omega)}\leqslant C\left(\|\mu\|_{L^{2}(\Omega)}+\|M\|^{2}_{W^{2,2}(\Omega)}+\|M\|^{4}_{W^{2,2}(\Omega)}+1\right).
			\end{array}
		\end{equation}
		%and for $p=2q,$ $q>1,$
		%\begin{equation}\label{Lpbound}
			%\begin{array}{ll}
			%	&\displaystyle(iii)\,\, \|\phi\|_{W^{2,q}(\Omega)}+\|\widetilde{\Psi}'_{0}(\phi)\|_{L^{q}(\Omega)}\leqslant C\left(\|\mu\|_{L^{q}(\Omega)}+\|M\|^{2}_{W^{1,p}(\Omega)}+\|M\|^{4}_{L^{2p}(\Omega)}+1\right).
			%\end{array}
		%\end{equation}
	\end{lem}
	\begin{proof}
		Without the magnetization vector $M$ a similar result was shown in \cite[Lemma 4.2]{AbDeGa113}. We will suitably adapt the line of arguments used in proving \cite[Lemma 4.2]{AbDeGa113} in our case.\\
		We test \eqref{identity4} by $(\phi-\overline{\phi}),$ where $\displaystyle\overline{\phi}=\frac{1}{|\Omega|}\int_{\Omega}\phi$ and obtain
		\begin{equation}\label{obtainidntity4}
			\begin{array}{ll}
				&\displaystyle\int_{\Omega}\mu(\phi-\overline{\phi})+\kappa\int_{\Omega}\frac{\phi+\phi_{k}}{2}(\phi-\overline{\phi}) -\int_{\Omega} H_{0}(\phi,\phi_{k})\frac{|\nabla M|^{2}}{2}(\phi-\overline{\phi})\\
				&\displaystyle-\int_{\Omega}\frac{H_{0}(\phi,\phi_{k})}{4\alpha^{2}}(|M|^{2}-1)^{2}(\phi-\overline{\phi})
				\displaystyle=\eta\int_{\Omega}\nabla\phi\cdot\nabla(\phi-\overline{\phi})+\int_{\Omega}\widetilde{\Psi}'_{0}(\phi)(\phi-\overline{\phi}).
			\end{array}
		\end{equation}
		One observes that $\displaystyle\int_{\Omega}\mu(\phi-\overline{\phi})=\int_{\Omega}(\mu-\overline{\mu})\phi,$ where $\displaystyle\overline{\mu}=\frac{1}{|\Omega|}\int_{\Omega}\mu.$ 
		Since  $\overline{\phi}\in(-1+\epsilon,1-\epsilon)$ for some $\epsilon>0$ (note that $\epsilon$ is independent of $\phi$) and $\lim\limits_{\phi\rightarrow \pm 1}\widetilde{\Psi}'_{0}(\phi)=\pm\infty,$ one has the inequality 
		\begin{equation}\label{ptwiswPsi}
		\begin{array}{l}
		\widetilde{\Psi}_{0}^{'}(\phi)(\phi-\overline{\phi})\geqslant C_{1}|\widetilde{\Psi}'_{0}(\phi)|-C_{2},
		\end{array}
		\end{equation}
		for constants $C_{1}>0$ and $C_{2}.$ The inequality \eqref{ptwiswPsi} can be proved by dividing $[-1,1]$ into three intervals $[-1,-1+\frac{\epsilon}{2}],$ $[-1+\frac{\epsilon}{2},1-\frac{\epsilon}{2}],$ $[1-\frac{\epsilon}{2},1],$ arguing by the blow up behavior of $\widetilde{\Psi}'_{0}$ at the endpoints $\{-1,1\}$ and the fact that $\widetilde{\Psi}'_{0}\in C([-1+\frac{\epsilon}{2},1-\frac{\epsilon}{2}]).$
		Hence integrating \eqref{ptwiswPsi} in $\Omega,$ we have the estimate
		\begin{equation}\label{Psiest}
			\begin{array}{l}
				\displaystyle\int_{\Omega} \widetilde{\Psi}'_{0}(\phi)(\phi-\overline{\phi})\geqslant C_{1}\int_{\Omega}|\Psi'_{0}(\phi)|-C_{3}
			\end{array}
		\end{equation}
		for constants $C_{1}>0$ and $C_{3}.$\\
		Further using \eqref{XiAssum}$_{2}$ and the definition \eqref{H0} of $H_{0}(\cdot,\cdot)$ one has $|H_{0}(\phi_{k+1},\phi_{k})|\leqslant c_{3}.$ Hence in view of  \eqref{obtainidntity4}, \eqref{Psiest} and the fact that $|\phi|,|\phi_{k}|\leqslant 1,$ we deduce
		\begin{equation}\label{estimatePsi}
			\begin{array}{ll}
				\displaystyle\int_{\Omega}|\widetilde{\Psi}'_{0}|&\displaystyle\leqslant C\left(\|\mu-\overline{\mu}\|_{L^{2}(\Omega)}+\|\nabla M\|_{L^{2}(\Omega)}^{2}+\|M\|^{4}_{L^{4}(\Omega)}+\|\nabla\phi\|_{L^{2}(\Omega)}^{2}+1\right)\\
				&\displaystyle\leqslant C\left(\|\nabla\mu\|_{L^{2}(\Omega)}+\|M\|^{2}_{W^{1,2}(\Omega)}+\|M\|^{4}_{W^{1,2}(\Omega)}+\|\nabla\phi\|_{L^{2}(\Omega)}^{2}+1\right),
			\end{array}
		\end{equation}
		where we have used Poincar\'{e}'s inequality to obtain the final step.\\
		Now we want to use the inequality \eqref{estimatePsi} to obtain an estimate of $\displaystyle\left|\int_{\Omega}\mu\right|.$ In that direction we integrate \eqref{identity4} to obtain
		\begin{equation}\label{estimatemu}
			\begin{array}{ll}
				\displaystyle\left|\int\limits_{\Omega} \mu\right|&\leqslant\displaystyle C\left(\int_{\Omega}|\widetilde{\Psi}'_{0}|+\|M\|^{2}_{W^{1,2}(\Omega)}+\|M\|^{4}_{W^{1,2}(\Omega)}+\|\nabla\phi\|_{L^{2}(\Omega)}^{2}+1\right).
			\end{array}
		\end{equation}
		Next using \eqref{estimatePsi} in \eqref{estimatemu} we furnish
		\begin{equation}\label{estimatemuav}
			\begin{array}{l}
				\displaystyle\left|\int\limits_{\Omega} \mu\right|\leqslant C\left(\|\nabla\mu\|_{L^{2}(\Omega)}+\|\nabla\phi\|^{2}_{L^{2}(\Omega)}\displaystyle+\|M\|^{2}_{W^{1,2}(\Omega)}+\|M\|^{4}_{W^{1,2}(\Omega)}+1\right).
			\end{array}
		\end{equation}
		Combining \eqref{estimatePsi} and \eqref{estimatemuav} we conclude the proof of Lemma~\ref{avestimates}.
		%Now the bound \eqref{boundsubdiff}$(ii)$ on the sub differential $\partial\widetilde{E}(\phi)=-\eta\Delta\phi+\widetilde{\Psi}'_{0}$ follows directly from the expression:
		%$$\partial\widetilde{E}(\phi)=\mu+{\kappa}\phi-\xi'(\phi)\frac{|\nabla M|^{2}}{2}-\frac{\xi'(\phi)}{4\alpha^{2}}(|M|^{2}-1)^{2}\mbox{ in }Q_{T}.$$
		%Finally using \eqref{regularityphi} and the bound \eqref{boundsubdiff}$(ii)$ we conclude the bound of $\|\widetilde{\Psi}'_{0}(\phi)\|_{L^{2}(\Omega)}$ as in \eqref{boundsubdiff}$(i).$
		%\\In order to prove \eqref{Lpbound} we use once again the expression \eqref{detilde} and the following inequalities:
		%\begin{equation}\label{LpbndM}
			%\begin{array}{l}
			%	\||\nabla M|^{2}\|_{L^{q}(\Omega)}\leqslant C\|\nabla M\|^{2}_{L^{p}(\Omega)}\leqslant C\|M\|^{2}_{W^{1,p}(\Omega)}.\\[2mm]
			%	\mbox{and}\,\,\|(|M|^{2}-1)^{2}\|_{L^{q}(\Omega)}\leqslant C\|M\|^{4}_{L^{2p}(\Omega)},
			%\end{array}
		%\end{equation}
		%to have $$\|\partial\widetilde{E}(\phi)\|_{L^{q}(\Omega)}\leqslant C\left(\|\mu\|_{L^{q}(\Omega)}+\|M\|^{2}_{W^{1,p}(\Omega)}+\|M\|^{4}_{L^{2p}(\Omega)}+1\right).$$
		%Next using \eqref{regLp} one has the desired estimate \eqref{Lpbound}.
	\end{proof}

	Next we recall the following result from \cite[Lemma 4.1]{KMS20}, which will be used to obtain \eqref{discreteestimate} (a discrete analogue of energy dissipation) in Theorem~\ref{existenceweaksoldiscrete}. 
	\begin{lem}\cite[Lemma 4.1]{KMS20} \label{algebriclem}
		Let $A,B\in \mathbb{R}^{3}.$ The following relation is true
		\begin{align}
			& \frac{1}{4}\bigl(|A|^{2}-1\bigr)^{2}-\frac{1}{4}\bigl(|B|^{2}-1\bigr)^{2}+\frac{1}{4}\bigl(|A|^{2}-|B|^{2}\bigr)^{2}+\frac{1}{2}|A\cdot(A-B)|^{2} +\frac{1}{2}|A-B|^{2}\nonumber\\
			&\leqslant (A-B)\cdot \bigl(|A|^{2}A-B\bigr)\label{AlgIneq1}.
		\end{align}
	\end{lem}
	Now we state and prove the central result of this section which corresponds to the existence of weak solution to the time discrete system \eqref{timediscretesystem}.
	\begin{thm}\label{existenceweaksoldiscrete}[Existence of weak solution to the problem \eqref{timediscretesystem}]
		Let Assumption \ref{Assumption}, \eqref{assumk} and \eqref{rhok} hold. Then there exists a quadruple $(v_{k+1},M_{k+1},\phi_{k+1},\mu_{k+1})$ which satisfies \eqref{regularityk1} and solves the identities \eqref{identity1}--\eqref{identity4}. Moreover, the following discrete version of the energy estimate holds
		\begin{equation}\label{discreteestimate}
			\begin{split}
				&  E_{tot}(v_{k+1},M_{k+1},\phi_{k+1})+2h\int_{\Omega}\nu(\phi_{k})|\mathbb{D}v_{k+1}|^{2}
				+h\int_{\Omega}|\nabla\mu_{k+1}|^{2}\\ &+h\int_{\Omega}\left|\mathrm{div}(\xi(\phi_{k})\nabla M_{k+1})-\frac{\xi(\phi_{k})}{\alpha^{2}}(|M_{k+1}|^{2}M_{k+1}-M_{k})\right|^{2}
				\leqslant E_{tot}(v_{k},M_{k},\phi_{k}),
			\end{split}
		\end{equation} 
		where $E_{tot}(v,M,\phi)$ is as defined in \eqref{defEtot}.
	\end{thm}
	\begin{proof}
		%It is natural to prove first the existence of a weak
		%solution and then to obtain the energy dissipation inequality \eqref{discreteestimate} as a special case of an inequality
		%derived in the existence proof. But we choose to present them in a reverse way. The reason is that our
		%approach consists in the application of Lemma~\ref{algebriclem} and identity \eqref{algebricidnty} in both steps but its presentation requires less space and is easier to understand in the proof of the energy inequality.\\
		For simplicity in notations we will omit the subscript $k+1$ and we use the notation $(v,M,\phi,\mu,J,\rho)=(v_{k+1},M_{k+1},\phi_{k+1},\mu_{k+1},J_{k+1},\rho_{k+1})$ in the rest of the proof. We will perform the proof in two steps (cf.\ Section \ref{inqsatsweak} and \ref{existencetimediscrete}).\\
		\subsection{Any weak solution $(v, M, \phi, \mu)$ of \eqref{timediscretesystem} in the sense of Definition~\ref{defweaksolta} satisfies \eqref{discreteestimate}--\eqref{defEtot}}\label{inqsatsweak}
	    In the following computations we will need some identities in the spirit of \cite{AbDeGa113}. We gather those identities in the following and refer to the proof of \cite[Lemma 4.3]{AbDeGa113} for details.
		\begin{equation}\label{someidentities}
			\begin{array}{ll}
				&(i)\,\,\displaystyle\int_{\Omega}\left((\mbox{div}\,J)\frac{v}{2}+\left(J\cdot\nabla\right)v\right)\cdot v=\int_{\Omega}\mbox{div}\,\left(J\frac{|v|^{2}}{2}\right)=0,\\
				&(ii)\,\,\displaystyle \int_{\Omega} \left(\mbox{div}(\rho_{k}v\otimes v)-(\nabla\rho_{k}\cdot v)\frac{v}{2}\right)\cdot v=0,\\
				&(iii)\,\,\displaystyle (\rho v-\rho_{k}v_{k})\cdot v=\left(\rho\frac{|v|^{2}}{2}-\rho_{k}\frac{|v_{k}|^{2}}{2}\right)+(\rho-\rho_{k})\frac{|v|^{2}}{2}+\rho_{k}\frac{|v-v_{k}|^{2}}{2}.
			\end{array}
		\end{equation}
	 First we consider the test function $\widetilde{\psi}_{1}=v$ in \eqref{identity1} and use the identities \eqref{someidentities} to render 
		\begin{equation}\label{identity1*}
			\begin{split}
				&\int_{\Omega} \frac{\rho|v|^{2}-\rho_{k}|v_{k}|^{2}}{2h}
				+\int_{\Omega}\rho_{k}\frac{|v-v_{k}|^{2}}{2h}-\int_{\Omega}\left(\frac{{\xi(\phi_{k})}}{\alpha^{2}}(|M|^{2}M-M_{k})\nabla M\right)\cdot v\\
				&+\int_{\Omega}\left(\mathrm{div}(\xi(\phi_{k})\nabla M)\nabla M\right)\cdot v
				=-2\int_{\Omega}\nu(\phi_{k})|\mathbb{D} v|^{2}-\int_{\Omega}(v\cdot\nabla)\mu\phi_{k}.
			\end{split}
		\end{equation}
		Next choosing $\widetilde\psi_2=-\mbox{div}(\xi(\phi_{k})\nabla M)+\frac{\xi(\phi_{k})}{\alpha^{2}}(|M|^{2}M-M_{k})$ in \eqref{identity2} we infer
		\begin{equation}\label{identity2*}
			\begin{array}{ll}
				&\displaystyle\int_{\Omega}\frac{(M-M_{k})}{h}\cdot \Bigl(-\mbox{div}(\xi(\phi_{k})\nabla M)
				+\frac{\xi(\phi_{k})}{\alpha^{2}}(|M|^{2}M-M_{k})\Bigr)\\
				&\displaystyle +\int_{\Omega}(v\cdot\nabla)M\cdot\Bigl(-\mbox{div}(\xi(\phi_{k})\nabla M)
				+\frac{\xi(\phi_{k})}{\alpha^{2}}(|M|^{2}M-M_{k})\Bigr)\\
				&\displaystyle+\int_{\Omega}\Bigl|\mbox{div}(\xi(\phi_{k})\nabla M)-\frac{\xi(\phi_{k})}{\alpha^{2}}(|M|^{2}M-M_{k})\Bigr|^{2}=0.
			\end{array}
		\end{equation}
		Multiplying \eqref{identity3} by $\mu$ and integrating over $\Omega$ one has
		\begin{equation}\label{identity3*}
			\begin{array}{l}
				\displaystyle\int_{\Omega}\frac{\phi-\phi_{k}}{h}\mu-\int_{\Omega}(v\cdot\nabla)\mu\phi_{k}=-\int_{\Omega}|\nabla\mu|^{2}.
			\end{array}
		\end{equation}
		Finally, multiplying \eqref{identity4} by $\displaystyle-\frac{\phi-\phi_{k}}{h}$ and integrating in $\Omega$ we have 
		\begin{equation}\label{identity4*}
			\begin{array}{ll}
				&\displaystyle-\int_{\Omega}\mu\frac{\phi-\phi_{k}}{h}-\kappa\int_{\Omega}\frac{\phi^{2}-\phi_{k}^{2}}{2h}+\int_{\Omega} H_{0}(\phi,\phi_{k})\frac{|\nabla M|^{2}}{2}\frac{(\phi-\phi_{k})}{h}+\int_{\Omega}\frac{H_{0}(\phi,\phi_{k})}{4\alpha^{2}}(|M|^{2}-1)^{2}\frac{(\phi-\phi_{k})}{h}\\
				&\displaystyle=-\eta\int_{\Omega}\nabla\phi\cdot\nabla\frac{(\phi-\phi_{k})}{h}-\int_{\Omega}\widetilde{\Psi}'_{0}(\phi)\frac{(\phi-\phi_{k})}{h}.
			\end{array}
		\end{equation}
		Adding the expressions \eqref{identity1*}--\eqref{identity4*} and recalling \eqref{H0}, we have
		\begin{equation}\label{energy1stp}
			\begin{split}
				&\frac{1}{2}\int_{\Omega}(\rho|v|^{2}-\rho_{k}|v_{k}|^{2})
				+\int_{\Omega}\rho_{k}\frac{|v-v_{k}|^{2}}{2}
				+2h\int_{\Omega}\nu(\phi_{k})|\mathbb{D} v|^{2}
				+\int_{\Omega}(M-M_{k})\cdot \Bigl(-\mbox{div}(\xi(\phi_{k})\nabla M)\\
				&+\frac{\xi(\phi_{k})}{\alpha^{2}}(|M|^{2}M-M_{k})\Bigr)+h\int_{\Omega}\Bigl|\mbox{div}(\xi(\phi_{k})\nabla M)-\frac{\xi(\phi_{k})}{\alpha^{2}}(|M|^{2}M-M_{k})\Bigr|^{2}\\
				&\displaystyle +\frac{1}{2}\int_{\Omega}\bigl(\xi(\phi)-\xi(\phi_{k})\bigr)|\nabla M|^{2}+\frac{1}{4\alpha^{2}}\int_{\Omega}\bigl(\xi(\phi)-\xi(\phi_{k})\bigr)(|M|^{2}-1)^{2}-\frac{\kappa}{2}\int_{\Omega}({\phi^{2}-\phi_{k}^{2}})\\
				&+\eta\int_{\Omega}\nabla\phi\cdot(\nabla\phi-\nabla\phi_{k})+\int_{\Omega}\widetilde{\Psi}_{0}'(\phi)(\phi-\phi_{k})+h\int_{\Omega}|\nabla \mu|^{2}=0.
			\end{split}
		\end{equation}
		Integrating by parts the fourth term of \eqref{energy1stp}, using
		\begin{equation}\label{algebricidnty}
			\begin{split}
				\ A\cdot(A-B)=\frac{|A|^{2}}{2}-\frac{|B|^{2}}{2}+\frac{|A-B|^{2}}{2}\,\,\text{ for all }\,\,A,\,B\in\mathbb{R}^{m},\,m\in\mathbb{N},
			\end{split}
		\end{equation}
		to expand $\xi(\phi_{k})\nabla M\cdot(\nabla M-\nabla M_{k})$ and $\nabla\phi\cdot(\nabla\phi-\nabla\phi_{k})$ respectively
		and using Lemma~\ref{algebriclem} on the term $(M-M_{k})\cdot\frac{\xi(\phi_{k})}{\alpha^{2}}(|M|^{2}M-M_{k})$,
		we render
		\begin{equation}\label{energy2stp}
			\begin{split}
				&\frac{1}{2}\int_{\Omega}(\rho|v|^{2}-\rho_{k}|v_{k}|^{2})
				+\int_{\Omega}\rho_{k}\frac{|v-v_{k}|^{2}}{2}+2h\int_{\Omega}\nu(\phi_{k})|\mathbb{D} v|^{2}+\frac{1}{2}\int_{\Omega}\xi(\phi)|\nabla M|^{2}\\
				&-\frac{1}{2}\int_{\Omega}\xi(\phi_{k})|\nabla M_{k}|^{2}+\frac{1}{2}\int_{\Omega}\xi(\phi_{k})|\nabla M-\nabla M_{k}|^{2}+\frac{1}{4\alpha^{2}}\int_{\Omega}\xi(\phi)\bigl(|M|^{2}-1\bigr)^{2}\\
				&-\frac{1}{4\alpha^{2}}\int_{\Omega}\xi(\phi_{k})\bigl(|M_{k}|^{2}-1\bigr)^{2}
				+\frac{1}{4\alpha^{2}}\int_{\Omega}\xi(\phi_{k})\bigl(|M|^{2}-|M_{k}|^{2}\bigr)^{2}\\
				&+\frac{1}{2\alpha^{2}}\int_{\Omega}\xi(\phi_{k})|M\cdot(M-M_{k})|^{2}
				+\frac{1}{2\alpha^{2}}\int_{\Omega}\xi(\phi_{k})|M-M_{k}|^{2}\\
				& +h\int_{\Omega}\Bigl|\mbox{div}(\xi(\phi_{k})\nabla M)-\frac{\xi(\phi_{k})}{\alpha^{2}}(|M|^{2}M-M_{k})\Bigr|^{2}+\frac{\eta}{2}\int_{\Omega}|\nabla\phi|^{2}-\frac{\eta}{2}\int_{\Omega}|\nabla\phi_{k}|^{2}\\
				&+\frac{\eta}{2}\int_{\Omega}|\nabla\phi-\nabla\phi_{k}|^{2}+\int_{\Omega}\left(\widetilde{\Psi}_{0}(\phi)-\widetilde{\Psi}_{0}(\phi_{k})\right)-\frac{\kappa}{2}\int_{\Omega}({\phi^{2}-\phi_{k}^{2}})+
				h\int_{\Omega}|\nabla \mu|^{2}\leqslant0,
			\end{split}
		\end{equation}
		where we have used
		\begin{equation}\label{ConvConseq}\
		\int_{\Omega}\widetilde{\Psi}_{0}'(\phi)(\phi-\phi_{k})\geqslant \int_{\Omega}\left(\widetilde{\Psi}_{0}(\phi)-\widetilde{\Psi}_{0}(\phi_{k})\right),
		\end{equation}
	    which follows from the convexity of $\widetilde{\Psi}_{0}.$ Dropping some positive terms from the left hand side of the inequality \eqref{energy2stp} and recalling \eqref{deftpsi0}, we conclude the obtainment of the discrete energy estimate \eqref{discreteestimate}.
		%%%%%%%%%%%%%%%%%%%%%%%%%%%%%%%%%%%%%%%%%%%%%%%%%%%%%%%%%%%%%%%%%%%%%%%%%%%%%%%%%%%%%%%%%%%%%%%%%%%%%%%%%%%%%%%%%%%%%%%%%%%%%%%%%%%%%%%%%%%%%%%%%%%%%%%%%%%%%%%%%%%%%%%%%%%%%%%%%%%%%%%%%%%%%%%%%%%%%%%%%%%%%%%%%%%%%%%%%%%%%%%%%%%%%%%%%%%%%%%%%%%%%%%%%%%%%%%%%%%%%%%%%%%%%%%%%%%%%%%%%%%%%%%%%%%%%%%%%%%%%%%%%%%%%%%%%%%%%%
		\subsection{Proof of the existence of weak solutions to \texorpdfstring{\eqref{timediscretesystem}}{someref}}\label{existencetimediscrete}
		We will apply the Leray-Schauder fixed point principle to prove the existence of a weak solution to the discretized system \eqref{timediscretesystem}. We start by considering the following spaces 
		\begin{equation}\label{defXY}
			\begin{split}
				& X=\WND\times W^{2,2}_{n}(\Omega)\times \mathcal{D}(\partial\widetilde{E})\times W^{2,2}_{n}(\Omega),\\
				& Y= \bigl(\WND\bigr)'\times L^2(\Omega)\times L^{2}(\Omega)\times L^{2}(\Omega),
			\end{split}
		\end{equation}
		with the norm defined as the sum of the individual components of the Cartesian products. We will write \eqref{timediscretesystem} in operator notation and for that we introduce $\mathcal{N}_{k},\mathcal{F}_{k}:X\to Y$. For $w=(v,M,\phi,\mu)\in X,$ the operator $\mathcal{N}_{k}$ is defined as follows
		\begin{equation}\label{Lk}
			\begin{split}
				\mathcal{N}_{k}(w)=\begin{pmatrix}
					\mathcal{A}v\\[3.mm]
					\displaystyle-\dvr(\xi(\phi_{k})\nabla M)+\frac{\xi(\phi_{k})}{\alpha^{2}}(|M|^{2}M-M_{k})+\int_\Omega M\\[3.mm]
					\partial\widetilde{E}(\phi)+\phi\\[3.mm]
					\displaystyle-\Delta\mu+\int_\Omega\mu
				\end{pmatrix} ,
			\end{split}
		\end{equation}
		where $\mathcal{A}: \WND\to \bigl(\WND\bigr)'$ is given for all $v\in \WND$ by
		\begin{align*}
			\langle\mathcal{A}v,\widetilde\psi_1\rangle&=2\int_{\Omega}\nu(\phi_{k})\mathbb{D} v\cdot\mathbb{D} \widetilde\psi_1\text{ for all }\widetilde\psi_1\in\WND.
		\end{align*}
	     For $w=(v,M,\phi,\mu)\in X,$ the operator $\mathcal{F}_{k}$ is defined as 
		\begin{equation}\label{Fk}
			\begin{array}{ll}
				\mathcal{F}_{k}(w)=\begin{pmatrix}
					\displaystyle\ -\frac{\rho v-\rho_{k}v_{k}}{h}-\mbox{div}(\rho_{k}v\otimes v)-\nabla\mu\phi_{k}+\frac{\xi(\phi_{k})}{\alpha^{2}}(|M|^{2}M-M_{k})\nabla M-\dvr(\xi(\phi_{k})\nabla M)\nabla M\\
		\hfill \displaystyle-\frac{1}{2}\left(\dvr\,J-\frac{\rho-\rho_{k}}{h}-v\cdot\nabla\rho_{k}\right)v-\left(J\cdot\nabla\right)v
					\\[8.mm]
					\displaystyle\ -\frac{M-M_{k}}{h}-(v\cdot\nabla)M+\int_\Omega M\\[8.mm]
					\displaystyle \ \mu+\kappa\frac{(\phi+\phi_{k})}{2}- H_{0}(\phi,\phi_{k})\frac{|\nabla M|^{2}}{2}-\frac{H_{0}(\phi,\phi_{k})}{4\alpha^{2}}(|M|^{2}-1)^{2}+{\phi}\\[8.mm]
					 \displaystyle\  -\frac{\phi-\phi_{k}}{h}-(v\cdot\nabla)\phi_{k}+\int_\Omega\mu
				\end{pmatrix}.
			\end{array}
		\end{equation}
		One observes that $w=w_{k+1}(=(v_{k+1},M_{k+1},\phi_{k+1},\mu_{k+1})\in X)$ is a weak solution to the time discrete problem \eqref{timediscretesystem} iff the following holds
		$$\mathcal{N}_{k}(w)=\mathcal{F}_k(w)\,\,\mbox{in}\,\, Y.$$
		%(Let us first define another operator $\mathcal{S}_{k}$ (whose expression coincides with that of $\mathcal{N}_{k}$ but with different domain). Consider $\mathcal{S}_{k}$ as an operator defined from $X_{2}\longrightarrow (X_{2})'$ where
		%$$X_{2}=\WND\times W^{1,2}(\Omega)\times W^{1,2}(\Omega)\times W^{1,2}(\Omega)$$ 
		%and the is given by the right hand side of \eqref{Lk}. Of course here the operators $\Delta_{N}$ and $\mbox{div}$ are interpreted in weak form)
		To prove the existence of a solution to this operator equation we next show some properties of the operators $\mathcal{N}_{k}$ and $\mathcal{F}_{k}.$
		%\begin{equation}\label{opweakform}
			%\begin{array}{ll}
			%	&\displaystyle\langle-\dvr_N \Phi,\psi_2\rangle=\int_\Omega \Phi\cdot\nabla\psi_2\text{ for all }\psi_2\in W^{1,2}(\Omega).\\
				%&\displaystyle\langle-\Delta_N\varphi,\widetilde\psi_3\rangle=\int_{\Omega}\nabla\varphi\cdot\nabla\widetilde\psi_3\text{ for all }\widetilde\psi_3\in W^{1,2}(\Omega).
			%\end{array}
		%\end{equation}
		%\subsection{$\mathcal{S}_{k}^{-1}:X_{2}\longrightarrow X_{2}^{-1}$ is bounded and continuous}
		%We first claim that $\mathcal{S}_k$ is monotone on $X_{2}$. To this end, we consider an arbitrary couple $w_1,w_2\in X_{2}$, $w_i=(v_i,M_i,\phi_i,\mu_i),\ i=1,2$ and we compute
		%\begin{align*}
			%\langle\mathcal{S}_k(w_1)-\mathcal{S}_k(w_2),w_1-w_2\rangle =& \int_\Omega |\nabla(v_1-v_2)|^2+\int_\Omega\xi(\phi_k)|\nabla (M_1-M_2)|^2\\&+\int_\Omega\frac{\xi(\phi_k)}{\alpha^2}\bigl(|M_1|^2M_1-|M_2|^2M_2\bigr)\cdot(M_1-M_2)+\bigl|\overline{M_1-M_2}\bigr|^2\\
			%& +\eta\int_\Omega|\nabla (\phi_1-\phi_2)|^2+\int_\Omega\bigl(\widetilde{\Psi}'_{0}(\phi_{1})-\widetilde{\Psi}'_{0}(\phi_{2})\bigr)(\phi_1-\phi_2)+\bigl(\overline{\phi_1-\phi_2}\bigr)^2\\
			%&+\int_\Omega|\nabla(\mu_1-\mu_2)|^2+\bigl(\overline{\mu_1-\mu_2}\bigr)^2=\sum_{i=1}^9 I_i.,
		%\end{align*}
		%{\color{red} There are problems with the above argument, so we will rather prove it componenetwise}
		
		\subsubsection{Invertibility and continuity of the inverse of $\mathcal{N}_{k}$ between suitable spaces }\label{invertNk}
		Let us consider the operator $\mathcal{N}_{k}$ component wise. It is well known that $\mathcal{A}: \WND\to \bigl(\WND\bigr)'$ is bounded and invertible, and $\mathcal{A}^{-1}:\bigl(\WND\bigr)'\longrightarrow\WND$ is bounded and continuous. For a proof one can adapt the arguments we are going to use to prove similar issues for the second component of $\mathcal{N}_{k}$. We choose to present a detailed argument for the second component since it is more involved than the first one.\\
		We remark that the second component of  $\mathcal{N}_{k}$ defines a bounded operator from $W^{2,2}_{n}(\Omega)\longrightarrow L^{2}(\Omega).$ 
		As a first step towards proving that the second component of $\mathcal{N}_{k}$ admits of a bounded and continuous inverse from $L^{2}(\Omega)$ to $W^{2,2}_{n}(\Omega),$ we first claim that the operator
		\begin{equation}\label{Moperator}
	    	\mathcal{B}_{k}(M)=-\dvr_{N}(\xi(\phi_{k})\nabla M)+\frac{\xi(\phi_{k})}{\alpha^{2}}(|M|^{2}M-M_{k})+\int_\Omega M:W^{1,2}(\Omega)\longrightarrow (W^{1,2}(\Omega))',
		\end{equation} 
		where $\dvr_{N}$ is interpreted in the following weak sense
		$$\langle-\dvr_N \Phi,\psi_2\rangle=\int_\Omega \Phi\cdot\nabla\psi_2\text{ for all }\psi_2\in W^{1,2}(\Omega)\,\,\mbox{and}\,\, \Phi\in L^{2}(\Omega),$$
	  is invertible and $\mathcal{B}_{k}^{-1}:(W^{1,2}(\Omega))'\longrightarrow W^{1,2}(\Omega)$ is bounded and continuous. In order to show this, we consider an arbitrary couple $M_{1},\,\,M_{2}\in W^{1,2}(\Omega).$ Then we compute the following duality product
		\begin{equation}\label{computedifference}
			\begin{split}
				&\langle \mathcal{B}_{k}(M_{1})-\mathcal{B}_{k}(M_{2}),M_{1}-M_{2}\rangle\\
				&=\int_\Omega\xi(\phi_k)|\nabla (M_1-M_2)|^2
				+\int_\Omega\frac{\xi(\phi_k)}{\alpha^2}\bigl(|M_1|^2M_1-|M_2|^2M_2\bigr)\cdot(M_1-M_2)+\left(\int_\Omega M_1-M_2\right)^2\\
				&=I_{1}+I_{2}+I_{3}.
			\end{split}
		\end{equation} 
		One observes $I_1\geq c_1\|\nabla (M_1-M_2)\|^2_{L^2(\Omega)}$ since $\xi$ is non degenerate (cf.\ \eqref{XiAssum}). The monotonicity of $\alpha\mapsto |\alpha|^2\alpha$ implies $I_2\geq 0.$ Employing the lower bound on $\xi$ again we infer $I_1+I_3\geq c\|M_1-M_2\|^2_{W^{1,2}(\Omega)}.$\\
		Hence 
		$$\displaystyle\langle \mathcal{B}_{k}(M_{1})-\mathcal{B}_{k}(M_{2}),M_{1}-M_{2}\rangle \geqslant c\|M_{1}-M_{2}\|^{2}_{W^{1,2}(\Omega)},$$
		implying $\mathcal{B}_{k}:W^{1,2}(\Omega)\longrightarrow (W^{1,2}(\Omega))'$ is strongly monotone.\\
		Since $W^{1,2}(\Omega)\hookrightarrow L^{6}(\Omega),$ one justifies the boundedness of $\mathcal{B}_{k}:W^{1,2}(\Omega)\longrightarrow (W^{1,2}(\Omega))'.$ Now using the Lebesgue dominated convergence theorem one checks that  $\mathcal{B}_k$ is radially continuous on $W^{1,2}(\Omega)$, i.e., for each pair $M,\widetilde M\in W^{1,2}(\Omega)$ the function $t\in\mathbb{R}\mapsto \langle\mathcal{B}_k(M+t\widetilde{M}),\widetilde M\rangle$ is continuous. It is not hard to check that $\langle \mathcal{B}_k(M),M\rangle \geq c\|M\|^2_{W^{1,2}(\Omega)}-c_k$, for any $M\in W^{1,2}(\Omega)$ with $c_k$ depending on $M_k$. The latter inequality implies that $\mathcal{B}_k$ is coercive on $W^{1,2}(\Omega)$, i.e.,
		$$\lim_{\|M\|_{W^{1,2}(\Omega)}\to\infty}\frac{\langle\mathcal{B}_k(M),M\rangle}{\|M\|_{W^{1,2}(\Omega)}}=\infty.$$ 
		Using \cite[Theorem 2.14]{Ro05} we obtain the existence of the inverse operator $\mathcal{B}_k^{-1}:(W^{1,2}(\Omega))'\to W^{1,2}(\Omega)$ that is bounded and Lipschitz continuous.\\
		We now claim that
		\begin{equation}\label{claim}
		\begin{array}{l}
		 \mathcal{B}^{-1}_{k}: L^{2}(\Omega)\longrightarrow W^{2,2}_{n}(\Omega)\,\,\mbox{is bounded and continuous.}
		\end{array}
		\end{equation}
		The proof of this claim can be performed by using a boot-strap argument and using the regularity results for the following set of equations 
		\begin{equation}\nonumber
		\begin{alignedat}{2}
		\Delta M=&\frac{1}{\xi(\phi_k)}\Bigl(F-\xi'(\phi_{k})\nabla M\nabla\phi_k+\frac{\xi(\phi_{k})}{\alpha^{2}}\bigl(|M|^{2}M-M_{k}\bigr)\Bigr)\quad&&\text{ in }\Omega,\\
		\partial_{n}M=&0&&\text{ on }\partial\Omega,
		\end{alignedat}
		\end{equation}
		%This is shown already in \cite[Section 4.2.1]{KMS20}.
	    for any $F\in L^{2}(\Omega)$, where we apply the fact that $\mathcal{B}_k^{-1}:(W^{1,2}(\Omega))'\to W^{1,2}(\Omega)$ is bounded and continuous. We refer the readers to \cite[Section 4.2.1, p~14-15]{KMS20} for a concrete proof. This proves our claim that the second component of $\mathcal{N}_{k}$ has a bounded and continuous inverse from $W^{2,2}_{n}(\Omega)$ and $L^{2}(\Omega).$\\[2.mm]
		Next we consider the third component of $\mathcal{N}_{k}.$ Recalling the definition of $\widetilde{E}$ from \eqref{defdom}--\eqref{Ephi}, it can be justified in view of Proposition \ref{ellipticreg} that
		\begin{equation}\label{PartET}
		    \partial\widetilde{E}+I:\mathcal{D}(\partial\widetilde{E})\longrightarrow L^{2}(\Omega)\text{ is invertible with  bounded inverse}
		\end{equation}
		($\mathcal{D}(\partial{\widetilde{E}})$ is identified with a subspace of $W^{2,2}_{n}(\Omega)$ as in \eqref{domainsubg}), where $I:\mathcal{D}(\partial\widetilde{E})\longrightarrow L^{2}(\Omega)$ is the inclusion map. Moreover, we can follow arguments from \cite[p.~466-467]{AbDeGa113} to show that the inverse operator
		\begin{equation}
		    (\partial\widetilde E+I)^{-1}:L^{2}(\Omega)\longrightarrow W^{2-s,2}_n(\Omega)\text{ is continuous for any }s\in(0,\tfrac{1}{4}).
		\end{equation}
		%The invertibility and the boundedness of the inverse of the operator $\partial\widetilde{E}+I$ (we recall the definition of $\widetilde{E}$ from \eqref{defdom}--\eqref{Ephi}) follows from the following lemma, the proof of which can be found in \cite[p. 17]{AbDeGa113}.
		%	\begin{lem}\label{invertmonotone}
		%	The operator $\partial\widetilde{E}$ is maximal monotone and hence
	%		$$\partial\widetilde{E}+I: \mathcal{D}(\partial\widetilde{E})\longrightarrow L^{2}(\Omega)$$
	%		($\partial{\widetilde{E}}$ is identified with a subspace of $W^{2,2}_{n}(\Omega)$ as in \eqref{domainsubg}) is invertible, where $I:\mathcal{D}(\partial\widetilde{E})\longrightarrow L^{2}(\Omega)$ is the inclusion map.\\
	%		Moreover the inverse considered as a mapping $L^{2}(\Omega)\longrightarrow W^{2-s,2}(\Omega)$ is continuous for $0<s<\frac{1}{4}.$ 
	%	\end{lem}
		From now on we fix $s=\frac{1}{8}$ for definiteness.\\
		Finally let us consider the last component of $\mathcal{N}_{k}.$ From standard elliptic theory, the operator 
		$$-\Delta(\cdot)+\int_{\Omega}\cdot: W^{2,2}_{n}(\Omega)\longrightarrow L^{2}(\Omega)$$ 
		is bounded invertible with bounded and continuous inverse.\\
		In summary we have shown that  
		\begin{equation}\label{NKInvert}
		\text{the map }\mathcal{N}_{k}:X\longrightarrow Y\text{ is bounded, invertible, the inverse is bounded}   
		\end{equation}
		and the inverse map is continuous from $Y$ to $\widetilde{X},$ where 
		\begin{equation}\label{defXt}
			\begin{array}{l}
				\widetilde{X}=\WND\times W^{2,2}_{n}(\Omega)\times W^{\frac{15}{8},2}_{n}(\Omega)\times W^{2,2}_{n}(\Omega).
			\end{array}
		\end{equation}
		Next we show that $\mathcal{F}_{k}:\widetilde{X}\longrightarrow Y$ is continuous and compact.
		
		%%%%%%%%%%%%%%%%%%%%%%%%%%%%%%%%%%%%%%%%%%%%%%%%%%%%%%%%%%%%%%%%%%%%%%%%%%%%%%%%%%%%%%%%%%%%%%%%%%%%%%%%%%%%%%%%%%%%%%%%%%%%%%%%%%%%%%%%%%%%%%%%%%%%%%%%%%%%%%%%%%%%%%%%%%%%%%%%%%%%%%%%%%%%%%%%%%%%%%%%%%%%%%%%%%%%%%%%%%%%%%%%%%%%%%%%%%%%%%%%%%%%%%%%%%%%%%%%%%%%%%%%%%%%%%%%%%%%%%%%%%%%%%%%%%%%%%%%%%%%%%%%%%%%%%%%%%%%%%
		\subsubsection{The operator $\mathcal{F}_{k}:\widetilde{X}\longrightarrow Y$ is continuous and compact}
		For the operator $\mathcal{F}_{k}=(\mathcal{F}_{k}^{1},\mathcal{F}_{k}^{2},\mathcal{F}_{k}^{3},\mathcal{F}_{k}^{4}),$ we will verify the continuity and compactness component wise.\\
		Let us start with $\mathcal{F}^{1}_k:\widetilde{X}\longrightarrow (W^{1,2}_{0,\dvr}(\Omega))'.$  In this direction we will collect estimates of terms appearing in the expression of $\mathcal{F}_{k}^{1}.$ The following estimates are obtained by using H\"{o}lder's inequality and standard Sobolev embeddings.
		\begin{equation}\label{boundednesFk1}
		\begin{split}
		\|\rho v\|_{L^{\frac{3}{2}}(\Omega)}&\leqslant C\|v\|_{W^{1,2}(\Omega)}\left(\|\phi\|_{L^{2}(\Omega)}+1\right),\\
		\|\mbox{div}(\rho_{k}v\otimes v)\|_{L^{\frac{3}{2}}(\Omega)}&\leqslant C_{k}\|v\|^{2}_{W^{1,2}(\Omega)},\\
		\|\nabla\mu\phi_{k}\|_{L^{\frac{3}{2}}(\Omega)}&\leqslant C_{k}\|\mu\|_{W^{2,2}(\Omega)},\\
		\|\xi(\phi_{k})(|M|^{2}M-M_{k})\nabla M\|_{L^{\frac{3}{2}}(\Omega)}&\leqslant C_{k}\left(\|M\|^{4}_{W^{2,2}(\Omega)}+\|M\|_{W^{2,2}(\Omega)}\right),\\
		\|\mathrm{div}(\xi(\phi_{k})\nabla M)\nabla M\|_{L^{\frac{3}{2}}(\Omega)}&\leqslant C_{k}\|M\|^{2}_{W^{2,2}(\Omega)},\\
		\|\left(\mbox{div}J\right)v\|_{L^{\frac{3}{2}}(\Omega)}&\leqslant C\|v\|_{W^{1,2}(\Omega)}\|\mu\|_{W^{2,2}(\Omega)},\\
		\|\left(J\cdot\nabla\right)v\|_{L^{\frac{3}{2}}(\Omega)}&\leqslant C\|v\|_{W^{1,2}(\Omega)}\|\mu\|_{W^{2,2}(\Omega)},\\
		\|(v\cdot\nabla\rho_{k})v\|_{L^{\frac{3}{2}}(\Omega)}&\leqslant C_{k}\|v\|^{2}_{W^{1,2}(\Omega)}.
		\end{split}
		\end{equation}
		%The explanation of the obtainment of the bounds \eqref{boundednesFk1}$_{1},$ \eqref{boundednesFk1}$_{2},$ \eqref{boundednesFk1}$_{4}$ and \eqref{boundednesFk1}$_{5}$ can be found in \cite[p. 17-18]{AbDeGa113}. Whereas for the derivation of \eqref{boundednesFk1}$_{7}$ and \eqref{boundednesFk1}$_{8}$ we refer the readers to \cite[p. 15-16]{KMS20}. The proof of \eqref{boundednesFk1}$_{3}$ is easy since the term $\nabla\mu\phi_{k}$ is linear in $\mu$ and the estimate \eqref{boundednesFk1}$_{6}$ can be obtained as follows:
		%$$\|(v\cdot\nabla\rho_{k})v\|_{L^{\frac{3}{2}}(\Omega)}\leqslant \|v\|^{2}_{L^{4}(\Omega)}\|\nabla\rho_{k}\|_{L^{6}(\Omega)}\leqslant C_{k}\|v\|^{2}_{W^{1,2}(\Omega)}.$$\\
		Hence the estimates above prove the boundedness of $\mathcal{F}_{k}^{1}$ from $\widetilde{X}$ to $L^{\frac{3}{2}}(\Omega).$ One can use similar estimates to show that $\mathcal{F}_{k}^{1}$ is continuous from $\widetilde{X}$ to $L^{\frac{3}{2}}(\Omega).$ Next the compact embedding $L^{\frac{3}{2}}(\Omega)\longrightarrow (W^{1,2}_{0,\dvr}(\Omega))'$ furnishes the continuity and compactness of $\mathcal{F}^{1}_{k}:\widetilde{X}\longrightarrow (W^{1,2}_{0,\dvr}(\Omega))'.$\\
		Now we prove that $\mathcal{F}_{k}^{2}:\widetilde{X}\longrightarrow L^{2}(\Omega)$ is continuous and compact. One first observes
		\begin{equation}\label{vgMW132}
		\displaystyle\|(v\cdot\nabla)M\|_{W^{1,\frac{3}{2}}(\Omega)}\leqslant C\|v\|_{W^{1,2}(\Omega)}\|M\|_{W^{2,2}(\Omega)}
		\end{equation}
		 as in \cite[p.~16]{KMS20}. The continuity of $\mathcal{F}_{k}^{2}:\widetilde{X}\longrightarrow W^{1,\frac{3}{2}}(\Omega)$ relies on a similar estimate and can be concluded without any difficulty. Further in view of the compactness of the embedding of $W^{1,\frac{3}{2}}(\Omega)$ to $L^2(\Omega)$ the asserted continuity and compactness follows.\\
		 Next we show that $\mathcal{F}_{k}^{3}: \widetilde{X}\longrightarrow L^{2}(\Omega)$ is continuous and compact. For the proof we take a different route than the ones used for $\mathcal{F}^{i}_{k},$ $i\in\{1,2\}.$ First we introduce
		 $$\widetilde{Y}=L^{2}(\Omega)\times W^{\frac{15}{8},2}(\Omega)\times W^{\frac{7}{4},2}(\Omega)\times W^{\frac{15}{8},2}(\Omega).$$
		 One observes that the embedding $\widetilde{X}\longrightarrow\widetilde{Y}$ is compact (since in dimension three the embedding $W^{m+k,2}\hookrightarrow W^{m,2}$ is compact for any $0<k<\frac{3}{2}$). We study the boundedness and continuity of the operator $\mathcal{G}:\widetilde Y\longrightarrow L^2(\Omega)$ that is defined as in the third component of \eqref{Fk}. Once we verify its boundedness and continuity we immediately conclude that $\mathcal F^3_k:\widetilde{X}\longrightarrow L^{2}(\Omega)$ is bounded, continuous and compact as the composition of the embedding $\widetilde{X}\longrightarrow \widetilde{Y}$ and $\mathcal G$.  In order to conclude the boundedness of $\mathcal G$, we collect the following estimates:
		 \begin{equation}\label{boundednesFk3}
		 \begin{split}
		 	\|H_{0}(\phi,\phi_{k})|\nabla M|^{2}\|_{L^{2}(\Omega)}&\leqslant C_{k}\|M\|^{2}_{W^{\frac{15}{8},2}(\Omega)},\\[1.mm]
		 	\|H_{0}(\phi,\phi_{k})(|M|^{2}-1)^{2}\|_{L^{2}(\Omega)}&\leqslant C_{k}\left(\|M\|^{4}_{W^{\frac{15}{8},2}(\Omega)}+1\right).
		 	%\|(v\cdot\nabla)\phi_{k}\|_{W^{1,\frac{3}{2}}(\Omega)}& \leqslant C_{k}\|v\|_{W^{1,2}(\Omega)}.
		 \end{split}
		 \end{equation}
		 Indeed, since $\xi(\cdot)\in C^{1}(\mathbb{R})$, one infers $\|H_{0}(\phi,\phi_{k})\|_{L^{\infty}(\Omega)}\leqslant C_{k}$ by using 
		 the mean value theorem and the upper bound of $\xi'(\cdot)$ (cf.\ assumption \eqref{XiAssum}). Further the fact that $\nabla M$ is bounded in $W^{\frac{7}{8},2}(\Omega)$ and the continuous embedding $W^{\frac{7}{8},2}(\Omega)\hookrightarrow L^{4}(\Omega)$ proves \eqref{boundednesFk3}$_{1}.$ Whereas \eqref{boundednesFk3}$_{2}$ is a consequence of the continuous embedding $W^{\frac{15}{8},2}(\Omega)\hookrightarrow L^{\infty}(\Omega).$ The boundedness and continuity of the linear terms in the expression of $\mathcal{F}_{k}^{3}$ are trivially concluded.\\ 
		 In the spirit of the boundedness estimates \eqref{boundednesFk3}, the continuity of the nonlinear terms in $\mathcal{F}_{k}^{3},$  i.e.\  $\displaystyle\frac{1}{2}H_{0}(\phi,\phi_{k})|\nabla M|^{2}$ and $\displaystyle\frac{1}{4\alpha^2}H_{0}(\phi,\phi_{k})(|M|^{2}-1)^{2},$ can be proved by the arguments as in \cite[p.~16]{KMS20} adjusted to the current set-up.
		Hence we have proved that $\mathcal{F}_{k}^{3}: \widetilde{X}\longrightarrow L^{2}(\Omega)$ is continuous and compact.\\
		Next we consider $\mathcal{F}_{k}^{4}.$ In the spirit of \eqref{vgMW132} we derive the following estimate
		\begin{equation}\label{vgpW132}
		\displaystyle\|(v\cdot\nabla)\phi_{k}\|_{W^{1,\frac{3}{2}}(\Omega)} \leqslant C_{k}\|v\|_{W^{ ,2}(\Omega)},
		\end{equation}
		which verifies the boundedness of $\mathcal{F}_{k}^{4}:\widetilde{X}\longrightarrow W^{1,\frac{3}{2}}(\Omega).$ By \eqref{vgpW132} the continuity of $\mathcal{F}_{k}^{4}:\widetilde{X}\longrightarrow W^{1,\frac{3}{2}}(\Omega)$ can be concluded in a straight forward manner. Next using the compact embedding $W^{1,\frac{3}{2}}(\Omega)\hookrightarrow L^{2}(\Omega),$ one at once renders that $\mathcal{F}_{k}^{4}:\widetilde{X}\longrightarrow L^{2}(\Omega)$ is continuous and compact.\\
		In view of the arguments above we finally have proved that $\mathcal{F}_{k}:\widetilde{X}\longrightarrow Y$ is continuous and compact.
	%	\subsubsection{$\mathcal{F}_{k}:\widetilde{X}\longrightarrow\widetilde{Y}$ is continuous}	
	\subsubsection{The fixed point argument}
		We now show the existence of a $w\in X$ (we recall the definition of $X$ from \eqref{defXY}$_{1}$) satisfying
		\begin{equation}\label{fxdpntmap}
			\begin{array}{l}
				\mathcal{N}_{k}(w)=\mathcal{F}_{k}(w)\,\,\mbox{in}\,\, Y.
			\end{array}
		\end{equation}
		%We will first prove the existence of such a $w$ in $\widetilde{X}$ as defined in \eqref{defXt} and next bootstrap the regularity to obtain $w\in X.$\\
		For that purpose it is sufficient to prove the existence of a fixed point of the operator $\mathcal{F}_{k}\circ \mathcal{N}^{-1}_{k}$ on $Y$, i.e., the existence of $z\in Y$ satisfying
		\begin{equation}\label{2ndfxd}
			\begin{array}{l}
				z=(\mathcal{F}_{k}\circ \mathcal{N}^{-1}_{k})z \,\,\mbox{in}\,\,Y,
			\end{array}
		\end{equation}
		since the invertibility of the operator $\mathcal{N}_{k}:{X}\to Y$ implies the obtainment of $w\in {X}$ satisfying \eqref{fxdpntmap} by using $w=\mathcal{N}^{-1}_{k}(z)$.\\
		In order to show the existence of a fixed point of the operator equation \eqref{2ndfxd} we apply the Leray-Schauder fixed point theorem \cite[Theorem 10.3]{GilTr01} to the continuous and compact operator $\mathcal{F}_{k}\circ \mathcal{N}^{-1}_{k}:Y\longrightarrow Y.$ To this end we verify that:
		\begin{equation}\label{lerayfxpt}
			\begin{array}{l}
				\mbox{There exists}\,\, r>0 \,\,\mbox{such that if}\,\, z\in Y\,\, \mbox{solves}\,\, z=\lambda(\mathcal{F}_{k}\circ \mathcal{N}^{-1}_{k})z\,\,\mbox{with}\,\,\lambda\in[0,1],\\
				\mbox{then it holds}\,\,\|z\|_{Y}\leqslant r.
			\end{array}
		\end{equation}
		Let $z\in Y$ satisfy $z=\lambda(\mathcal{F}_{k}\circ \mathcal{N}^{-1}_{k})z$ in $Y$ with some $\lambda\in [0,1]$. Then
		$$w=(v,M,\phi,\mu)=\mathcal{N}_{k}^{-1}z,$$
		solves
		\begin{equation}\label{LkFk}
			\begin{array}{l}
		    \mathcal{N}_{k}(w)-\lambda\mathcal{F}_{k}(w)=0\,\,\mbox{in}\,\, Y.
			\end{array}
		\end{equation}
		Let us first prove that
		\begin{equation}\label{boundtw}
			\|w\|_{\widetilde{X}}\leqslant C_{k},
		\end{equation}
		with $C_k$ independent of $\lambda\in[0,1].$ Then we will bootstrap the regularity to have
			\begin{equation}\label{boundw}
			\|w\|_{{X}}\leqslant C_{k},
			\end{equation}
			with $C_k$ independent of $\lambda\in[0,1],$ 
		from which \eqref{lerayfxpt} follows due to the boundedness of $\mathcal{N}_k: X\longrightarrow Y$, cf.\ \eqref{NKInvert}.  One recalls the definitions of $\mathcal{N}_{k}$ and $\mathcal{F}_{k}$ from \eqref{Lk} and \eqref{Fk}, tests the first component of \eqref{LkFk} by $v$, the second component by $-\dvr(\xi(\phi_k)\nabla M)+\frac{\xi(\phi_{k})}{\alpha^{2}}(|M|^{2}M-M_{k})$, the third component by $\frac{\phi-\phi_{k}}{h}$ and the fourth component by $\mu$. The application of \eqref{algebricidnty}, Lemma~\ref{algebriclem} and identities \eqref{someidentities} (similar arguments leading to \eqref{energy2stp} from \eqref{energy1stp}) yield the following after dropping some positive terms from the left hand side (similar to the obtainment of \eqref{discreteestimate} from \eqref{energy2stp})
		\begin{equation}\label{energylambda}
			\begin{split}
				&\frac{\lambda}{h}\left(\frac{1}{2}\int_{\Omega}\rho|v|^{2}-\frac{1}{2}\int_{\Omega}\rho_{k}|v_{k}|^{2}\right)+2\int_{\Omega}\nu(\phi_{k})|\mathbb{D}v|^{2}+\frac{\lambda}{h}\left(\frac{1}{2}\int_{\Omega}\xi(\phi)|\nabla M|^{2}
				-\frac{1}{2}\int_{\Omega}\xi(\phi_{k})|\nabla M_{k}|^{2}\right)\\
				& +\frac{\lambda}{h}\left(\frac{1}{4\alpha^{2}}\int_{\Omega}\xi(\phi)\bigl(|M|^{2}-1\bigr)^{2}-\frac{1}{4\alpha^{2}}\int_{\Omega}\xi(\phi_{k})\bigl(|M_{k}|^{2}-1\bigr)^{2}\right)\\
				& +\int_{\Omega}\left|\mbox{div}(\xi(\phi_{k})\nabla M)-\frac{\xi(\phi_{k})}{\alpha^{2}}(|M|^{2}M-M_{k})\right|^{2}+\frac{1}{h}\left(\frac{\eta}{2}\int_{\Omega}|\nabla\phi|^{2}-\frac{\eta}{2}\int_{\Omega}|\nabla\phi_{k}|^{2}\right)\\
				&+\frac{1}{h}\int_{\Omega}\left(\widetilde{\Psi}_{0}(\phi)-\widetilde{\Psi}_{0}(\phi_{k})\right)-\lambda\frac{1}{h}\int_{\Omega}\kappa\frac{(\phi^{2}-\phi_{k}^{2})}{2}+\int_{\Omega}|\nabla \mu|^{2}+(1-\lambda)\left(\int_\Omega\mu\right) ^{2}\\
				&\displaystyle+\frac{(1-\lambda)}{h}\int_{\Omega}\left(\frac{\phi^{2}}{2}-\frac{\phi^{2}_{k}}{2}\right)
				+(1-\lambda)\int_\Omega M \int_\Omega\left(-\mbox{div}(\xi(\phi_{k})\nabla M)+\frac{\xi(\phi_{k})}{\alpha^{2}}(|M|^{2}M-M_{k})\right)\leqslant0.
			\end{split}
		\end{equation}
		Once again we recall that in obtaining the above inequality one expands $\xi(\phi_{k})\nabla M\cdot(\nabla M-\nabla M_{k})$ and $\nabla\phi\cdot(\nabla\phi-\nabla\phi_{k})$ by using \eqref{algebricidnty}
		and uses Lemma~\ref{algebriclem} to expand the term $(M-M_{k})\cdot\frac{\xi(\phi_{k})}{\alpha^{2}}(|M|^{2}M-M_{k}).$\\
		In order to obtain the inequality \eqref{energylambda} we also have used \eqref{ConvConseq}.
		When $0\leqslant\lambda<1,$ we use Young's and H\"{o}lder's inequality to estimate the term appearing in the last line of \eqref{energylambda} to infer:
		%\begin{equation}\label{PhiProdEst}
		%\begin{split}
		%\left|\frac{1-\lambda}{h}\overline\phi\overline\phi_k\right|\leqslant &\frac{1-\lambda}{h}\left(\gamma|\Omega|^3\int_\Omega|\phi|^4+c\gamma^{-\frac{1}{3}}|\Omega|\int_\Omega|\phi_k|^\frac{4}{3}\right)\\
		%\leqslant &\frac{1-\lambda}{h}\left(3\gamma|\Omega|^3\left(\int_\Omega(\phi^2-1)^2+\frac{1}{2}|\Omega|\right)+c\gamma^{-\frac{1}{3}}|\Omega|\int_\Omega|\phi_k|^\frac{4}{3}\right),
		%\end{split}
		%\end{equation}
		\begin{equation}\label{Jensen}
			\begin{split}
				&\left|(1-\lambda)\int_\Omega M \int_{\Omega}\left(-\mbox{div}(\xi(\phi_{k})\nabla M)+\frac{\xi(\phi_{k})}{\alpha^{2}}(|M|^{2}M-M_{k})\right)\right|\\
				&\leqslant (1-\lambda)\epsilon|\Omega|\int_{\Omega}\left|\mbox{div}(\xi(\phi_{k})\nabla M)-\frac{\xi(\phi_{k})}{\alpha^{2}}(|M|^{2}M-M_{k})\right|^{2}+c\frac{1-\lambda}{\epsilon}\left(\int_\Omega M\right)^{2}
			\end{split}
		\end{equation}
		for some positive parameter $\epsilon>0$.\\
		Since $|1-\lambda|\leqslant 1,$ for small enough choice of the parameter $\epsilon>0,$ the first term on the right hand side  \eqref{Jensen} can be absorbed in the fifth summand appearing on the left hand side of \eqref{energylambda}. We will now estimate the second term on the right hand side of \eqref{Jensen}. In that direction we follow the arguments used to show \cite[p.~18, $(4.32)$]{KMS20} to obtain
		\begin{equation}\label{normboundM}
		\ \int_{\Omega}|\nabla M|^{2}+\int_{\Omega} |M|^{4}+(1-\lambda)\left(\int_\Omega M\right)^{2}\leqslant C_{k},
		\end{equation}
		where $C_{k}>0$ is independent of $\lambda>0.$\\
	   For a small enough choice of the parameter  $\epsilon>0,$ using \eqref{normboundM} and \eqref{Jensen} in \eqref{energylambda} we obtain in particular 
		\begin{equation}\label{energylambda*}
			\begin{split}
				& 2h\int_{\Omega}\nu(\phi_{k})|\mathbb{D}v|^{2}
				+\frac{h}{2}\int_{\Omega}\left|\mbox{div}(\xi(\phi_{k})\nabla M)-\frac{\xi(\phi_{k})}{\alpha^{2}}(|M|^{2}M-M_{k})\right|^{2}+\frac{\eta}{2}\int_{\Omega}|\nabla\phi|^{2}\\
				&+\int_{\Omega}\widetilde{\Psi}_{0}(\phi)+h\int_{\Omega}|\nabla \mu|^{2}+h(1-\lambda)\left(\int_\Omega \mu \right)^{2}-\lambda\int_{\Omega}\kappa \frac{\phi^{2}}{2}\\
				&\leqslant \int_{\Omega}\frac{\rho_{k}|v_{k}|^{2}}{2}+\frac{1}{2}\int_{\Omega}\phi_{k}^{2}+\frac{\eta}{2}\int_{\Omega}|\nabla\phi_{k}|^{2}+\int_{\Omega}|\kappa|\frac{\phi^{2}_{k}}{2}+\int_{\Omega}\widetilde{\Psi}_{0}(\phi_{k})+\frac{1}{2}\int_{\Omega}\xi(\phi_{k})|\nabla M_{k}|^{2}\\
				&+\frac{1}{4\alpha^{2}}\int_{\Omega}\xi(\phi_{k})\bigl(|M_{k}|^{2}-1\bigr)^{2}+ C_{k} \leqslant C_{k}.
			\end{split}
		\end{equation}
		Indeed, we obtain \eqref{energylambda*} from \eqref{energylambda} by using the positive lower bound on $\rho$ from \eqref{BoundsOnRho}, which follows since $w=(v,M,\phi,\mu)=\mathcal{N}_{k}^{-1}z\in X$, cf.\ \eqref{NKInvert}, implying  $\phi\in\mathcal{D}(\partial\widetilde{E})$ and hence $\phi\in[-1,1]$ a.e.  The fact that $\phi\in[-1,1]$ a.e.\  and $\lambda\in[0,1]$ alongside $\widetilde{\Psi}_0\in C([-1,1])$ allows us to obtain
		$$\left|\int_{\Omega}\widetilde{\Psi}_0(\phi)\right|\leqslant C\quad\mbox{and}\quad \left|\int_{\Omega}\kappa\frac{\phi^{2}}{2}\right|\leqslant C,$$ 
		for some positive constant $C$ and hence the term $\displaystyle\int_{\Omega}\widetilde{\Psi}_{0}(\phi)$ and $\displaystyle-\lambda\int_{\Omega}\kappa\frac{\phi^{2}}{2}$ can be dropped from the left hand side of \eqref{energylambda*}. Hence  
		\begin{equation}\label{energylambda**}
			\begin{split}
				& 2h\int_{\Omega}\nu(\phi_{k})|\mathbb{D}v|^{2}
				+\frac{h}{2}\int_{\Omega}\left|\mbox{div}(\xi(\phi_{k})\nabla M)-\frac{\xi(\phi_{k})}{\alpha^{2}}(|M|^{2}M-M_{k})\right|^{2}+\frac{\eta}{2}\int_{\Omega}|\nabla\phi|^{2}\\
				&+h\int_{\Omega}|\nabla \mu|^{2}+h(1-\lambda)\left(\int_\Omega \mu \right)^{2}\leqslant C_{k}. %\int_{\Omega}\frac{\rho_{k}|v_{k}|^{2}}{2}+\frac{1}{2}\int_{\Omega}|\phi_{k}|^{2}+\frac{\eta}{2}\int_{\Omega}|\nabla\phi_{k}|^{2}+\int_{\Omega}|\kappa|\frac{\phi^{2}_{k}}{2}+\int_{\Omega}\widetilde{\Psi}_{0}(\phi_{k})+\frac{1}{2}\int_{\Omega}\xi(\phi_{k})|\nabla M_{k}|^{2}\\
				%&+\frac{1}{4\alpha^{2}}\int_{\Omega}\xi(\phi_{k})\bigl(|M_{k}|^{2}-1\bigr)^{2}\leqslant C_{k}.
			\end{split}
		\end{equation}
		One uses \eqref{energylambda**}, \eqref{nuAssum}, Korn's and Poincar\'e's inequality and the fact that $\phi\in[-1,1]$ a.e.\ to render $\|v\|_{W^{1,2}(\Omega)}+\|\phi\|_{W^{1,2}(\Omega)}\leqslant C_{k}.$  We conclude $\|M\|_{W^{1,2}(\Omega)}\leqslant C_{k}$ independently of $\lambda$ from \eqref{normboundM}. The second term of \eqref{energylambda**}$_{1}$ provides
		\begin{equation}\label{divMinq}
			\left\|\dvr(\xi(\phi_{k})\nabla M)-\frac{\xi(\phi_{k})}{\alpha^{2}}(|M|^{2}M-M_{k})\right\|_{L^{2}(\Omega)}\leqslant C_{k}.
		\end{equation}
		%Since $W^{1,2}(\Omega)\hookrightarrow L^{6}(\Omega)$ one has $\|M\|_{L^{6}(\Omega)}\leqslant C_{k}$ and hence in view of \eqref{divMinq} one furnishes
		%\begin{equation*}
		%	\|\dvr(\xi(\phi_{k})\nabla M)\|_{L^{2}(\Omega)}\leqslant C_{k}.
		%\end{equation*}
		%Since $\phi_{k}\in W^{2,2}(\Omega)$ and $M\in W^{1,2}(\Omega)$, one concludes similarly to the corresponding proof for $M$, cf.\ arguments leading to \eqref{Mkp3}, that first $\|\Delta M\|_{L^\frac{3}{2}(\Omega)}\leqslant C_k$. Recalling that $\partial_{n}M\mid_{\partial\Omega}=0,$ and $M$ solves an elliptic problem we get $\|M\|_{W^{2,\frac{3}{2}}(\Omega)}\leqslant C_{k}$ that finally 
		One can now use elliptic regularity results and a bootstrap argument (exactly as the one used to prove the claim \eqref{claim}, cf.\ \cite[Section 4.2.1, p.~14--15]{KMS20}) to furnish that $\|M\|_{W^{2,2}(\Omega)}\leqslant C_k$ from \eqref{divMinq}. \\
		Next one obtains  $\|\nabla\mu\|_{L^{2}(\Omega)}\leqslant C_{k}$ from \eqref{energylambda**}. In view of Poincar\'{e}'s inequality it is sufficient to show 
		\begin{equation}\label{MuAvEst}
			\left|\int_\Omega\mu\right|\leqslant C_k,
		\end{equation}
		in order to prove $\|\mu\|_{W^{1,2}(\Omega)}\leqslant C_{k}.$ 
		For $\lambda\in[0,\frac{1}{2}),$ \eqref{MuAvEst} follows from the estimate of the last term which appears in the left hand side of \eqref{energylambda**}. For $\lambda\in[\frac{1}{2},1]$ we follow similar arguments used to show \eqref{boundsubdiff}. Hence \eqref{MuAvEst} holds independently of the values of $\lambda\in[0,1]$ and as a consequence $\|\mu\|_{W^{1,2}(\Omega)}\leqslant C_{k}$ follows.\\
		So far we have obtained the following
		\begin{equation}\label{estimatesofar}
			\begin{array}{l}
				\displaystyle\|v\|_{W^{1,2}(\Omega)}+\|M\|_{W^{2,2}(\Omega)}+\|\phi\|_{W^{1,2}(\Omega)}+\|\mu\|_{W^{1,2}(\Omega)}\leqslant C_{k}.
			\end{array}
		\end{equation}
		Further one recalls that $\mu$ solves the following equation
		\begin{equation}\label{eqmu}
			\begin{alignedat}{2}
				\displaystyle-\Delta\mu+\int_\Omega{\mu}&=-\lambda\frac{\phi-\phi_{k}}{h}-\lambda(v\cdot\nabla)\phi_{k}+\lambda\int_\Omega\mu&&\quad\mbox{in}\quad\Omega,\\
				 \partial_{n}\mu&=0 &&\quad\mbox{in}\quad\partial\Omega.
			\end{alignedat}
		\end{equation}
		In view of the fact that $\lambda\in[0,1]$ and the estimate \eqref{estimatesofar} we observe that the right hand side of \eqref{eqmu}$_1$ can be estimated in $L^{2}(\Omega)$ and hence by standard elliptic regularity theory
		$$\|\mu\|_{W^{2,2}(\Omega)}\leqslant C_{k}.$$
		Further, from the identity
		\begin{equation}\label{pwiseidentity}
			\begin{array}{l}
				\displaystyle\partial\widetilde{E}(\phi)+\phi=\lambda\phi+\lambda\mu+\lambda\kappa\frac{\phi+\phi_{k}}{2}-\lambda H_{0}(\phi,\phi_{k})\frac{|\nabla M|^{2}}{2}-\lambda\frac{H_{0}(\phi,\phi_{k})}{4\alpha^{2}}\left(|M|^{2}-1\right)^{2}
			\end{array}
		\end{equation} 
		and \eqref{estimatesofar} one has
		\begin{equation}\label{estdtE}
			\begin{array}{l}
				\|\partial\widetilde{E}(\phi)+\phi\|_{L^{2}(\Omega)}\leqslant C_{k}.
			\end{array}
		\end{equation}
		Inequality \eqref{estdtE} along with the estimate of $\phi$ from \eqref{estimatesofar} imply that $\|\partial\widetilde{E}(\phi)\|_{L^{2}(\Omega)}\leqslant C_{k}.$ Next in view of the inequality \eqref{regularityphi} one in particular concludes that
		%From the continuity of the map $(\partial\widetilde{E}+I)^{-1}:L^{2}(\Omega)\longrightarrow W^{\frac{15}{8},2}_n(\Omega)$ (we refer to Lemma~\ref{invertmonotone}) one has:
		$$\|\phi\|_{W^{\frac{15}{8},2}(\Omega)}\leqslant C_{k}$$ and hence altogether we have
		\begin{equation}\label{twtX}
			\begin{array}{l}
				\|w\|_{\widetilde{X}}+\|\partial\widetilde{E}(\phi)\|_{L^{2}(\Omega)}=\|(v,M,\phi,\mu)\|_{\widetilde X}+\|\partial\widetilde{E}(\phi)\|_{L^{2}(\Omega)}\leqslant C_{k}.
			\end{array}
		\end{equation}
		Once again using \eqref{regularityphi} and \eqref{twtX} one at once concludes \eqref{boundw} and consequently proves \eqref{lerayfxpt}. Finally the fact that $\mathcal{N}_{k}:X\rightarrow Y$ has a bounded inverse yields the existence of a fixed point to the operator equation \eqref{fxdpntmap}. This finishes the proof of Theorem~\ref{existenceweaksoldiscrete}. 
	\end{proof}
	
%%%%%%%%%%%%%%%%%%%%%%%%%%%%%%%%%%%%%%%%%%%%%%%%%%%%%%%%%%%%%%%%%%%%%%%%%%%%%%%%%%%%%%%%%%%%%%%%%%%%%%%%%%%%%%%%%%%%%%%%%%%%%%%%%%%%%%%%%%%%%%%%%%%%%%%%%%%%%%%%%%%%%%%%%%%%%%%%%%%%%%%%%%%%%%%%%%%%%%%%%%%%%%%%%%%%%%%%%%%%%%%%%%%%%%%%%%%%%%%%%%%%%%%%%%%%%%%%%%%%%%%%%%%%%%%%%%%%%%%%%%%%%%%%%%%%%%%%%%%%%%%%%%%%%%%%%%%%%%%%%%%%%%%%%%%%%%%%%%%%%%%%%%%%%%%%%%%%%
	\section{Proof of Theorem~\ref{Thm:Main}}\label{Thmmain}	
	Let $T>0$ be fixed. Let $0=t_0<t_1<\ldots <t_k<\ldots$, $k\in\eN_0$ be a strictly increasing sequence such that for each $k\in\eN_0$ $t_{k+1}-t_k=h$ where $h=\frac{1}{N}$ for $N\in\eN$ fixed. Applying Theorem~\ref{existenceweaksoldiscrete} successively, we construct a sequence $\{(v_{k+1},M_{k+1},\phi_{k+1},\mu_{k+1})\}$, $k\in\eN_0$ of solutions to problem \eqref{timediscretesystem} assuming $(v_k,M_k,\phi_k)\in L^2_{\dvr}(\Omega)\times W^{2,2}_n(\Omega)\times W^{2,2}_n(\Omega)$ with $-1\leqslant\phi_{k}\leqslant 1.$ Obviously, the assumption $(M_0,\phi_0)\in W^{1,2}(\Omega)\times W^{1,2}(\Omega)$ excludes the possibility of application of Theorem~\ref{existenceweaksoldiscrete} directly. Instead,
	% employing the density of $W^{2,2}_n(\Omega)$ in $W^{1,2}(\Omega)$, cf.\ \cite[Remark 1.2 (iii)]{Dro02} for the case of a domain with $C^{1,1}$ boundary. Accordingly, 
	we consider sequences $\{M_0^N\}\subset W^{2,2}_n(\Omega)$, $\{\phi^N_0\}\subset W^{2,2}_n(\Omega)$ with $|\phi^{N}_{0}|\leqslant 1$ such that
	\begin{equation}\label{InitMPhiApp}
		\begin{alignedat}{2}
			M^N_0&\to M_0&&\text{ in }W^{1,2}(\Omega),\\
			\phi^N_0&\to \phi_0&&\text{ in }W^{1,2}(\Omega)
		\end{alignedat}
	\end{equation}
	as $N\to\infty$. Such an approximating sequence $\{\phi^{N}_{0}\}$ can be constructed by solving a heat equation with initial data $\phi_{0},$ setting $\phi^{N}_{0}$ as the solution to the heat equation at $t=\frac{1}{N}$ and using parabolic regularity. The details of this construction can be found in \cite[Section 5.1, p.~471]{AbDeGa113}. Similar arguments apply in constructing $\{M^{N}_{0}\}.$ Adapting the notation from \cite{KMS20} we introduce two types of interpolants related to the unknowns. At first we fix $N\in\eN$. The piecewise constant interpolants of $(v,M,\phi)$ are defined on $[-h,\infty)$ and the one of $\mu$ on $[0,\infty)$ via
	\begin{equation}\label{definterpolinitial}
		v^N(t)=v_0,\ M^N(t)=M^N_0,\ \phi^N(t)=\phi^N_0\text{ for }t\in[-h,0)
	\end{equation}
	and 
	\begin{equation}\label{definterpol}
		f^N(t)=f_k\text{ for }t\in[(k-1)h,kh),
	\end{equation}
	where $f^N$ stands for the interpolants $v^N$, $M^N$, $\phi^N$, $\mu^N$ and $f_k$ represents the corresponding $v_k$, $M_k$, $\phi_k$ and $\mu_k$, $k\in \eN$. 
	We note that
	\begin{equation}\label{rhoN}
	\begin{array}{ll}
	 \rho^N=\frac{1}{2}(\widetilde\rho_1+\widetilde\rho_2)+\frac{1}{2}(\widetilde\rho_2-\widetilde\rho_1)\phi^N.
	 \end{array}
	\end{equation}
	Next, a piecewise affine interpolant $\widetilde f^N$ is defined as
	\begin{equation}\label{AffInterpolantDef}
		\widetilde f^N(t)=\frac{(k+1)h-t}{h}f^N(t-h)+\frac{t-kh}{h}f^N(t)\text{ for }t\in[kh,(k+1)h),\,\,k\in\mathbb{N}_{0}.
	\end{equation}
	For our purposes it is sufficient to consider only the piecewise affine interpolants $\{\widetilde {\rho v}^N\}$, $\{\widetilde M^N\}$ and $\{\widetilde \phi^N\}$, where the convention $(\rho v)^{N}=\rho^{N}v^{N}$ is used. We also introduce the notation for the shift and the difference quotient in time of a function $f$ as follows
	\begin{align*}
		f_h(t)=&f(t-h),\\
		\partial^{-}_{t,h}f(t)=&\frac{1}{h}(f-f_h)(t).
	\end{align*}
	As immediate consequences of the latter definition and \eqref{AffInterpolantDef} one gets
	\begin{equation}\label{AffInterpProp}
		\begin{split}
			\tder\widetilde f^N(t)&=\partial^-_{t,h}f^N(t)\,\,\mbox{for all}\,t\in [kh,(k+1)h),\,k\in\mathbb{N}_{0},\\
			\|\widetilde f^N\|_{L^p(0,\tau;X)}&\leqslant \|f^N\|_{L^p(0,\tau;X)}+\|f^N_h\|_{L^p(0,\tau;X)}\text{ for any }p\in[1,\infty]\,\,\mbox{and}\,\,\tau>0.
		\end{split}
	\end{equation}
	We will next state identities that are satisfied by interpolants. Let $\tau\in(0,\infty)$ is chosen arbitrarily. We find $k_\tau\in\eN_0$ such that $\tau\in [k_\tau h,(k_{\tau}+1)h)$. Further, we fix an arbitrary $\psi_1\in L^2(0,\infty;V(\Omega))$, set $\widetilde\psi_1=\int_{kh}^b\psi_1$ in \eqref{identity1}, where 
	\begin{equation*}
		b= \begin{cases}
			(k+1)h & k<k_\tau,\\
			\tau & k=k_\tau,
		\end{cases}
	\end{equation*}
	and sum the resulting expressions over $k\in\{0,1,\dots, k_\tau\}$ to obtain 	
	\begin{equation}\label{MomEqInterp}
		\begin{split}
			\int_0^\tau&\left(\int_{\Omega} \partial^-_{t,h}(\rho^Nv^N)\cdot\psi_{1}-\int_{\Omega}(\rho^N_hv^N\otimes v^N)\cdot\nabla{\psi}_{1}-\int_{\Omega}(v^N\otimes J^N)\cdot \nabla\psi_{1}\right.\\
			&\quad\left.-\int_{\Omega}\left(\frac{{\xi(\phi^N_h)}}{\alpha^{2}}(|M^N|^{2}M^N-M^N_h)\nabla M^N\right)\cdot\psi_{1}+\int_{\Omega}\left(\dvr\left(\xi(\phi^N_h)\nabla M^N\right)\nabla M^N\right)\cdot\psi_{1}\right)\\
			=&\int_0^\tau\left(-2\int_{\Omega}\nu(\phi^{N}_{h})\mathbb{D}v^{N}\cdot\mathbb{D}\psi_{1}-\int_{\Omega}\nabla\mu^N\phi^N_h\cdot\psi_{1}\right)
		\end{split}
	\end{equation}
	for all $\tau\in (0,\infty)$ and $\psi_1\in L^2(0,\infty;V(\Omega)),$ where $\displaystyle J^{N}=-\frac{\widetilde{\rho}_{2}-\widetilde{\rho}_{1}}{2}\nabla\mu^{N}$. In obtaining \eqref{MomEqInterp} from \eqref{identity1}, we once again use identity \eqref{JKIdent}. 
	Similarly, we get
	\begin{equation}\label{intidentity2}
		\begin{split}
			&\int_{0}^{\tau}\left(\int_{\Omega}\partial^{-}_{t,h}M^{N}\cdot\psi_{2}+\int_{\Omega}(v^{N}\cdot\nabla)M^{N}\cdot\psi_{2}\right)\\
			& =\int_{0}^{\tau}\int_{\Omega}\left(\dvr({\xi(\phi^{N}_{h})}\nabla
			M^{N})-\frac{\xi(\phi^{N}_{h})}{\alpha^{2}}(|M^{N}|^{2}M^{N}-M^{N}_{h})\right)\cdot\psi_{2},
		\end{split}
	\end{equation}
	for all $\tau\in(0,\infty)$ and $\psi_{2}\in L^2(0,\infty;W^{1,2}(\Omega))$.
	Moreover, by obvious manipulations we deduce from \eqref{identity3} that
	\begin{equation}\label{intidentity3}
		\int_{0}^{\tau}\left(\int_{\Omega} \partial^{-}_{t,h}\phi^{N}\psi_{3}+\int_{\Omega}(v^{N}\cdot\nabla)\phi^{N}_{h}\psi_{3}\right)=-\int_{0}^{\tau}\int_{\Omega}\nabla\mu^{N}\cdot\nabla\psi_{3}
	\end{equation}
	for all $\tau\in(0,\infty)$ and $\psi_3\in L^2(0,\infty;W^{1,2}(\Omega))$ and from \eqref{identity4} it follows that
	\begin{equation}\label{EqMuInterp}
		\begin{split}
			&\int_0^\tau\int_\Omega\left(\mu^N+\kappa\frac{\phi^N+\phi^N_h}{2}-H_{0}(\phi^N,\phi^N_h)\frac{|\nabla M^N|^{2}}{2}-\frac{H_{0}(\phi^N,\phi^N_h)}{4\alpha^{2}}(|M^N|^{2}-1)^{2}\right)\psi_4\\
			&=\int_0^\tau\int_\Omega\left(-\eta\Delta\phi^N+\widetilde{\Psi}'_{0}(\phi^N)\right)\psi_4
		\end{split}
	\end{equation}
	for all $\tau\in(0,\infty)$ and $\psi_4\in L^\infty(0,\tau;L^\infty(\Omega))$.
	\subsection{Compactness of sequences of interpolants}
	The goal of this section is to collect all the necessary convergences of (sub)sequences of interpolants allowing for the passage $h\to 0$ (equivalently $N\rightarrow\infty$) in order to show the existence of a weak solution to the original problem. 
	
	\subsubsection{Uniform bounds on sequences of interpolants and compactness}\label{Ubic}
	In order to obtain the uniform bounds we begin with the energy inequality for the interpolants $v^N$, $M^N$, $\phi^N$ and $\mu^N$. Summing in \eqref{discreteestimate} over $k\in\eN_0$ we obtain
	\begin{equation}\label{EnIneqInterp}
		\begin{split}
			E_{tot}(v^N(t),M^N(t),\phi^N(t))&+2\int_0^t\int_\Omega\nu(\phi^{N}_{h})|\mathbb{D} v^{N}|^{2}+\int_0^t\int_\Omega|\nabla\mu^N|^2\\
			&+\int_0^t\int_\Omega \left|\dvr(\xi(\phi^N_h)\nabla M^N)-\frac{\xi(\phi^N_h)}{\alpha^2}(|M^N|^2M^N-M^N_h)\right|^2\\
			&\leqslant E_{tot}(v_0,M^N_0,\phi^N_0)\leqslant C\left(E_{tot}(v_{0},M_{0},\phi_{0})+1\right)
		\end{split}
	\end{equation}
	first for each $t\in h\eN_0,$ where the inequality in the final line of \eqref{EnIneqInterp} follows from \eqref{InitMPhiApp}. As all interpolants involved in the latter inequality are constant on intervals of the form $[kh,(k+1)h),$ one concludes that \eqref{EnIneqInterp} holds for all $t\in[0,\infty)$. Therefore recalling the definition of $E_{tot}$ in \eqref{defEtot} and using \eqref{BoundsOnRho} we conclude from \eqref{EnIneqInterp} that in particular
	\begin{equation}\label{CInterpolantsBound}
		\begin{array}{ll}
			&\{v^N\}\text{ is bounded in } L^\infty(0,T+1;L^2(\Omega))\cap L^2(0,T+1;W^{1,2}(\Omega)),\\
			&\{v^N\}\text{ is bounded in }L^\frac{10}{3}(Q_{T+1})	,\\
			&\{M^N\}\text{ is bounded in } L^\infty(0,T+1;W^{1,2}(\Omega)),\\
			&\{\phi^N\}\text{ is bounded in } L^\infty(0,T+1;W^{1,2}(\Omega)),\\
			&\{\nabla \mu^N\}\text{ is bounded in } L^2(Q_{T+1}),\\
			&\{\dvr(\xi(\phi^N_h)\nabla M^N)-\frac{\xi(\phi^N_h)}{\alpha^2}(|M^N|^2M^N-M^N_h)\}\text{ is bounded in }L^2(0,T+1;L^2(\Omega)),\\[2.mm]
			&|\phi^{N}|\leqslant 1\,\,\text{a.e.\  in}\,\, Q_{T}.
		\end{array}
	\end{equation}
	The bound \eqref{CInterpolantsBound}$_{1}$ is obtained by using \eqref{nuAssum} and Korn inequality. All the sequences in \eqref{CInterpolantsBound}$_{1}$-\eqref{CInterpolantsBound}$_{6}$ in the respective norms are bounded by $C(E_{tot}(v_{0},M_{0},\phi_{0})+1)^{\frac{1}{2}}$ as a consequence of \eqref{EnIneqInterp}. Since $\phi_{k+1}\in\mathcal{D}(\partial\widetilde E)$ for all $k\in\mathbb{N}_{0}$, we have $|\phi_{k+1}|\leqslant 1$; one uses the fact that $|\phi^{N}_{0}|\leqslant 1$ and the definition of the interpolants \eqref{definterpolinitial}--\eqref{definterpol} to conclude \eqref{CInterpolantsBound}$_{7}.$ Let us note that \eqref{CInterpolantsBound}$_2$ is a consequence of bounds \eqref{CInterpolantsBound}$_{1},$ the embedding $W^{1,2}(\Omega)\hookrightarrow L^{6}(\Omega)$ and the following interpolation
	$$[L^{\infty}(0,T+1;L^{2}(\Omega)),L^{2}(0,T+1;L^{6}(\Omega))]_{\theta=\frac{3}{5}}=L^{\frac{10}{3}}(Q_{T+1}).$$
	The boundedness \eqref{CInterpolantsBound}$_3$ follows by \eqref{EnIneqInterp}, assumption \eqref{XiAssum} and the bound of $\left\{\frac{\xi(\phi^N_h)}{\alpha^2}(|M^N|^2-1)^2\right\}$ in $L^\infty(0,T+1;L^1(\Omega))$.
	Applying \eqref{boundsubdiff} we have
	\begin{equation*}
		\int_0^{T+1}\left|\int_\Omega\mu^{N}\right|\leqslant G(T+1)
	\end{equation*}
	for a monotone function $G:(0,\infty)\to(0,\infty)$. Combining the latter bound with \eqref{CInterpolantsBound}$_5$ we get
	\begin{equation}\label{MuCInterpBound}
		\{\mu^N\}\text{ is bounded in }L^2(0,T+1;W^{1,2}(\Omega)).
	\end{equation}
	Moreover, by the definition of a time shifted function we get
	\begin{equation}\label{TimeShiftBounds}
		\begin{alignedat}{2}
			\{v^N_h\}&\text{ is bounded in }&&L^\infty(0,T+1;L^2(\Omega)),\\
			\{M^N_h\}&\text{ is bounded in }&&L^\infty(0,T+1;W^{1,2}(\Omega)),\\
			\{\phi^N_h\}&\text{ is bounded in }&&L^\infty(0,T+1;W^{1,2}(\Omega)).
		\end{alignedat}
	\end{equation}
    All the sequences in \eqref{TimeShiftBounds} in the respective norms are bounded by $C(E_{tot}(v_{0},M_{0},\phi_{0})+1)^{\frac{1}{2}}.$ We conclude directly from the definition of the interpolants that
	\begin{equation}\label{PhiInterpBound}
		\{\phi^N\},\{\phi^N_h\},\{\widetilde\phi^N\}\subset [-1,1].
	\end{equation}
	Taking into account the definition of $\rho^N$ (as defined in \eqref{rhoN}) and $\rho^N_h$ it follows that
	\begin{equation}\label{RhoNBound}
		\{\rho^N\},\{\rho^N_h\}\text{ are bounded in }L^\infty(0,T+1;W^{1,2}(\Omega))\text{ and }L^\infty(Q_{T+1}).
	\end{equation}
	As a consequence of bounds \eqref{CInterpolantsBound}$_{1,2,3,4}$ and \eqref{MuCInterpBound} one has up to subsequences that are not explicitly relabeled
	\begin{equation}\label{InterpWeakCon}
		\begin{alignedat}{2}
			v^N&\rightharpoonup v&&\text{ in }L^2(0,T;W^{1,2}(\Omega)),\\
			v^N&\rightharpoonup^* v&&\text{ in }L^\infty(0,T;L^2(\Omega)),\\
			v^N&\rightharpoonup v&&\text{ in }L^\frac{10}{3}(Q_T),\\
			M^N&\rightharpoonup^* M&&\text{ in }L^\infty(0,T;W^{1,2}(\Omega)),\\
			\phi^N&\rightharpoonup^* \phi&&\text{ in }L^\infty(0,T;W^{1,2}(\Omega)),\\
			\mu^N&\rightharpoonup \mu&&\text{ in }L^2(0,T;W^{1,2}(\Omega)).
		\end{alignedat}
	\end{equation}
	Next we will collect some strong convergence results of the interpolants. Some of these results are already proved in \cite{KMS20}. First following the arguments from \cite[$(5.31)$, $(5.32)$ and $(5.33)$]{KMS20} we obtain
	\begin{equation}\label{PhiInterpStrongComp}
		\phi^N,\phi^N_h,\widetilde\phi^N\to \phi\text{ in }L^2(0,T;L^4(\Omega)).
	\end{equation}
	By \cite[$(5.39)$ and $(5.40)$]{KMS20} we have
	\begin{equation}\label{MInterpStrongComp}
		M^N\to M\text{ in }L^8(0,T;L^4(\Omega)),\ M^N_h\to M\text{ in }L^2(Q_T).
	\end{equation}
	Moreover, the convergence 
	\begin{equation}\label{SeqDvrConv}
		\dvr(\xi(\phi^N_h)\nabla M^N)-\frac{\xi(\phi^N_h)}{\alpha^2}(|M^N|^2M^N-M^N_h)\rightharpoonup \dvr(\xi(\phi)\nabla M)-\frac{\xi(\phi)}{\alpha^2}(|M|^2M-M)\text{ in }L^2(Q_T)
	\end{equation}
	follows by \cite[$(5.58)$]{KMS20}.
	Combining \eqref{PhiInterpBound} with \eqref{PhiInterpStrongComp} we get up to a nonrelabeled subsequence
	\begin{equation}\label{PhiInterpConv}
		\phi^N\to\phi\text{ in }L^p(Q_T) \text{ for all }p\in[1,\infty)\text{ and a.e.\  in }Q_T.
	\end{equation}
	To prove this claim one uses the strong convergence of $\phi^{N}$ from \eqref{PhiInterpStrongComp}, boundedness from \eqref{PhiInterpBound} and the following inequalities 
	\begin{equation*}
	\begin{array}{ll}
	\displaystyle\|\phi^{N}-\phi\|^{p}_{L^{p}(Q_{T})}&\displaystyle\leqslant 
	 C \int_{0}^{T}\|\phi^{N}-\phi\|^{(1-\frac{4}{p})p}_{L^{\infty}(\Omega)}\|\phi^{N}-\phi\|^{{4}}_{L^{4}(\Omega)}\\
	 &\displaystyle\leqslant C \int_{0}^{T}\|\phi^{N}-\phi\|^{(1-\frac{4}{p})p}_{L^{\infty}(\Omega)}\|\phi^{N}-\phi\|^{{2}}_{L^{\infty}(\Omega)}\|\phi^{N}-\phi\|^{{2}}_{L^{4}(\Omega)}\\
	 &\displaystyle \leqslant C\|\phi^{N}-\phi\|^{2+(1-\frac{4}{p})p}_{L^{\infty}(Q_{T})}\|\phi^{N}-\phi\|^{2}_{L^{2}(0,T;L^{4}(\Omega))}
	\end{array}
	\end{equation*}
	for $p\in[4,\infty)$ and  $\|\cdot\|_{L^{p}(\Omega)}\leqslant C\|\cdot\|_{L^{4}(\Omega)}$ for any $p\in[1,4).$
	 Taking into account the definition of $\rho^N$ (we refer to \eqref{rhoN}) and $\rho^N_h$ we obtain
	\begin{equation}\label{RhoInterpConv}
		\rho^N,\rho^N_h\to\rho\text{ in }L^p(Q_T) \text{ for all }p\in[1,\infty)\text{ and a.e.\  in }Q_T.
	\end{equation}
	We now focus on the proof of the compactness of interpolants for the velocity with respect to the topology of a suitable function space. This proof is not straightforward as no uniform bound is available on a sequence of time derivatives of piecewise affine  interpolants for the velocity. We investigate the convergence of $\{\widetilde{\rho v}^N\}$. From \eqref{MomEqInterp} it follows in particular that
	\begin{equation}\label{reformwk}
	\begin{split}
	\int_0^T& \langle \partial^-_{t,h}(\rho^Nv^N),\psi_1\rangle \\
	=&\int_0^T\int_\Omega\left( \rho^N_hv^N\otimes v^N+v^N\otimes J^N\right)\cdot\nabla\psi_1-2\nu(\phi^{N}_{h})\mathbb{D}v^{N}\cdot\mathbb{D}\psi_1\\
	&\qquad\quad+\left(\left(\frac{\xi(\phi^N_h)}{\alpha^2}(|M^N|^2M^N-M^N_h)-\dvr(\xi(\phi^N_h)\nabla M^N)\right)\nabla M^N-\nabla \mu^N\phi^N_h\right)\cdot\psi_1
	\end{split}
	\end{equation}
	for all $\psi_1\in L^8(0,\infty;V(\Omega))$. We note that $\left\{\rho^N_h v^N\otimes v^N\right\}$ is bounded in $L^2(0,T;L^{\frac{3}{2}}(\Omega))$ and $\{v^N\otimes J^N\}$ is bounded in $L^\frac{8}{7}(0,T;L^\frac{4}{3}(\Omega)).$ The explanation of achieving these bounds can be found in \cite[p.~474]{AbDeGa113}. Further one easily shows that $\{\nabla\mu^N\phi^N_h\}$ is bounded in $L^2(Q_{T})$ by using that $\{\phi^{N}_{h}\}$ is bounded in $L^{\infty}(Q_{T})$ and $\{\mu^{N}\}$ is bounded in $L^{2}(0,T;W^{1,2}(\Omega)).$ Using \eqref{nuAssum} and \eqref{CInterpolantsBound}$_{1},$  $\nu(\phi^{N}_{h})\mathbb{D}v^{N}$ is bounded in $L^{2}(Q_{T}).$ Moreover, by \eqref{CInterpolantsBound}$_{3,6}$ we have the bound on\\ $\left\{\left(\frac{\xi(\phi^N_h)}{\alpha^2}(|M^N|^2M^N-M^N_h)-\dvr(\xi(\phi^N_h)\nabla M^N)\right)\nabla M^N\right\}$ in $L^2(0,T;L^1(\Omega))$. Since $\partial_{t}\widetilde{\rho v}^{N}=\partial^{-}_{t,h}(\rho^{N}v^{N})$, one uses \eqref{reformwk} and the fact that the Leray projector $\mathbb{P}_{\dvr}$  commutes with the time derivative to infer that
	\begin{equation}\label{DiffQRVNBound}
	\left\{\tder\mathbb{P}_{\dvr}\left(\widetilde{\rho v}^N\right)\right\}\text{ is bounded in }L^\frac{8}{7}(0,T;(V(\Omega))').
	\end{equation}
	Taking into account the uniform bound on $\left\{\mathbb{P}_{\dvr}\left(\widetilde{\rho v}^N\right)\right\}$ in $L^\infty(0,T;L^2(\Omega))$, which follows from the continuity of $\mathbb{P}_{\dvr}$, \eqref{CInterpolantsBound}$_{1}$ and \eqref{RhoNBound}, the Aubin-Lions lemma gives the compactness of $\left\{\mathbb{P}_{\dvr}\left(\widetilde{\rho v}^N\right)\right\}$ with respect to the strong topology of $L^2(0,T;(W^{1,2}_{0,\dvr}(\Omega))')$. Moreover, in view of \eqref{InterpWeakCon}$_{1}$ and the almost everywhere convergence \eqref{RhoInterpConv} we can follow line by line the arguments presented in \cite[p.~90-91, $(3.95)$]{WeberThesis} to conclude that
	\begin{equation}\label{WTRhoVWeakly}
	\mathbb{P}_{\dvr}\left(\widetilde{\rho v}^N\right)\rightharpoonup \mathbb{P}_{\dvr}(\rho v)\text{ in }L^2(0,T;L^2_{\dvr}(\Omega)).
	\end{equation}
	Consequently, for a nonrelabeled subsequence we obtain
	\begin{equation}\label{WTRhoVStrongly}
	\mathbb{P}_{\dvr}\left(\widetilde{\rho v}^N\right)\to \mathbb{P}_{\dvr}(\rho v)\text{ in }L^2(0,T;(W^{1,2}_{0,\dvr}(\Omega))').
	\end{equation}
	Since
	\begin{equation*}
	\mathbb{P}_{\dvr}\left(\widetilde{\rho v}^N(t)\right)-\mathbb{P}_{\dvr}\left(\rho^N(t)v^N(t)\right)=(t-(k+1)h)\partial_{t}\mathbb{P}_{\dvr}\left(\widetilde{\rho v}^{N}\right)(t)\text{ for all } t\in[kh,(k+1)h), k\in\eN_0
	\end{equation*}
	and $|t-(k+1)h|\leqslant h$ we infer using  \eqref{DiffQRVNBound} that
	\begin{equation}\label{trvconv}
	\begin{array}{l}
	\mathbb{P}_{\dvr}\left(\widetilde{\rho v}^N\right)-\mathbb{P}_{\dvr}\left(\rho^Nv^N\right)\to 0\text{ in } L^\frac{8}{7}(0,T;(V(\Omega))').
	\end{array}
	\end{equation}
	Next we consider the following combination of interpolation and duality
	\begin{equation*}
	    (W^{1,2}_{0,\dvr}(\Omega))'=\left((L^2_{\dvr}(\Omega), V(\Omega))_{\frac{1}{2},2}\right)' =\left(L^2_{\dvr}(\Omega), (V(\Omega))'\right)_{\frac{1}{2},2},
	\end{equation*}
	where the first equality is a special case of \cite[(5.2.17)]{Abels2007} and the second one follows by \cite[Theorem 3.7.1]{BerLof76}. Employing the inequality that corresponds to the latter interpolation we obtain
	\begin{equation*}
	\begin{array}{ll}
	&\displaystyle \left\|\mathbb{P}_{\dvr}(\widetilde{\rho v}^{N})-\mathbb{P}_{\dvr}(\rho^{N}v^{N})\right\|_{L^{2}(0,T;(W^{1,2}_{0,\dvr}(\Omega))')}\\
	&\displaystyle\leqslant C\left\|\mathbb{P}_{\dvr}(\widetilde{\rho v}^{N})-\mathbb{P}_{\dvr}(\rho^{N}v^{N})\right\|_{L^{\infty}(0,T;L^{2}_{\dvr}(\Omega))}^{\frac{1}{2}}\left\|\mathbb{P}_{\dvr}(\widetilde{\rho v}^{N})-\mathbb{P}_{\dvr}(\rho^{N}v^{N})\right\|_{L^{1}(0,T;(V(\Omega))')}^{\frac{1}{2}}.
	\end{array}
	\end{equation*}
	Combining the latter inequality with the bound on $\left(\mathbb{P}_{\dvr}(\widetilde{\rho v}^{N})-\mathbb{P}_{\dvr}(\rho^{N}v^{N})\right)$ in $L^{\infty}(0,T;L^{2}(\Omega))$ 
	and the convergence \eqref{trvconv} we furnish that 
	\begin{equation}\label{trvconv*}
	\begin{array}{l}
	\mathbb{P}_{\dvr}(\widetilde{\rho v}^N)-\mathbb{P}_{\dvr}(\rho^Nv^N)\to 0\,\,\mbox{in}\,\, L^{2}(0,T;((W^{1,2}_{0,\dvr}(\Omega))').
	\end{array}
	\end{equation}
	The convergences \eqref{WTRhoVStrongly} and \eqref{trvconv*} together furnish that
	\begin{equation}\label{RhoVNStrongly}
	\mathbb{P}_{\dvr}(\rho^Nv^N)\to\mathbb{P}_{\dvr}(\rho v)\text{ in }L^2(0,T; (W^{1,2}_{0,\dvr}(\Omega))').
	\end{equation} 
	Next, by \eqref{RhoInterpConv} and \eqref{InterpWeakCon}$_3$ we conclude $(\rho^N)^\frac{1}{2}v^N\rightharpoonup \rho^\frac{1}{2}v$ in $L^2(Q_T)$. Furthermore, combining \eqref{RhoVNStrongly} and \eqref{InterpWeakCon}$_1$ it follows that 
	\begin{equation}\label{limpassagerv2}
	\begin{array}{ll}
	\displaystyle\int_0^T\int_\Omega \rho^N|v^N|^2&\displaystyle=\int_0^T\langle \mathbb{P}_{\dvr}(\rho^Nv^N),v^N\rangle_{(W^{1,2}_{0,\dvr}(\Omega))',W^{1,2}_{0,\dvr}(\Omega)}\\
	& \displaystyle\to\int_0^T\langle \mathbb{P}_{\dvr}(\rho v),v\rangle_{(W^{1,2}_{0,\dvr}(\Omega))',W^{1,2}_{0,\dvr}(\Omega)}
	\displaystyle=\int_0^T\int_\Omega \rho|v|^2
	\end{array}
	\end{equation}
	implying $\|(\rho^N)^\frac{1}{2}v^N\|_{L^2(Q_T)}\to \|\rho^\frac{1}{2}v\|_{L^2(Q_T)}$. Hence passing to a nonrelabeled subsequence one has
	\begin{equation}
	(\rho^N)^\frac{1}{2}v^N\to \rho^\frac{1}{2}v\text{ in }L^2(Q_T) \text{ and a.e.\  in }Q_T.
	\end{equation}
	Moreover, \eqref{RhoInterpConv} and the existence of a positive lower bound on $\{\rho^N\}$, obtained by a similar argument as in Remark~\ref{Rem:DensityBound}, imply $(\rho^N)^{-\frac{1}{2}}\to \rho^{-\frac{1}{2}}$ a.e.\  in $Q_T$. As \eqref{CInterpolantsBound}$_2$  ensures the equiintegrability of the sequence $\{|v^N|^q\}$, $q\in[1,\frac{10}{3})$ we conclude by the Vitali convergence theorem
	\begin{equation}\label{VNStrongly}
	v^N=(\rho^N)^{-\frac{1}{2}}(\rho^N)^\frac{1}{2}v^N\to v\text{ in }L^q(Q_T),\ q\in[1,\tfrac{10}{3}).
	\end{equation}
	In particular the strong convergence $v^{N}\rightarrow v$ in $L^{2}(Q_{T})$ (as a consequence of \eqref{VNStrongly}), the boundedness of $(v^{N}-v)$ in $L^{2}(0,T;L^{6}(\Omega))$ and the following interpolation inequality
	$$\displaystyle\|v^{N}-v\|_{L^{2}(0,T;L^{4}(\Omega))}\leqslant C\|v^{N}-v\|_{L^{2}(0,T;L^{6}(\Omega))}^{\frac{3}{4}}\|v^{N}-v\|_{L^{2}(Q_{T})}^{\frac{1}{4}}$$
	provides
	\begin{equation}\label{strngconvL24}
	\begin{array}{l}
	v^N\to v\,\,\text{ in }\,\,L^{2}(0,T;L^{4}(\Omega)).
	\end{array}
	\end{equation}
	The last important convergence is 
	\begin{equation}\label{MNStrongly}
	M^N\to M\text{ in }L^2(0,T;W^{1,2}(\Omega)).
	\end{equation}
	The convergence \eqref{MNStrongly} is crucial in order to pass to the limit in the term containing $|\nabla M^{N}|^{2}$ in \eqref{EqMuInterp} and the term comprising of $\nabla M^{N}$ in the momentum equation \eqref{MomEqInterp}. The proof of \eqref{MNStrongly} relies on the monotone structure of $\dvr(\xi(\phi^{N}_{h})\nabla M^{N})$ and can be done by following the arguments used to show \cite[$(5.41),$ Section 5.1.2]{KMS20}. In order to do so, the strong convergence of $M^{N}$ to $M$ in $L^{8}(0,T;L^{4}(\Omega))$ and $v^{N}$ to $v$ in $L^{2}(0,T;L^{4}(\Omega))$ are used. Since these convergences are available in the present scenario (we refer to \eqref{MInterpStrongComp} and \eqref{strngconvL24}), we face no particular difficulty to follow line by line the proof of \cite[$(5.41),$ Section 5.1.2]{KMS20}.
	\subsubsection{Some uniform estimates on $M^{N}$ and $\phi^{N}$}\label{phiNMNuest}
	In this section we will obtain further uniform estimates that involve the integrability of $\nabla M^N$ w.r.t. spatial variables for an exponent greater than $2$, the integrability of the second gradient of $\phi^{N}$ and of $\widetilde{\Psi}'_{0}(\phi^{N})$ w.r.t. spatial variables for an exponent greater than $1$ depending only on the energy estimate \eqref{EnIneqInterp} for the interpolants. These improved estimates will aid in recovering the weak formulation of Cahn-Hilliard equations. In that direction we will make use of an abstract elliptic regularity result from \cite{konrad}. The central result of this section is Lemma~\ref{UregMphi} which will be proved by using the following result.
	\begin{lem}\label{Lem:RegImpr}
		Let $\Omega$ be a bounded domain of class $C^{1}$ in $\mathbb{R}^{d},$ $d\geqslant 2$. Let $\widetilde{\xi}:\Omega\rightarrow \mathbb{R}^{+}$ be a bounded, measurable function satisfying
		\begin{equation}\label{nondegeneracy}
				0<c_1\leq\widetilde{\xi}(\cdot)\leqslant c_2\text{ on }\Omega, \text{ for some}\,\, c_{1},c_{2}>0.
		\end{equation}
		Then there is $p>2$ such that a solution $M=(M_{1},M_{2},M_{3})\in W^{1,2}(\Omega)$ of the following elliptic problem with homogeneous Neumann boundary condition
		\begin{equation}\label{Neumannprob}
			\begin{alignedat}{2}
				\dvr ({\widetilde{\xi}\nabla M})=&g&&\text{ in }\Omega,\\
				\partial_{n}M=&0&&\text{ on }\partial\Omega,
			\end{alignedat}
		\end{equation}
		where $g\in L^s(\Omega)$ with $s\geq \frac{2d}{d+2}$ satisfies
		\begin{equation}\label{W1pest}
			\begin{array}{l}
				\|M\|_{W^{1,p}(\Omega)}\leqslant C\left(\|g\|_{L^s(\Omega)}+\|M\|_{W^{1,2}(\Omega)}\right).
			\end{array}
		\end{equation}
		The constant $C$ depends on $d,$ $c_{1}$  and $c_{2}$ and the domain $\Omega.$
	\end{lem}
	\begin{proof}The result stated in Lemma~\ref{Lem:RegImpr} is a special case of the more general result in \cite[Remark 13]{konrad}. For the convenience of the readers we present the proof. First, we rewrite \eqref{Neumannprob} component wise in the form
		\begin{equation}\label{Neumannprob*}
			\begin{alignedat}{2}
				-\sum\limits_{i=1}^{d}\partial_{i}({\widetilde{\xi}}\partial_{i}M_{k})+M_{k}=&-g_{k}+M_{k}&&\text{ in }\Omega,\\
				\partial_{n}\,M_k=&0&&\text{ on }\partial\Omega,
			\end{alignedat}
		\end{equation}
		for $k\in\{1,2,3\}$. The operator $A:W^{1,2}(\Omega)\to(W^{1,2}(\Omega))^{'}$ appearing in the weak form of the latter problem is defined as 
		\begin{equation*}
			\langle Au,w\rangle=\int_\Omega a_{ij}\partial_i u\partial_j w+uw\text{ for }u,w\in W^{1,2}(\Omega)
		\end{equation*}
		with the matrix $(a_{ij})_{i,j\in\{1,..,d\}}$ given by
		\begin{equation}\nonumber
			a_{ij}(x)=\left\{ \begin{array}{ll}
				\widetilde{\xi}(x)&\quad\mbox{when}\quad i=j,\\
				0&\quad\mbox{when}\quad i\neq j.
			\end{array}\right.
		\end{equation}
		Taking into account \eqref{nondegeneracy} we infer $a_{ij}(\cdot)\in L^{\infty}(\Omega)$ and the ellipticity condition 
		$$\sum_{i,j=1}^{d}a_{ij}(x)\theta_{i}\theta_{j}\geqslant c_{1}|\theta|^{2}\text{ for any }\theta=(\theta_{1},\ldots,\theta_{d})\in\mathbb{R}^d.$$
		In order to apply \cite[Theorem 1]{konrad}, we note that as $\Omega$ is of class $C^1$, assumptions of \cite[Theorem 1]{konrad} are fulfilled, cf.\ \cite[Remark 1 and 7]{konrad}. Hence we conclude the existence of some $p>2$ such that the operator $A$ maps $W^{1,p}(\Omega)$ onto $(W^{1,p}(\Omega))^{'}$. The upper bound on $p$ can be found in \cite[Theorem 1]{konrad}, from which one also infers the inequality
		\begin{equation}\label{AIneq}
			\|A^{-1}f\|_{W^{1,p}(\Omega)}\leqslant c\|f\|_{(W^{1,p}(\Omega))^{'}}
		\end{equation}
		due to the linearity of $A^{-1}$. We notice that the constant $c$ depends on $\Omega$ and constants from \eqref{nondegeneracy}. In order to conclude \eqref{W1pest}, we employ \eqref{AIneq} with $f$ being the right hand side of \eqref{Neumannprob*}$_1$ and use the embedding  $L^s(\Omega)\hookrightarrow(W^{1,2}(\Omega))'$ for $s\geqslant \frac{2d}{(d+2)}$ and $(W^{1,2}(\Omega))^{'}$ to $(W^{1,p}(\Omega))^{'}$ for $p>2$. 
	\end{proof}
	\begin{lem}\label{UregMphi}
		Let $(v^{N},M^{N},\phi^{N},\mu^{N})$ be the interpolants defined by \eqref{definterpolinitial}--\eqref{definterpol}, satisfying \eqref{MomEqInterp}--\eqref{EqMuInterp} and the energy estimate \eqref{EnIneqInterp}. Then $M^{N}$ satisfies
		\begin{equation}\label{boundM}
			\begin{array}{l}
				\|M^{N}\|_{L^{2}(0,T;W^{1,p}(\Omega))}\leqslant CE^{+}_{tot}(v_0,M_0,\phi_0)^{\frac{3}{2}},
			\end{array}
		\end{equation}
		where $E^{+}_{tot}(v_0,M_0,\phi_0)=\left(E_{tot}(v_0,M_0,\phi_0)+1\right),$ for some $p>2$ and the positive constant $C$ might depend on $\alpha$, $c_{1},$ $c_{2},$ $c_{3}$ , cf.\ \eqref{XiAssum}, Sobolev embedding constants and the domain $\Omega.$\\
		Further $\phi^N$ and $\widetilde{\Psi}'_{0}(\phi^N)$ satisfy
		\begin{equation}\label{boundphiPsi}
				\|\phi^{N}\|_{L^{2}(0,T;W^{2,q}(\Omega))}+\|\widetilde{\Psi}'_{0}(\phi^{N})\|_{L^2(0,T;L^q(\Omega))}\leqslant CE^{+}_{tot}(v_0,M_0,\phi_0)^3.
		\end{equation}
			where $1<q=\frac{2p}{p+2}<2$ and the positive constant $C$ in \eqref{boundphiPsi} depends on $\alpha$, $c_{1},$ $c_{2},$ $c_{3},$ Sobolev embedding constants and the domain $\Omega.$
	\end{lem}
	\begin{proof}
		In order to prove \eqref{boundM} we consider \eqref{intidentity2} for a.e.\  $t\in(0,T)$ as the elliptic problem
		\begin{equation}\label{bounddivL2}
			\begin{alignedat}{2}
				\dvr(\xi(\phi^N_h)\nabla M^N)=&\partial^{-}_{t,h}M^N+(v^N\cdot\nabla)M^N+\frac{\xi(\phi^N_h)}{\alpha^2}(|M^N|^2M^N-M^N_h)&&\text{ in }\Omega,\\
				\partial_{n}M^{N}=&0&&\text{ on }\partial\Omega,
			\end{alignedat}
		\end{equation}
		which is possible since, in view of \eqref{CInterpolantsBound}, all the terms involved in \eqref{bounddivL2} are defined a.e. We focus on the estimate of the right hand side in \eqref{bounddivL2}$_1$. Using the H\"older inequality, the Sobolev embedding and \eqref{CInterpolantsBound} we get
		\begin{equation}\label{MConvTermEst}
			\begin{split}
				\|(v^N\cdot\nabla) M^N\|_{L^2(0,T;L^\frac{3}{2}(\Omega))}&\leqslant \|v^N\|_{L^2(0,T;L^6(\Omega))}\|\nabla M^N\|_{L^\infty(0,T;L^2(\Omega))}\\
				&\leqslant C\| v^N\|_{L^2(0,T;W^{1,2}(\Omega))}\|\nabla M^N\|_{L^\infty(0,T;L^2(\Omega))}
				\leqslant C E^{+}_{tot}(v_0,M_0,\phi_0).
			\end{split}
		\end{equation}
		Since $L^{2}(0,T;W^{1,2}(\Omega))$ is dense in $L^{2}(0,T;L^{3}(\Omega)),$ one can choose test functions $\psi_{2}\in L^{2}(0,T;L^{3}(\Omega))$ in \eqref{intidentity2} and use \eqref{CInterpolantsBound}$_{1,6}$ to compute the following
		%Using the latter estimate and \eqref{EnIneqInterp}, cf.\ \cite[the justification of inequality (5.36)]{KMS20} it follows from \eqref{intidentity2}
		%that 
		\begin{equation}\label{TDerMEst}
			\begin{array}{ll}
				\displaystyle\|\partial^{-}_{t,h}M^N\|_{L^2(0,T;L^\frac{3}{2}(\Omega))}
				&\displaystyle=\sup_{\{\psi_{2}\in L^{2}(0,T;L^{3}(\Omega))\suchthat \|\psi_{2}\|_{L^{2}(0,T;L^{3}(\Omega))}\leq 1\}}\int_{0}^{T}\int_{\Omega}\partial^{-}_{t,h}M^N\psi_{2}\\[4.mm]
				&\displaystyle\leq C\bigg(\| v^N\|_{L^2(0,T;W^{1,2}(\Omega))}\|\nabla M^N\|_{L^\infty(0,T;L^2(\Omega))}\\
				&\qquad +\left\|\dvr(\xi(\phi^N_h)\nabla M^N)-\frac{\xi(\phi^N_h)}{\alpha^2}(|M^N|^2M^N-M^N_h)\right\|_{L^{2}(Q_{T})}\bigg)\\[3.mm]
				&\displaystyle\leq C\left(E^{+}_{tot}(v_0,M_0,\phi_0)+E^{+}_{tot}(v_0,M_0,\phi_0)^\frac{1}{2}\right).
			\end{array}
		\end{equation}
		Moreover, using bounds on $\xi$ in \eqref{XiAssum}, the H\"older inequality, the Sobolev embedding and the definition of $M^N_h$ we arrive at
		\begin{equation}\label{RestEst}
			\begin{split}
				\left\|\frac{\xi(\phi^N_h)}{\alpha^2}(|M^N|^2M^N-M^N_h)\right\|_{L^2(0,T;L^\frac{3}{2}(\Omega))}\leq& C\left(\|M^N\|^3_{L^6(0,T;L^\frac{9}{2}(\Omega))}+\|M^N_h\|_{L^2(0,T;L^\frac{3}{2}(\Omega))}\right)\\
				\leq& C\left(\|M^N\|^3_{L^\infty(0,T;W^{1,2}(\Omega))}+\|M^N_h\|_{L^\infty(0,T;W^{1,2}(\Omega))}\right)\\
				\leq& C\left(E^{+}_{tot}(v_0,M_0,\phi_0)^\frac{3}{2}+E^{+}_{tot}(v_0,M_0,\phi_0)^\frac{1}{2}\right).
			\end{split}
		\end{equation}
		Applying Lemma~\ref{Lem:RegImpr} with $s=\frac{3}{2}$ to \eqref{bounddivL2} we obtain
		\begin{equation*}
		\begin{split}
			\|M^N\|_{W^{1,p}(\Omega)}\leqslant& C\Biggl(\|\partial^{-}_{t,h}M^N\|_{L^\frac{3}{2}(\Omega)}+\|(v^N\cdot\nabla) M^N\|_{L^\frac{3}{2}(\Omega)}+\left\|\frac{\xi(\phi^N_h)}{\alpha^2}(|M^N|^2M^N-M^N_h)\right\|_{L^\frac{3}{2}(\Omega)}\\
			&+\|M^N\|_{W^{1,2}(\Omega)}\Biggr)
			\end{split}
		\end{equation*}
		a.e.\  in $(0,T)$ with $p>2$ and the constant $C$ independent of the time variable. We combine the latter inequality with \eqref{MConvTermEst}, \eqref{TDerMEst}, \eqref{RestEst} and the Young inequality to conclude \eqref{boundM}.
		
		To show \eqref{boundphiPsi}, we will use \eqref{EqMuInterp}. Using \eqref{CInterpolantsBound}$_{4},$ \eqref{MuCInterpBound} and \eqref{TimeShiftBounds}$_{3}$ one can bound the first two terms appearing in the left hand side of \eqref{EqMuInterp} in $L^{2}(0,T;L^{6}(\Omega))$ by a constant multiple of $E_{tot}(v_0,M^N_0,\phi^N_0)^{\frac{1}{2}}.$ Further since $\|H_{0}(\phi^N,\phi^N_h)\|_{L^{\infty}(Q_{T})}\leqslant c_{3}$ and by \eqref{CInterpolantsBound}$_3$ along with \eqref{boundM}, one has
		\begin{equation}\label{someLpbndgM}
			\begin{split}
				\tfrac{1}{2}&\left\|H_{0}(\phi^N,\phi^N_h)|\nabla M^N|^{2}\right\|_{L^{2}(0,T;L^{\frac{2p}{p+2}}(\Omega))}\\
				&\leqslant C\|\nabla M^{N}\|_{L^{2}(0,T;L^{p}(\Omega))}\|\nabla M^{N}\|_{L^{\infty}(0,T;L^{2}(\Omega))}\leqslant CE^{+}_{tot}(v_0,M_0,\phi_0),
			\end{split}
		\end{equation}
		and
		\begin{equation}\label{someLPbndgM*}
			\begin{split}
				&\tfrac{1}{4\alpha^2}\left\|H_{0}(\phi^N,\phi^N_h)(|M^N|^{2}-1)^{2}\right\|_{L^{2}(0,T;L^{\frac{2p}{p+2}}(\Omega))} \leqslant C\left\|M^{N}\|^{3}_{L^{\infty}(0,T;L^{6}(\Omega))}\|M^{N}\|_{L^{2}(0,T;L^{p}(\Omega))}+1\right)\\
				\vspace{1em}
				&\leqslant C\left( E^{+}_{tot}(v_0,M_0,\phi_0)^{\frac{3}{2}}E^{+}_{tot}(v_0,M_0,\phi_0)^{\frac{3}{2}}+1\right)\leqslant CE^{+}_{tot}(v_0,M_0,\phi_0)^{3}.
			\end{split}
		\end{equation}
		In both of \eqref{someLpbndgM} and \eqref{someLPbndgM*} the positive constant $C$ might depend on $\alpha$, $c_{1},$ $c_{2},$ $c_{3},$ Sobolev embedding constants and $|\Omega|.$ 
		From the discussion above (in particular the inequalities \eqref{someLpbndgM} and \eqref{someLPbndgM*}) one infers from \eqref{EqMuInterp}
		\begin{equation}\label{eqphiN}
			\begin{alignedat}{2}
				-\eta\Delta\phi^N+\widetilde{\Psi}'_{0}(\phi^N)=&f(M^{N},\phi^{N},\phi^{N}_{h},\mu^{N})&&\mbox{ in }\Omega,\\
				\partial_{n}\phi^{N}=&0&&\mbox{ on }\partial\Omega
			\end{alignedat}
		\end{equation} 
		a.e.\  in $(0,T)$ with 
		\begin{equation*}
		\left\|f(M^{N},\phi^{N},\phi^{N}_{h},\mu^{N})\right\|_{L^{2}(0,T;L^{\frac{2p}{p+2}}(\Omega))}\leqslant CE^{+}_{tot}(v_0,M_0,\phi_0)^3,
		\end{equation*}
		where $C$ might depend on $\alpha$, $c_{1},$ $c_{2},$ $c_{3},$ Sobolev embedding constants and the domain $\Omega.$\\
		Finally, applying the inequality \eqref{regLp} from Proposition~\ref{ellipticreg} to  \eqref{eqphiN} and the Young inequality again we obtain \eqref{boundphiPsi}.
	\end{proof}
	\subsubsection{Additional convergences of $\{M^N\}$, $\{\phi^N\}$, $\{\Psi'(\phi^N)\}$}\label{Additionalconv}
    In view of estimate \eqref{boundM}, we immediately obtain that for some $p>2$ we have up to a nonrelabeled subsequence 
    \begin{equation*}
        M^N\rightharpoonup M\text{ in }L^2(0,T;W^{1,p}(\Omega)), 
    \end{equation*}
    where $M$ comes from \eqref{InterpWeakCon}. This concludes \eqref{AddReg}$_1$. Similarly, by \eqref{boundphiPsi} we have up to a nonrelabeled subsequence 
    \begin{equation}\label{PhiNW2Cnv}
        \phi^N\rightharpoonup \phi\text{ in }L^2(0,T;W^{2,\frac{2p}{p+2}}(\Omega)),
    \end{equation}
    proving \eqref{AddReg}$_2$. The next task is to show that up to a nonrelabeled subsequence
    \begin{equation}\label{PsiPhiNCnv}
        \widetilde{\Psi}'_0(\phi^N)\rightharpoonup \widetilde{\Psi}'_0(\phi)\text{ in }L^2(0,T;L^{\frac{2p}{p+2}}(\Omega)),
    \end{equation}  
    from which \eqref{AddReg}$_3$ follows. We observe that the estimate of $\widetilde{\Psi}'_0(\phi^N)$ in \eqref{boundphiPsi} implies that
    \begin{equation*}
        \widetilde{\Psi}'_0(\phi^N)\rightharpoonup \zeta \text{ in }L^2(0,T;L^{\frac{2p}{p+2}}(\Omega)).
    \end{equation*}   
    Hence we have to identify $\zeta$. Let us begin with showing that 
		\begin{equation}\label{PointConvComp}
		\widetilde\Psi'_0(\phi^N)\to\widetilde\Psi'_0(\phi)\text{ a.e.\  in }Q_T.
		\end{equation}
		To this end we adopt arguments devised in the context of the Cahn--Hillard equations with a logarithmic free energy, see \cite[p.~1510]{DebDet95}, and developed  for the case of the Navier--Stokes--Cahn--Hilliard system with a singular potential, see \cite[p.~285]{FriGra12}.
		We define for arbitrary but fixed $\delta\in(0,1)$ the quantity $a_\delta=\min\left\{\widetilde\Psi'_0(1-\delta),-\widetilde\Psi'_0(-1+\delta)\right\}$. Then we have
		$a_\delta\leqslant |\widetilde\Psi'_0(s)|$ for $1>|s|>1-\delta$ as $\widetilde\Psi'_0$ is nondecreasing. Hence we obtain
		\begin{equation*}
		a_\delta\left|\left\{(t,x)\in Q_T: 1>|\phi^N(t,x)|>1-\delta\right\}\right|\leq\int_{Q_T}|\widetilde\Psi'_0(\phi^N)|\leqslant c
		\end{equation*}
		by \eqref{boundphiPsi} and H\"older's inequality.
		Combining the pointwise convergence $\phi^N\to \phi$ from \eqref{PhiInterpConv} and the Fatou Lemma with the latter inequality we conclude
		\begin{equation}\label{DeltaIneq}
		|\{(t,x)\in Q_T:1\geq |\phi(t,x)|\geq 1-\delta\}|\leqslant \liminf_{N\to\infty}\left|\left\{(t,x)\in Q_T:1>|\phi^N(t,x)|>1-\delta\right\}\right|\leqslant ca_\delta^{-1}
		\end{equation}
		for any $\delta\in(0,1)$. Taking into account assumption \eqref{PsiReg}$_1$ it follows that $a_\delta\to\infty$ as $\delta\to 0_+$. Hence the limit passage $\delta\to 0_+$ in \eqref{DeltaIneq} yields 
		\begin{equation*}
		|\{(t,x)\in Q_T:|\phi(t,x)|=1\}|=0,
		\end{equation*}
		in other words $|\phi|<1$ a.e.\  in $Q_T$. This bound, the pointwise convergence $\phi^N\to\phi$ from \eqref{PhiInterpConv} and the assumed regularity $\widetilde\Psi'_0\in C^1((-1,1))$, cf.\ Assumption 1.1, imply \eqref{PointConvComp}. Having \eqref{PointConvComp} and the bound on $\{\widetilde\Psi'_0(\phi^N)\}$ from \eqref{boundphiPsi} at hand we apply the Vitali convergence theorem to conclude that $\widetilde\Psi'_0(\phi^N)\to \widetilde\Psi'_0(\phi)$ in $L^1(Q_T)$. Hence we have $\zeta=\widetilde\Psi'_0(\phi)$ and \eqref{PsiPhiNCnv} is proved.\\
		The convergence \eqref{PsiPhiNCnv} along with the fact that $\phi\in[-1,1]$ and \eqref{deftpsi}-\eqref{deftpsi0} in particular imply that $\Psi'(\phi)\in L^{2}(0,T;L^{\frac{2p}{p+2}}(\Omega)).$
	%%%%%%%%%%%%%%%%%%%%%%%%%%%%%%%%%%%%%%%%%%%%%%%%%%%%%%%%%%%%%%%%%%%%%%%%%%%%%%%%%%%%%%%%%%%%%%%%%%%%%%%%%%%%%%%%%%%%%%%%%%%%%%%%%%%%%%%%%%%%%%%%%%%%%%%%%%%%%%%%%%%%%%%%%%%%%%%%%%%%%%%%%%%%%%%%%%%%%%%%%%%%%%%%%%%%%%%%%%%%%%%%%%%%%%%%%%%%%%%%%%%%%%%%%%%%%%%%%%%%%%%%%%%%%%%%%%%%%%%%%%%%%%%%%%%%%%%%%%%%%%%%%%%%%%%%%%%%%%%%%%%%%%%%%%%%%%%%%%%%%%%%%%%%%%%%%%%%%%%%%%%%%%%%%%%%%%%%%%%%%%%%%%%%%%%%%%%%%%%%%%%%%%%%%%%%%%%%%%%%%% 
	\subsubsection{The energy inequality for the weak solution}\label{energyweak}
	This section is devoted to the proof of the fact that the quadruple $(v,M,\phi,\mu)$ obtained as limits of interpolants (we refer to \eqref{InterpWeakCon}) satisfies \eqref{EnIneq*}
	for all $t\in(0,T),$ where $E_{tot}$ is as defined in \eqref{defEtot}. To this end we take into account \eqref{InterpWeakCon}, \eqref{PhiInterpStrongComp}, \eqref{MInterpStrongComp}, \eqref{RhoInterpConv}, \eqref{strngconvL24} and \eqref{MNStrongly} and select subsequences that will not be relabeled such that for a.e.\  $t\in(0,T)$
	\begin{equation}\label{NConvergencesAETime}
		\begin{alignedat}{2}
			\rho^N(t)&\to\rho(t)&&\text{ in }L^2(\Omega),\\
			v^N(t)&\to v(t)&&\text{ in }L^4(\Omega),\\
			\nabla M^N(t)&\to \nabla M(t)&&\text{ in }L^2(\Omega),\\
			M^N(t)&\to M(t)&&\text{ in }L^4(\Omega),\\
			\nabla\phi^N(t)&\rightharpoonup\nabla \phi(t)&&\text{ in }L^2(\Omega),\\
			\phi^N(t)&\to\phi(t)&&\text{ in }L^2(\Omega)\text{ and a.e.\  in }\Omega.
		\end{alignedat}
	\end{equation}
	We want to show that for a.e.\  $t\in(0,T)$
	\begin{equation}\label{EtotWLSC}
		E_{tot}(v(t),M(t),\phi(t))\leqslant \liminf_{N\to\infty}E_{tot}(v^N(t),M^N(t),\phi^N(t)).
	\end{equation}
	We argue as in \cite[Section 5.2, p.~28--29]{KMS20} and focus only on the terms from $E_{tot}(v^N(t),M^N(t),\phi^N(t))$ that are not treated in \cite{KMS20}. We fix $t\in(0,T)$, in which convergences from \eqref{NConvergencesAETime} are available. By \eqref{NConvergencesAETime}$_{1,2}$ we get
	\begin{equation*}
		\lim_{N\to\infty}\frac{1}{2}\int_\Omega \rho^N(t)|v^N(t)|^2=\frac{1}{2}\int_\Omega \rho(t)|v(t)|^2.
	\end{equation*}
	Since $\widetilde{\Psi}_{0}\in C([-1,1]),$ we obtain by using \eqref{NConvergencesAETime}$_{6}$ and dominated convergence theorem
	\begin{equation*}
		\lim_{N\to\infty}\int_\Omega\widetilde \Psi(\phi^N(t))=\lim_{N\to\infty}\int_\Omega\left(\widetilde \Psi_0(\phi^N(t))-\frac{\kappa}{2}(\phi^N(t))^2\right)=\int_\Omega\left(\widetilde \Psi_0(\phi(t))-\frac{\kappa}{2}(\phi(t))^2\right)=\int_\Omega\widetilde \Psi(\phi(t)).
	\end{equation*}
	The remaining details for the proof of \eqref{EtotWLSC} can be found in \cite[Section 5.2, p.~28--29]{KMS20}. Applying the convergences from \eqref{InitMPhiApp} we conclude $E_{tot}(v_0,M^N_0,\phi^N_0)\to E_{tot}(v_0,M_0,\phi_0)$ in a straightforward way. Hence to conclude \eqref{EnIneq*} it suffices to combine \eqref{EtotWLSC}, the fact that
	$$\displaystyle\sqrt{2\nu(\phi^{N}_{h})}\mathbb{D}(v^{N})\rightharpoonup \sqrt{2\nu(\phi)}\mathbb{D}(v)\,\,\mbox{in}\quad L^{2}(Q_{T}),$$
	which can be proved as in \cite[eq. (5.57)]{KMS20},
	the weak lower semicontinuity of norms with \eqref{InterpWeakCon}$_{6}$ and \eqref{SeqDvrConv}.
	\subsubsection{Continuity with respect to time of $v,M,\phi$}
	This section aims to show that some of the limit functions obtained in previous sections are continuous w.r.t. time variable in a certain sense. Namely, we show
	\begin{equation}\label{TContWeak}
		\begin{split}
			\rho v\in &C_w([0,T];L^2(\Omega)),\\
			v\in &C_{w}([0,T];L^2(\Omega)),\\
			M\in &C_w([0,T];W^{1,2}(\Omega)),\\
			M\in &C([0,T];L^2(\Omega)),\\
			\phi\in &C_w([0,T];W^{1,2}(\Omega)),\\
			\phi\in &C([0,T];L^2(\Omega)).
		\end{split}
	\end{equation}
	First, for the proof of \eqref{TContWeak}$_{3,4,5,6}$ we refer to \cite[Section 5.3, p.~29]{KMS20}. Let us next prove \eqref{TContWeak}$_{1}.$ As $\rho\in L^\infty(Q_T)$ and $v\in L^\infty(0,T;L^2(\Omega))$, we have 
	\begin{equation}\label{rvLinf}
	\begin{array}{l}
	\displaystyle\rho v\in L^\infty(0,T;L^2(\Omega)).
	\end{array}
	\end{equation}
	Next in view of \eqref{DiffQRVNBound} one has up to a nonrelabeled subsequence
	\begin{equation*}
	\begin{array}{l}
	\displaystyle\tder\mathbb{P}_{\dvr}(\widetilde{\rho v}^N)\rightharpoonup \tder\mathbb{P}_{\dvr}(\rho v)\,\,\mbox{in}\,\, L^\frac{8}{7}(0,T;(V(\Omega))')
	\end{array}
	\end{equation*}
	(where the identification of the limit follows from \eqref{WTRhoVWeakly}). Then the fact that $\tder\mathbb{P}_{\dvr}(\rho v)\in L^\frac{8}{7}(0,T;(V(\Omega))')$ implies $\mathbb{P}_{\dvr}(\rho v)\in C([0,T];(V(\Omega))')$. This along with \eqref{rvLinf} renders 
	\begin{equation}\label{Cwrhov}
	\begin{array}{l}
	\mathbb{P}_{\dvr}(\rho v)\in C_{w}([0,T];L^{2}_{\dvr}(\Omega))
	\end{array}
	\end{equation}
	by using \cite[Ch. III, Lemma 1.4]{Tem77}.\\
	Next using the definition \eqref{Leray1}--\eqref{Leray} of the Leray projector $\mathbb{P}_{\dvr}$ we write
	\begin{equation}\label{rhov}
	\begin{array}{l}
	\displaystyle\rho v=\mathbb{P}_{\dvr}(\rho v)+\nabla p,
	\end{array}
	\end{equation}
	where $p(t)\in W^{1,2}(\Omega),$ $\displaystyle\int_{\Omega} p(t)=0$ and $p(t)$ solves the weak Neumann problem \eqref{Leray}. Now one can follow the arguments used in \cite[Section 5.2, p.~475--476]{AbDeGa113} to show that $\nabla p\in C_{w}([0,T];L^{2}(\Omega)).$ This along with \eqref{Cwrhov} furnishes the proof of \eqref{TContWeak}$_{1}.$\\
	Finally, we wish to show \eqref{TContWeak}$_{2}.$ By definition one needs to prove $v(\cdot,t_{n})\rightharpoonup v(\cdot,t)$ in $L^{2}(\Omega)$ for any sequence $\{t_n\}\subset[0,T]$ such that $t_{n}\rightarrow t.$ In view of the non-degeneracy of $\rho,$ one first infers from \eqref{rhov} $$v(\cdot,t)=\frac{1}{\rho(\cdot,t)}\mathbb{P}_{\dvr}(\rho v)(\cdot,t)+\frac{1}{\rho(\cdot,t)}\nabla p(\cdot,t),$$
	(with this definition one also defines $v$ in a set of measure zero, so that $v$ is defined everywhere in $[0,T]$)
	 uses \eqref{TContWeak}$_{1},$ $\nabla p\in C_{w}([0,T];L^{2}(\Omega))$ and the fact that $\rho\in C([0,T];L^{2}(\Omega))$ (which follows from \eqref{TContWeak}$_{6}$) to show that $v(\cdot,t_{n})\rightharpoonup v(\cdot,t)$ in $L^{1}(\Omega).$ Finally, since $v(\cdot,t_{n})$ is uniformly bounded in $L^{2}(\Omega),$ one concludes that $v(\cdot,t_{n})\rightharpoonup v(\cdot,t)$ in $L^{2}(\Omega)$ and thereby finishing the proof of \eqref{TContWeak}$_{2}.$
	\subsection{Recovering the weak formulations}\label{recoveringweak}
	In this section we verify that the quadruple $(v,M,\phi,\mu)$ satisfies the formulation of the problem in the sense of Definition~\ref{DefWS} by performing the limit passage $N\to\infty$ in \eqref{MomEqInterp}--\eqref{EqMuInterp}. We start with the momentum equation. We consider a fixed $\psi_1\in C^1_c([0,T);V(\Omega))$ in \eqref{MomEqInterp}. Since $\widetilde{\rho v}^{N}$ is bounded in $L^{\infty}(0,T;L^{2}(\Omega))$, which follows from \eqref{CInterpolantsBound}$_{1}$ and \eqref{RhoNBound}, we have 
	\begin{equation}\label{TRVTimeConv}
		\mathbb{P}_{\dvr}(\widetilde {\rho v}^N(t))\rightharpoonup\mathbb{P}_{\dvr}(\rho v(t))\text{ in }L^2(\Omega)\text{ for a.e.\  }t\in(0,T),
	\end{equation}
	where the weak limit in \eqref{TRVTimeConv} is identified by using \eqref{WTRhoVWeakly}. Fixing $\tau\in(0,T)$ such that \eqref{TRVTimeConv} holds we take into consideration that $\partial^-_{t,h}\rho^N v^N=\partial_t\widetilde{\rho v}^N$ by \eqref{AffInterpProp}$_1$ and integrate by parts with respect to time in \eqref{MomEqInterp}  to obtain
	\begin{equation*}
		\begin{split}
			&\int_\Omega\widetilde{\rho v}^N(\tau)\psi_1(\tau)-\int_\Omega\widetilde{\rho v}^N(0)\psi_1(0)+\int_0^\tau\left(\int_{\Omega}- \widetilde{\rho v}^N\cdot\tder\psi_{1}-\int_{\Omega}(\rho^N_hv^N\otimes v^N)\cdot\nabla{\psi}_{1}-\int_{\Omega}v^N\otimes J^N\cdot \nabla\psi_{1}\right.\\
			&\left.+\int_{\Omega}\left(\dvr(\xi(\phi^N_h)\nabla M^N)-\frac{{\xi(\phi^N_h)}}{\alpha^{2}}(|M^N|^{2}M^N-M^N_h)\right)\nabla M^N\cdot\psi_{1}\right)\\
			&=\int_0^\tau\left(-2\int_{\Omega}\nu(\phi^{N}_{h})\mathbb{D} v^N\cdot\mathbb{D}\psi_{1}-\int_{\Omega}\nabla\mu^N\phi^N_h\cdot\psi_{1}\right).
		\end{split}
	\end{equation*}
	Thanks to \eqref{TRVTimeConv} and the definition \eqref{Leray1}--\eqref{Leray} of the Leray projector, we pass to the limit in the first term on the left hand side of the latter identity. By the definition of $\widetilde{\rho v}^N(0)$ we have, employing also \eqref{InitMPhiApp}$_2$,
	\begin{equation*}
	\begin{split}
		\widetilde{\rho v}^N(0)=\rho^N(-h)v^N(-h)&\displaystyle=\tfrac{1}{2}\left(\widetilde \rho_1+\widetilde\rho_2+(\widetilde\rho_2-\widetilde\rho_1)\phi^N_0\right)v_0\\
		&\to \tfrac{1}{2}\left(\widetilde \rho_1+\widetilde\rho_2+(\widetilde\rho_2-\widetilde\rho_1)\phi_0\right)v_0=\rho_0v_0\text{ in }L^1(\Omega),
		\end{split}	
	\end{equation*}
	which allows us to perform the passage in the second term. To pass to the limit in the third term we use \eqref{WTRhoVWeakly} and the definition \eqref{Leray1}--\eqref{Leray} of the Leray projector. We perform the limit passage in the fourth term with the help of \eqref{RhoInterpConv} and \eqref{VNStrongly}. We recall that $\displaystyle J^N=-\frac{\widetilde\rho_2-\widetilde\rho_1}{2}\nabla\mu^N$. Hence combining the convergences \eqref{InterpWeakCon}$_6$ and \eqref{VNStrongly} ensures the limit passage in the fifth term. For the limit passage in the last term on the left hand side we use \eqref{SeqDvrConv} and \eqref{MNStrongly}. The limit passage on the right hand side is ensured by 
	$$\displaystyle {\nu(\phi^{N}_{h})}\mathbb{D}(v^{N})\rightharpoonup {\nu(\phi)}\mathbb{D}(v)\,\,\mbox{in}\quad L^{2}(Q_{T}),$$
	whose proof can be found in \cite[eq. (5.56)]{KMS20}
	and \eqref{InterpWeakCon}$_6$ combined with \eqref{PhiInterpStrongComp}. We arrive at
	\begin{equation}
		\begin{split}
			&\int_\Omega\rho v(\tau)\psi_1(\tau)-\int_\Omega\rho v(0)\psi_1(0)+\int_0^\tau\left(\int_{\Omega}-\rho v\cdot\tder\psi_{1}-\int_{\Omega}(\rho v\otimes v)\cdot\nabla{\psi}_{1}-\int_{\Omega}v\otimes J\cdot \nabla\psi_{1}\right.\\
			&\left.+\int_{\Omega}\left(\dvr(\xi(\phi)\nabla M)-\frac{{\xi(\phi)}}{\alpha^{2}}(|M|^{2}M-M)\right)\nabla M\cdot\psi_{1}\right)\\
			&=\int_0^\tau\left(-2\int_{\Omega}\nu(\phi)\mathbb{D} v\cdot\mathbb{D}\psi_{1}-\int_{\Omega}\nabla\mu\phi\cdot\psi_{1}\right).
		\end{split}
	\end{equation}
	Next we consider $t\in(0,T)$ and a sequence $\{\tau^k\},$ s.t. $\tau^k\to t$ and the latter identity holds for $\tau=\tau^k$. Employing \eqref{TContWeak}$_1$ and the fact that all terms under the integration sign over the time interval are integrable with respect to time we conclude \eqref{WeakForm}$_1$ by the limit passage $k\to \infty$. 
	
	We note that the validity of identities \eqref{WeakForm}$_{2,3}$ (by the limit passage in \eqref{intidentity2} and \eqref{intidentity3}) can be proved by following line by line the arguments used to show \cite[$(2.4)_{2}$ and $(2.4)_{3}$]{KMS20} in \cite[Section 5.4, p.~31]{KMS20}. In order to verify that \eqref{WeakForm}$_4$ is fulfilled, we pass to the limit $N\to\infty$ in \eqref{EqMuInterp}. In view of the convergences \eqref{InterpWeakCon}$_{6}$, \eqref{PhiInterpStrongComp}, \eqref{MNStrongly} and \eqref{MInterpStrongComp} we conclude 
	\begin{align*}
	   &\mu^N+\frac{\kappa}{2}(\phi^N+\phi^N_h)-H_0(\phi^N,\phi^N_h)\frac{|\nabla M^N|^{2}}{2}-\frac{H_0(\phi^N,\phi^N_h)}{4\alpha^{2}}(|M^N|^{2}-1)^{2}\\&\rightharpoonup \mu+\kappa\phi-\xi'(\phi)\frac{|\nabla M|^{2}}{2}-\frac{\xi'(\phi)}{4\alpha^{2}}(|M|^{2}-1)^{2} \text{ in }L^1(Q_T).
	\end{align*}		 
	Indeed, the passage to the limit in the first two terms is straightforward and the $L^1$ weak convergence of the remaining two terms is explained in detail in \cite[Section 5.4, p.~32]{KMS20}. For the limit passage in the terms on the right hand side of \eqref{EqMuInterp} we use the convergence
	\begin{equation*}  
	-\eta\Delta \phi^N+\widetilde \Psi'_0(\phi^N)\rightharpoonup -\eta\Delta \phi+\widetilde \Psi'_0(\phi)\text{ in }L^1(Q_T),
	\end{equation*}
	which follows by \eqref{PhiNW2Cnv} and \eqref{PsiPhiNCnv}. Thus we arrive at 
	  	\begin{equation*}
			\int_0^T\int_\Omega\left(\mu+\kappa\phi-\xi'(\phi)\frac{|\nabla M|^{2}}{2}-\frac{\xi'(\phi)}{4\alpha^{2}}(|M|^{2}-1)^{2}\right)\psi_4=\int_0^T\int_\Omega\left(-\eta\Delta\phi+\widetilde{\Psi}'_{0}(\phi)\right)\psi_4
	\end{equation*}
	for all $\psi_4\in L^\infty(0,T;L^\infty(\Omega))$. Hence it follows that identity \eqref{WeakForm}$_4$ is fulfilled.
		
	\subsection{The attainment of initial data \texorpdfstring{$v_0, M_0, \phi_0$}{InitC}}\label{attainmentinitial}
	In this section, we prove \eqref{InitDataAtt} with the help of \eqref{WeakForm}, which we proved in the previous section. First we show the following identities
	\begin{equation}\label{InitValIdent}
		\begin{alignedat}{2}
			v(0)=&v_0&&\text{ a.e.\ in }\Omega,\\
			M(0)=&M_0&&\text{ a.e.\ in }\Omega,\\
			\phi(0)=&\phi_0&&\text{ a.e.\ in }\Omega.
		\end{alignedat}
	\end{equation}
	Setting $\psi_3(t,x)=\theta(t)\vartheta(x)$ in \eqref{WeakForm}$_3$, where $\theta\in C^1_c([0,T))$ with $\theta(0)>0$ and $\vartheta\in C^{\infty}_{c}(\Omega)$ are arbitrary but fixed, we obtain using \eqref{TContWeak}$_{5}$
	\begin{equation*}
		\int_\Omega \phi_{0}\theta(0)\vartheta=\lim_{t\to 0_+}\int_\Omega \phi(t)\theta(t)\vartheta=\int_\Omega \phi(0)\theta(0)\vartheta,
	\end{equation*}
	which implies \eqref{InitValIdent}$_3.$	Setting $\psi_1(t,x)=\theta(t)\omega(x)$ in \eqref{WeakForm}$_1$, where $\theta\in C^1_c([0,T))$ with $\theta(0)>0$ and $\omega\in V(\Omega)$ are arbitrary but fixed, yields
	\begin{equation*}
		\int_\Omega \rho_0v_0\cdot\theta(0)\omega=\lim_{t\to 0_+}\int_\Omega \rho(t)v(t)\cdot\theta(t)\omega=\int_\Omega \rho(0)v(0)\cdot\theta(0)\omega,
	\end{equation*}
	where the second equality follows by \eqref{TContWeak}$_{1}$ and \eqref{InitValIdent}$_3$ implies $\rho(0)=\rho_0$. Setting in the latter identity $\omega=v_0-v(0)$, which is allowed due to the density of $V(\Omega)$ in $L^2_{\dvr}(\Omega)$, implies \eqref{InitValIdent}$_1$. We note that the fact that $\rho_0$ has a positive lower bound was also used. Finally, we repeat the above arguments to justify \eqref{InitValIdent}$_2$.
	
	With the help of \eqref{TContWeak} we will show that the energy inequality \eqref{EnIneq*} holds for all $t\in[0,T]$. We start by considering an arbitrary $t\in[0,T]$ and a sequence $\{t^k\}$ such that $t^k\geq t$, $t^k\to t$ as $k\rightarrow\infty$ and 
	\begin{equation}\label{tkconv}
		\begin{alignedat}{2}
			\rho(t^k)v(t^k)&\rightharpoonup \rho(t)v(t)&&\mbox{ in }L^2(\Omega),\\
			v(t^k)&\rightharpoonup v(t)&&\mbox{ in }L^2(\Omega),\\
			M(t^{k})&\rightharpoonup M(t)&&\mbox{ in } W^{1,2}(\Omega),\\
			M(t^{k})&\rightarrow M(t)&&\mbox{ in }L^{2}(\Omega),\\
			\phi(t^{k})&\rightharpoonup \phi(t)&&\mbox{ in }W^{1,2}(\Omega),\\
			\phi(t^{k})&\rightarrow \phi(t)&&\mbox{ in }L^{2}(\Omega)\text{ and a.e.\  in }\Omega,\\
			\rho(t^k)&\rightarrow \rho(t)&&\mbox{ in }L^{2}(\Omega)\text{ and a.e.\  in }\Omega
		\end{alignedat}
	\end{equation}
	and \eqref{EnIneq*} holds for each $t^k$. The convergence \eqref{tkconv}$_{7}$ follows from \eqref{tkconv}$_{6}$ by using the definition \eqref{DefRho} of $\rho.$ The existence of such a sequence $\{t^k\}$ is ensured by \eqref{TContWeak}. Because of the convexity of $|\cdot|^2$ (i.e.\  the inequality $|A|^{2}-|B|^{2}\geqslant 2B\cdot(A-B),$ for all $A,B\in\mathbb{R}^{m},$ $m\geqslant 1$) and convergences \eqref{tkconv}$_{1,2,7}$ it follows that
	\begin{equation}\label{1stPartEnLSC}
	\begin{array}{ll}
		\displaystyle\liminf_{k\to\infty}\frac{1}{2}\int_\Omega\rho(t^k)|v(t^k)|^2 &\displaystyle\geq  \liminf_{k\to\infty}\int_\Omega\left(\frac{1}{2}\rho(t^k)|v(t)|^2+\rho(t^k)v(t)\cdot(v(t^k)-v(t))\right)\\
		&\displaystyle=\frac{1}{2}\int_\Omega\rho(t)|v(t)|^2.
		\end{array}
	\end{equation}
	For the passage to the limit $k\to\infty$ in both terms we have used that $\rho(t^k)v(t)\to \rho(t)v(t)$ in $L^2(\Omega)$ (which follows by using Lebesgue's dominated convergence theorem and the fact that $\rho$ is bounded) and also \eqref{tkconv}$_{2}$ in the second term.
	%In the last calculation we have used that to show that $\rho(t^k)v(t)\to \rho(t)v(t)$ in $L^2(\Omega)$ (which follows by using Lebesgue dominated convergence theorem and the fact that $\rho$ is bounded) and \eqref{tkconv}$_{2}$
	%$$\liminf_{k\to\infty}\int_\Omega\rho(t^k)v(t)\cdot(v(t^k)-v(t))=0.$$ 
%	The passage of limit in the other term $\displaystyle\int_{\Omega}\frac{1}{2}\rho(t^{k})|v(t)|^{2}$ follows by using once again the strong convergence of $\rho(t^k)v(t)$ to $ \rho(t)v(t)$ in $L^2(\Omega).$\\
	Due to the weak lower semicontinutiy of convex functionals, the fact that $\widetilde{\Psi}\in C([-1,1])$ and \eqref{tkconv}$_{5,6}$, we obtain
	\begin{equation}\label{2ndPartEnLSC}
		\begin{split}
			&\liminf_{k\to\infty}\int_\Omega\left(\frac{\eta}{2}|\nabla\phi(t^k)|^2+\widetilde{\Psi}(\phi(t^k)) \right)\geq \int_\Omega \left(\frac{\eta}{2}|\nabla\phi(t)|^2+\widetilde{\Psi}(\phi(t)\right).
		\end{split}
	\end{equation}
	Moreover, we obtain
	\begin{equation}\label{3rdPartEnLSC}
		\begin{split}
			&\liminf_{k\to\infty} \int_{\Omega}\left(\xi(\phi(t^k))|\nabla M(t^k)|^2+\frac{\xi(\phi(t^k))}{\alpha^2}(|M(t^k)|^2-1)^2\right)\\
			&\geq \int_{\Omega}\left(\xi(\phi(t))|\nabla M(t)|^2+\frac{\xi(\phi(t))}{\alpha^2}(|M(t)|^2-1)^2\right),
		\end{split}
	\end{equation}
	by arguing as in \cite[(5.77)]{KMS20}. 
	Altogether, \eqref{1stPartEnLSC}, \eqref{2ndPartEnLSC}, \eqref{3rdPartEnLSC} and the absolute continuity of the map
	$$t\mapsto \int_0^t\left(\|\sqrt{2\nu}\mathbb{D} v\|^2_{L^2(\Omega)}+\|\nabla \mu\|^2_{L^2(\Omega)}+\left\|\dvr(\xi(\phi)\nabla M)-\frac{\xi(\phi)}{\alpha^2}M(|M|^2-1)\right\|^2_{L^2(\Omega)}\right)$$
	 imply that \eqref{EnIneq*} holds for all $t\in[0,T]$. Hence in particular it follows that
	\begin{equation}\label{LSIneq}
		\limsup_{t\to 0_+}E_{tot}(v(t),M(t),\phi(t))\leqslant E_{tot}(v_0,M_0,\phi_0).
	\end{equation}
	Employing again \eqref{TContWeak} along with \eqref{InitValIdent} (similarly as we have obtained \eqref{1stPartEnLSC}--\eqref{3rdPartEnLSC}) we deduce
	\begin{equation*}
		\liminf_{t\to 0_+}E_{tot}(v(t),M(t),\phi(t))\geq E_{tot}(v(0),M(0),\phi(0))=E_{tot}(v_0,M_0,\phi_0),
	\end{equation*}
	which along with \eqref{LSIneq} infers
	\begin{equation}\label{LimEn0}
		\lim_{t\to 0_+}E_{tot}(v(t),M(t),\phi(t))= E_{tot}(v_0,M_0,\phi_0).
	\end{equation}
	Taking into account the definition of $E_{tot},$ employing the inequalities $|A|^2-|B|^2\geq 2B\cdot(A-B)+2|A-B|^2$ (which follows from the strong convexity of $|\cdot|^2$) and $|A|^4-|B|^4\geq 4|B|^{2}B\cdot(A-B)$ (which follows from the convexity of $|\cdot|^4$) for all $A,B\in\mathbb{R}^m,$
	one obtains the following for each $t\in(0,T)$
	\begin{equation}\label{ineqdetot}
		\begin{split}
			E&_{tot}(v(t),M(t),\phi(t))-E_{tot}(v_0,M_0,\phi_0)\\
			\geq& \frac{1}{2}\int_\Omega(\rho(t)-\rho_0)|v_0|^2+ \int_\Omega \rho(t)v_0\cdot(v(t)-v_0)+\int_\Omega\rho(t)|v(t)-v_0|^2+\int_\Omega\left(\xi(\phi(t))-\xi(\phi_0)\right)|\nabla M_0|^2\\
			&+\int_\Omega 2\xi(\phi(t))\nabla M_0\cdot(\nabla M(t)-\nabla M_0)+2\int_\Omega \xi(\phi(t))|\nabla M(t)-\nabla M_0|^2\\
			&+\frac{1}{4\alpha^2}\int_\Omega \left(\xi(\phi(t))-\xi(\phi_0)\right)|M_0|^4 +\frac{1}{\alpha^2}\int_\Omega \xi(\phi(t))|M_0|^2M_0\cdot\left(M(t)-M_0\right)\\
			&-\frac{1}{2\alpha^2}\int_\Omega  \left(\xi(\phi(t))|M(t)|^2-\xi(\phi_0)|M_0|^2\right)+\frac{1}{4\alpha^2}\int_\Omega  \left(\xi(\phi(t))-\xi(\phi_0)\right)\\
			&+\eta\int_\Omega\nabla\phi_0\cdot\left(\nabla\phi(t)-\nabla\phi_0\right)+\eta\int_\Omega|\nabla\phi(t)-\nabla\phi_0|^2+\int_{\Omega}\left(\widetilde{\Psi}_{0}(\phi)-\widetilde{\Psi}_{0}(\phi_{0})\right)\\
			&-\frac{\kappa}{2}\int_\Omega \left(\phi^2(t)-\phi_0^2\right)
			=\sum_{m=1}^{14}I_m(t).
		\end{split}
	\end{equation}
	Now we show \eqref{InitDataAtt} by taking the limsup $t\to 0_+$ on both sides of the inequality \eqref{ineqdetot}. We consider an arbitrary sequence $\{t^k\}$ such that $t^k\to 0_+$ as $k\to\infty$. The sequence $\{t^k\}$ has a subsequence $\{t^{k'}\}$ such that the following holds
	\begin{equation}\label{PhiTkcnv}
		\phi(t^{k'})\to\phi_0,\ \text{ a.e.\ in }\Omega\text{ as }k'\to\infty
	\end{equation}
	by \eqref{TContWeak}$_6$ and \eqref{InitValIdent}$_3$. Accordingly, we have
	\begin{equation}\label{RhoTkcnv}
		\rho(t^{k'})\to\rho_0,\ \text{ a.e.\ in }\Omega\text{ as }k'\to\infty
	\end{equation}
	by \eqref{DefRho}. For the proof of
	\begin{equation}\label{LimImI}
		\lim_{k^{'}\to\infty}I_m(t^{k'})=0\text{ for }m=4,5,7,8,9,10,11
	\end{equation}
	we refer to \cite[(5.82)-(5.85)]{KMS20}. Next we deal with $I_1$, $I_2$, $I_{13}$ and $I_{14}$. Convergence \eqref{RhoTkcnv} and the fact that $\rho$ is a bounded function imply
	\begin{equation}\label{LimI1}
		\lim_{k'\to\infty}I_1(t^{k'})=0
	\end{equation}
	by the Lebesgue dominated convergence theorem. Moreover, we have that $\rho(t^{k'})v_0\to \rho_0v_0$ in $L^2(\Omega)$, which along with \eqref{TContWeak}$_{2}$ and \eqref{InitValIdent}$_1$ yields
	\begin{equation}\label{LimI2}
		\lim_{k'\to\infty}I_2(t^{k'})=0.
	\end{equation}
	Since $\widetilde{\Psi}_{0}\in C([-1,1]),$ the following
	\begin{equation}\label{LimI13}
		\lim_{k'\to\infty}I_{13}(t^{k'})=0
	\end{equation}
	is obtained as an immediate consequence of \eqref{PhiTkcnv} and the Lebesgue dominated convergence theorem.\\
	Finally, by \eqref{TContWeak}$_{6}$ and \eqref{InitValIdent}$_3$ we obtain
	\begin{equation}\label{LimI14}
		\lim_{k'\to\infty}I_{14}(t^{k'})=0.
	\end{equation}	  
	Hence 
	\begin{equation}\label{limsupinq}
		\limsup_{k '\to\infty} \left(\underline{\rho}\|v(t^{k'})-v_0\|^2_{L^2(\Omega)}+c_1\|\nabla M(t^{k'})-\nabla M_0\|^2_{L^2(\Omega)}+\eta\|\nabla\phi(t^{k'})-\nabla\phi_0\|^2_{L^2(\Omega)}\right)\leqslant 0
	\end{equation}
	follows from \eqref{ineqdetot} by \eqref{LimImI}--\eqref{LimI14} provided that we apply \eqref{XiAssum}$_2$ and take into consideration that there is a positive lower bound on $\rho$, which we denote by $\underline{\rho}.$\\
	The inequality \eqref{limsupinq} along with \eqref{TContWeak}$_{4,6}$ infer 
	\begin{equation}\label{LimTkP}
		\lim_{k '\to\infty}\left(\|v(t^{k'})-v_0\|_{L^2(\Omega)}+\|M(t^{k'})- M_0\|_{W^{1,2}(\Omega)}+\|\phi(t^{k'})-\phi_0\|_{W^{1,2}(\Omega)}\right)=0.
	\end{equation} 
	Since $\{t^{k}\}$ is an arbitrary sequence possessing a subsequence satisfying \eqref{LimTkP}, one concludes the proof of \eqref{InitDataAtt}. 
	%If this is not the case, one could find a sequence $\{t^{\bar{k}}\}$ with $t^{\bar{k}}\to 0_+$ from which no subsequence satisfying \eqref{LimTkP} can be selected. This contradicts our finding that any sequence $\{t^k\}$ with $t^k\to 0$ possesses a subsequence $\{t^{k'}\}$ for which \eqref{LimTkP} holds.
	\subsection{Attainment of the boundary condition and some regularity results for $M$ in Lebesgue spaces}\label{bndryregM} In this section we discuss the proofs of the items $(ii)$ and $(iii)$ of Theorem~\ref{Thm:Main}. For the proof of the item $(ii)$ we refer the readers to \cite[Section 6.1]{KMS20}. The item $(iii)$ was formally commented in \cite[Section 6.2]{KMS20} but one needs to suitably regularize the magnetization equation to make the arguments concrete. Here we provide the details for the proof of item $(iii).$\\
	In the direction of proving item $(iii)$ of Theorem~\ref{Thm:Main}, we first show that for given $v$ and $\phi$ in the functional settings \eqref{functionalspaces}$_{1,3}-$\eqref{AddReg}$_{2}$ there is a unique $M$ satisfying \eqref{functionalspaces}$_{2}-$\eqref{AddReg}$_{1}$ and solving the weak formulation \eqref{WeakForm}$_{2}$ of the magnetization equation.  
	 Since $\partial_{t}M\in L^{2}(0,T;L^{\frac{3}{2}}(\Omega)),$ equation \eqref{WeakForm}$_{2}$ can be rewritten as:
	 \begin{equation}\label{weakformM}
	 	\begin{array}{ll}
	 		\displaystyle\int_0^t\int_\Omega\biggl(\partial_{t}M+(v\cdot\nabla) M\biggl)\cdot\psi_{2}=-\int_0^t\int_\Omega \xi(\phi)\nabla M\cdot\nabla\psi_{2}-\int_0^t\int_\Omega\frac{1}{\alpha^2}\bigl(\xi(\phi)(|M|^2-1)M\bigr)\cdot\psi_{2}
	 	\end{array}
	 \end{equation}
	 for $t\in(0,T)$ and $\psi_{2}\in C^{1}_{c}(0,T;W^{1,2}(\Omega))$ or equivalently
	 \begin{equation}\label{equiweakM}
	 \begin{array}{ll}
	 	\displaystyle \int_\Omega \partial_{t}M\cdot\psi_{2}+\int_\Omega(v\cdot\nabla) M\cdot\psi_{2}=-\int_\Omega \xi(\phi)\nabla M\cdot\nabla\psi_{2}-\int_\Omega\frac{1}{\alpha^2}\bigl(\xi(\phi)(|M|^2-1)M\bigr)\cdot\psi_{2}
	 \end{array}
	 \end{equation}
	 for a.e.\ $t\in(0,T)$ and $\psi_{2}\in W^{1,2}(\Omega).$\\
	 Let $M_{1}$ and $M_{2}$ belong to \eqref{functionalspaces}$_{2}-$\eqref{AddReg}$_{1}$ and solve \eqref{weakformM} with $v$ and $\phi$ in the framework \eqref{functionalspaces}$_{1,3}-$ \eqref{AddReg}$_{2}.$ One can now take the difference of the equations solved by $M_{1}$ and $M_{2}$ and consider $(M_{1}-M_{2})$ as a test function, which is possible since $C^{1}_{c}(0,T;W^{1,2}(\Omega))$ is dense in $L^{2}(0,T;W^{1,2}(\Omega)).$ Consequently using the incompressibility of $v$ and the inequality $\bigl(|M_1|^2M_1-|M_2|^2M_2\bigr)\cdot(M_1-M_2)\geqslant 0$ (since the map $\alpha\mapsto |\alpha|^{2}\alpha$ is monotone) one furnishes 
	 $$\frac{1}{2}\|(M_{1}-M_{2})(t)\|^{2}_{L^{2}(\Omega)}\leqslant C\int_{0}^{t}\|M_{1}-M_{2}\|^{2}_{L^{2}(\Omega)},$$
	 for a.e.\ $t\in(0,T).$ Hence by the Gr\"{o}nwall inequality one at once renders that $M_{1}=M_{2}$ a.e.\ in $Q_{T}.$\\
	 Now we plan to use test functions of the form $|M|^{r-2}M$ with $r>2$ in \eqref{equiweakM}. But due to the lack of regularity (particularly one needs for a.e.\ $t\in (0,T),$ $M\in L^{r-1}(\Omega)$ for arbitrary $r>2$) this does not qualify as a test function. Instead we consider a regularized magnetization equation, i.e.\ we first take a sequence $\{\phi^{m}\}_{m}$ in $L^{2}(0,T;C^{\infty}(\overline{\Omega}))$ such that 
	 $$\phi^{m}\rightarrow \phi\,\,\mbox{in}\,\, L^{2}(Q_{T})$$
	 (such a sequence can easily be constructed by a suitable argument involving cut-off and convolution by mollifiers). Now let $M^{m}$ be the weak solution to \eqref{weakformM} or \eqref{equiweakM} corresponding to $\phi^{m}$ with boundary condition $\partial_{n}M^{m}\mid_{\Sigma_{T}}=0$ and initial condition $M^{m}(\cdot,0)=M_{0}\in W^{1,2}(\Omega).$ Our idea is to consider $|M^{m}|^{r-2}M^{m}$ as a test function in the equation solved by $(\phi^{m},M^{m})$ thereby proving an uniform estimate of $M^{m}$ in $L^{r}(\Omega)$ and next pass $m\rightarrow\infty$ to construct a weak solution $M$ corresponding to $\phi$ for \eqref{weakformM} or equivalently \eqref{equiweakM} which also solves the desired $L^{r}(\Omega)$ estimate. Of course, because of the uniqueness of the solution of the magnetization equation corresponding to the fixed pair $(v,\phi)$ and the initial data $M_{0},$ which we have already proved, this process will give the same $M$ solving \eqref{WeakForm}$_{2}.$\\
	 With the help of a time discretization scheme one can prove the existence of a weak solution $M^{m}\in L^{\infty}(0,T;W^{1,2}(\Omega))\cap W^{1,2}(0,T;L^{\frac{3}{2}}(\Omega))$ of \eqref{weakformM} or equivalently \eqref{equiweakM} corresponding to a vector field $v$ (satisfying \eqref{functionalspaces}$_{1}$) and $\phi^{m}.$  Moreover we notice that, in a strong form, this $M^{m}$ solves
	 	\begin{equation}\label{Mmstrng}
	 	\begin{alignedat}{2}
	 	\Delta M^{m}=&\frac{1}{\xi(\phi^{m})}\Bigl(\partial_{t}M^{m}+(v\cdot\nabla)M^{m}-\xi'(\phi^{m})\nabla M^{m}\cdot\nabla\phi^{m}+\frac{\xi(\phi^{m})}{\alpha^{2}}\bigl(|M^{m}|^{2}-1\bigr)M^{m}\Bigr)\quad&&\text{ in }\Omega,\\
	 	\partial_{n}M^{m}=&0&&\text{ on }\partial\Omega.\\
	 	\end{alignedat}
	 	\end{equation}
	 In view of the fact that $M^{m}\in L^{\infty}(0,T;W^{1,2}(\Omega))\cap W^{1,2}(0,T;L^{\frac{3}{2}}(\Omega))$ the right hand of \eqref{Mmstrng}$_{1}$ can be estimated in $L^{2}(0,T;L^{\frac{3}{2}}(\Omega))$ and hence by standard elliptic regularity $M^{m}\in L^{2}(0,T;W^{2,\frac{3}{2}}(\Omega))\hookrightarrow L^{2}(0,T;L^{r}(\Omega))$ for any $0<r<\infty.$ Hence for a.e.\  $t\in(0,T),$ $|M^{m}|^{r-1}M^{m}(t),$ $r>2$ can be used as a test function in \eqref{equiweakM}. Consequently
	  \begin{equation}\label{Lpmag}
	  \begin{array}{ll}
	  &\displaystyle\frac{1}{r}\partial_{t}\|M^{m}\|^{r}_{L^{r}(\Omega)}+\int_{\Omega}(v\cdot\nabla)M^{m}|M^{m}|^{r-2}M^{m}+\int_{\Omega} \xi(\phi^{m})(r-1)|M^{m}|^{r-2}|\nabla M^{m}|^{2}\\
	  &\displaystyle+\frac{1}{\alpha^{2}}\int_{\Omega}\xi(\phi^{m})|M^{m}|^{r+2}-\frac{1}{\alpha^{2}}\int_{\Omega}\xi(\phi^{m})|M^{m}|^{r}=0.
	  \end{array}
	  \end{equation} 
	  Once again integrating by parts the second term and using that $\mbox{div}\,v=0$ on $\Omega$ one concludes from \eqref{Lpmag} that:
	  \begin{equation}\label{LpbGron}
	  \begin{array}{l}
	  \displaystyle\partial_{t}\|M^{m}\|^{r}_{L^{r}(\Omega)}\leqslant \frac{c_{2}r}{\alpha^{2}}\|M^{m}\|^{r}_{L^{r}(\Omega)},
	  \end{array}
	  \end{equation}
	  where $c_{2}>0$ is the constant appearing in the assumption \eqref{XiAssum}. Now if one assumes $M_{0}\in W^{1,2}(\Omega)\cap L^{r}(\Omega),$ $r>6,$ using Gronwall's inequality one has the following from \eqref{LpbGron}:
	  \begin{equation}\label{afLpGron11}
	  \begin{array}{l}
	  \displaystyle\|M^{m}(t)\|_{L^{r}(\Omega)}\leqslant \|M_{0}\|_{L^{r}(\Omega)}e^{\frac{c_{2}}{\alpha^{2}}t}\quad\mbox{for all}\quad t\in[0,T].
	  \end{array}
	  \end{equation}
	  Additionally if $M_{0}\in L^{\infty}(\Omega),$ one can take the limit $r\rightarrow \infty$ in \eqref{afLpGron11} to conclude that:
	  \begin{equation}\label{afLpGronLin1}
	  \begin{array}{l}
	  \displaystyle\|M^{m}(t)\|_{L^{\infty}(\Omega)}\leqslant \|M_{0}\|_{L^{\infty}(\Omega)}e^{\frac{c_{2}}{\alpha^{2}}t}\quad\mbox{for all}\quad t\in[0,T].
	  \end{array}
	  \end{equation}
	  Now we let $m\rightarrow\infty$ in the equation solved by $(\phi^{m},M^{m}),$ i.e.\ \eqref{weakformM} with $(\phi,M)$ replaced by $(\phi^{m},M^{m}).$ The limit passage in the equation is obtained in a standard way (roughly it consists in showing the weak compactness of $M^{m}$ in $L^{2}(0,T;W^{1,2}(\Omega))$ and next using Aubin-Lions to achieve the strong compactness in $L^{2}(0,T;L^{4}(\Omega))$). Finally in view of the estimates \eqref{afLpGron11} and \eqref{afLpGronLin1}, which are independent of $m,$ one concludes \eqref{afLpGron} and \eqref{afLpGronLin}.
	  
	\subsection{Summary of the proof of Theorem~\ref{Thm:Main}}
	For the sake of the readers we summarize the proof of Theorem~\ref{Thm:Main} with exact references to the sections. 
	\begin{itemize}
	  	\item For the obtainment of the regularities \eqref{functionalspaces} with the exception of $M\in W^{1,2}(0,T;L^{\frac{3}{2}}(\Omega)),$ $\phi\in L^{2}(0,T;W^{2,1}(\Omega))$ and $\Psi'(\phi)\in L^{1}(Q_{T}),$ we refer the readers to \eqref{InterpWeakCon} and \eqref{TContWeak}. One can obtain the $W^{1,2}(0,T;L^{\frac{3}{2}}(\Omega))$ regularity of $M$ simply by estimating $\partial_{t}M\in L^{2}(0,T;L^{\frac{3}{2}}(\Omega))$ by using \eqref{diffviscoelastic*}$_{3}$ and the available regularities for $v$ and $M.$ More precisely $(v\cdot\nabla)M$ can be estimated in $L^{2}(0,T;L^{\frac{3}{2}}(\Omega))$ as in \eqref{MConvTermEst} and the boundedness of $\dvr(\xi(\phi)\nabla M)-\frac{\xi(\phi)}{\alpha^{2}}(|M|^{2}-1)M$ in $L^{2}(Q_{T})$ follows from \eqref{EnIneq*}. The additional $p-$regularities \eqref{AddReg} of $M,$ $\phi$ and $\Psi'(\phi)$ can be found in Section~\ref{Additionalconv} and they are of course stronger than $\phi\in L^{2}(0,T;W^{2,1}(\Omega))$ and $\Psi'(\phi)\in L^{1}(Q_{T})$ (stated as a part of \eqref{functionalspaces}).
	  	\item The weak formulation \eqref{WeakForm} solved by $(v,M,\phi,\mu)$ is proved in Section~\ref{recoveringweak}.
	  	\item The energy estimate \eqref{EnIneq*} is obtained in Section~\ref{energyweak}.
	  	\item The attainment of the initial data in the sense of \eqref{InitDataAtt} is obtained in Section~\ref{attainmentinitial}.
	  	\item The items $(ii)$ and $(iii)$ of Theorem~\ref{Thm:Main} corresponding to the attainment of boundary condition for $M$ and some regularity in Lebesgue spaces are proved in Section~\ref{bndryregM}.
	  \end{itemize}
	  In view of the above items we finally conclude the proof of Theorem~\ref{Thm:Main}.  \hfill $\Box$
	%%%%%%%%%%%%%%%%%%%%%%%%%%%%%%%%%%%%%%%%%%%%%%%%%%%%%%%%%%%%%%%%%%%%%%%%%%%%%%%%%%%%%%%%%%%%%%%%%%%%%%%%%%%%%%%%%%%%%%%%%%%%%%%%%%%%%%%%%%%%%%%%%%%%%%%%%%%%%%%%%%%%%%%%%%%%%%%%%%%%%%%%%%%%%%%%%%%%%%%%%%%%%%%%%%%%%%%%%%%%%%%%%%%%%%%%%%%%%%%%%%%%%%%%%%%%%%%%%%%%%%%%%%%%%%%%%%%%%%%%%%%%%%%%%%%%%%%%%%%%%%%%%%%%%%%%%%%%%%%%%%%%%%%%%%%%%%%%%%%%%%%%%%%%%%%%%%%%%%%%%%%%%%%%%%%%%%%%%%%%%%%%%%%%%%%%%%%%%%%%%%%%%%%%%%%%%%%%%%%%%%%%%%%%%%%%%%%%%%%%%%%%%%%%%%%%%%%%%%%%%%%%%%%%%%%%%%%%%%%%%%%%%%%%%%%%%%%%%%%%%%%%%%%%%%%%%%%%%%%%%%%%%%%%%%%%%%%%%%%%%%%%%%%%%%%%%%%%%%%%%%%%

\noindent{\textbf{Acknowledgment}.} This work is funded by the Deutsche Forschungsgemeinschaft (DFG, German Research Foundation), grant SCHL 1706/4-2, project number 391682204. S.M. is partially supported by the Alexander von Humboldt foundation. The work of M.K. received funding from the Czech Sciences Foundation (GA\v{C}R), GA19-04243S and in the framework of RVO: 67985840.

\bibliographystyle{plain}

\end{document}